\documentclass[12pt]{article}

\usepackage[utf8]{inputenc}
\usepackage{verbatim}

\usepackage{amsfonts}
\usepackage[x11names]{xcolor}
\usepackage{caption}
\usepackage{mathrsfs}
\usepackage[T1]{fontenc}
\usepackage{tikz}
\usepackage{floatrow}
\usepackage{mathtools}

\usepackage{stackrel}
\usepackage{multicol}

\usepackage{amsmath,amsthm,amssymb,setspace,enumitem,epsfig,titlesec,verbatim,color,array,eurosym,multirow,blkarray,lscape,bm,nicefrac,comment,soul,graphicx,MnSymbol,wasysym,subfigure,multirow}
\usepackage[top=2.5cm,left=2.8cm,right=2.8cm,bottom=2.5cm]{geometry}
\usepackage{float}

\usepackage{cancel}
\usepackage{ulem}
\usetikzlibrary{decorations.pathreplacing,angles,quotes}

\newtheorem{theorem}{Theorem}
\numberwithin{theorem}{section}

\newtheorem{lemma}[theorem]{Lemma}
\newtheorem{proposition}[theorem]{Proposition}
\newtheorem{example}[theorem]{Example}
\newtheorem{definition}[theorem]{Definition}

\theoremstyle{definition}
\newtheorem{remark}[theorem]{Remark}
\definecolor{G1}{rgb}{0.0, 0.5, 0.0}
\definecolor{FG}{rgb}{0.0, 0.5, 0.0}

\newcommand{\Z}{\mathbb{Z}}
\newcommand{\T}{\mathbb{T}}
\newcommand{\R}{\mathbb{R}}

\allowdisplaybreaks
\usepackage{authblk}

\title{Phase Plane Analysis  on Time Scales for a Lotka--Volterra Competition Model}
\date{ }

\author[1,2]{Sabrina H. Streipert}
\author[3]{Gail S. K. Wolkowicz}
\affil[1]{Department of Mathematics, University of Pittsburgh, PA}
\affil[2]{NSF-Simons National Institute for Theory and Mathematics in Biology, Chicago, IL}
\affil[3]{Department of Mathematics and Statistics, McMaster University, Canada}

\begin{document}

\maketitle

\begin{abstract}
    In this work, we formulate a two-species Lotka--Volterra (LV) competition model on an arbitrary time scales, that contains as special cases the classical LV competition model on a continuous domain, as well as its discrete analogue, the Leslie--Gower model. To derive its global dynamics, we introduce a dynamic dynamic phase plane  analysis that extends the traditional phase plane analysis  that is known to be a powerful tool in the analysis of planar differential equations. In the case of  the discrete time scale $\T=\Z$, the method is consistent with the augmented phase portrait introduced for  discrete planar maps. However, for non-constant graininess, this dynamic phase plane is time-dependent. Despite the dynamic character, we argue that, in some cases,  time-independent 
    information  can  be identified that results in the  determination of the global dynamics of  a planar dynamic system.  
    This dynamic phase plane method, albeit only demonstrated in the context of the particular two-species competition model, can be extended to other planar systems on time scales, providing a novel technique to study the global dynamics of planar systems on time scales that unify discrete and continuous modeling. 
\end{abstract}

\section{Introduction}

The theory of time scales, developed in \cite{Hilger1988}, has gained attention due to its unifying character of continuous and discrete calculus and its resulting potential for real-life applications. By generalizing   differentiation and integration to a general nonempty closed subset of the real numbers, called a time scale and denoted by $\T$, equations involving the (delta) derivative of an unknown function that describe a process evolving on the time domain $\T$ can be formulated. For example, $\mathbf{y}^\Delta = A \mathbf{y}+\mathbf{b}$ is a system of linear dynamic equations, where $\mathbf{y}^\Delta$ is the time scales analogue of a derivative on the time domain $\T$, called (delta) derivative, $A\in \R^{n\times n}, \mathbf{b}\in \R^n$, and $\mathbf{y}\colon \T \to \R^n$. %For $t_0\in \T$, the solution of this dynamic equation is given by $y(t)=e_a(t,t_0)y(t_0)+\int_{t_0}^t e_a(t,\sigma(s))b(s)\, \Delta s$, where $e_a(t,t_0)$ is the time scales analogue of the exponential function and the integral is over the measurable set $\T$. 
In the special case, when $\T=\R$, the delta derivative of $f\colon \T\to \R$ is identical to the classical derivative  and in the special case of $\T=\Z$, the delta derivative is defined by the forward Euler operator so that $f^\Delta(t)=\Delta f(t)=f(t+1)-f(t)$. The study of such {\it dynamic equations} on time scales unifies and extends the theories of differential and difference equations.

Since time scales allows for any nonempty closed subset of the real numbers, the time domain over which time-dependent processes take place could be, for example, the union of  isolated points (e.g., $\T=\bigcup_i \{t_i\}$ for $t_i<t_{i+1}$ and $t_i\in \R$)  or the union of continuous time intervals and isolated time points (e.g., $\T=\bigcup_{i=0}^\infty [a_i,b_i] \cup \{c_i\}$ for $a_i<b_i<c_i<a_{i+1}$ and $a_i,b_i, c_i\in \R$ for all $i\geq 0$), making it a suitable framework to study processes that change at non-equidistant time points. This is particularly useful if we describe, for example, different seasonal behaviors of species, where species reproduce continuously during one season but hibernate during another season \cite{Bohner1, StTS}. The theory of time scales is a modern field of applied analysis with established fundamentals (see, e.g., textbooks \cite{Bohner1, Bohner2}) and many nontrivial results including conditions for the existence and uniqueness of solutions to dynamic equations and Sturm-Liouville theory on time scales. Extensions have included perturbation theory \cite{Olszewski}, Liapunov stability analysis \cite{KloedenZmorzynska2006, martynyuk2016, SIEGMUND2002255}, and  Laplace transforms  \cite{ Bohner1, Bohner2, DAVIS20071291}.

Despite this progress on time scales, many analytical techniques that are well established in the theory of differential equations remain open problems in the theory of time scales. One such powerful technique is the {\it phase plane analysis} that is a standard method used to assess the dynamics of planar differential equations. For example, the solutions to the popular two species Lotka--Volterra competition model can be fully determined by this phase plane analysis \cite{Edelstein1988}. Already considered ineffective in the study of planar discrete maps, this method has certainly been neglected in the study of time scales. Motivated by a recent advancement  in \cite{StWo}, where the authors introduced an {\it augmented phase plane analysis} to provide insights into the global dynamics of discrete planar maps, we derive such technique  for the general case of dynamic equations on time scales. More precisely, we formulate a {\it dynamic (augmented) phase plane analysis on time scales} to gain insights into planar dynamic equations. As an example, we introduce a time scales analogue of the two species Lotka--Volterra competition model and use our method to derive its global dynamics.

The Lotka--Volterra competition model is one of the classical mathematical models describing inter-specific competition, pioneered independently by Alfred J.~Lotka  \cite{Lotka1925} and Vito Volterra \cite{Volterra1926a,Volterra1926b}. 
These models became foundational for theoretical ecology and helped 
formalize ideas such as competitive exclusion and coexistence 
\cite{Kingsland2015}. 
The classical two-species competition model is given by
\begin{equation}\label{LVR}
    x'=rx\left(1-\frac{x}{K}\right)-\widetilde{\alpha} xy, \qquad \qquad y'=sy\left(1-\frac{y}{L}\right)-\widetilde{\beta} xy,\qquad \qquad t\in \mathbb{R},
\end{equation}
where $x(t), y(t)$ represent the size of the two competing species at time $t\in\mathbb{R}$,  $r, s>0$ represent the intrinsic growth rates for species $x$ and $y$, respectively,  $K, L>0$ are the corresponding carrying capacities, and $\widetilde{\alpha}, \widetilde{\beta}>0$ give the strength of the competition between these two species.

A fundamental property of this model is that  generically,  except for the generic case of a line of equilibria, it admits four possible 
equilibria: extinction of both species, survival of only one species  $x$ or $y$,
and the coexistence of both species. 
The classical phase plane analysis can be used to describe the global dynamics and show that the outcome depends on the relative
strength of intra- and inter-specific competition \cite{Edelstein1988}. 
If inter-specific competition is sufficiently strong, one species
excludes the other, leading to the classical {\it competitive exclusion}
principle. 
In contrast, if intra-specific competition dominates, the model predicts
stable coexistence of the two species. 
These conclusions can be determined geometrically from the relative
positions of the zero-growth isoclines in the $(x,y)$ phase plane.

The following two species recurrence is often considered as  its discrete analogue:
\begin{equation}\label{LVZ}
    x_{t+1} = \frac{(1+r) x_t}{1+\frac{r}{K}x_t+\widetilde{\alpha} y_t}, \qquad \qquad     y_{t+1} = \frac{(1+s) y_t}{1+\frac{y}{L}y_t+\widetilde{\beta} x_t},\qquad \qquad t\in \mathbb{Z}, 
\end{equation}
where the interpretation of the parameters is the same as in \eqref{LVR}, but instead, $x_t, y_t$ represent the size of the two competing species at the discrete time points $t\in \mathbb{Z}$. This discrete planar system is also referred to as ``Leslie--Gower'' competition model and has been analyzed, for example in \cite{Cushing2004},  using the theory of monotone flows. In \cite{StWo}, the authors  introduce  the augmented phase plane analysis and use it  to obtain the global dynamics of \eqref{LVZ}, providing an alternative proof. They show that \eqref{LVZ} has the same nullclines, equilibria, and dynamics as known for the continuous model \eqref{LVR}. Furthermore,   just as for the analogous ordinary differential equations model  where, in the absence of a competitor,   \eqref{LVR} reduces to  logistic growth of the species so that solutions for positive initial conditions  converge monotonically to the carrying capacity, the solutions of \eqref{LVZ} also converge to their respective carrying capacities in the absence of the competitor. These similarities  in the dynamics  of \eqref{LVR} and \eqref{LVZ} justify the claim that \eqref{LVZ} is the discrete analogue of the continuous 2-species Lotka--Volterra competition model.

 Motivated by the analysis of \eqref{LVZ} in \cite {StWo},  we introduce the dynamic phase plane on time scales and  formulate  a time scales analogue of the 2-species  Lotka--Volterra competition model  that we use  as an example to demonstrate  the effectiveness of  analysis exploiting the dynamic phase plane on time scales.  This extends the work in \eqref{LVZ} to an arbitrary time scale,  where the  local dynamics of a system are already much harder to analyze.  

Throughout, we assume that 
 $\T$ is an arbitrary  time scale that is unbounded from above, i.e.,  there exists an infinite sequence $\{t_n\}$ such that  $t_n \in \T$ and $\lim_{n \to \infty}t_n=\infty$. In case the reader  is not    familiar with the basic time scales theory, we provide a brief summary of the preliminaries in the Appendix~\ref{sec:appendix}. For more background on time scales, the interested reader is referred to the introductory books \cite{Bohner1, Bohner2}.

\section{A Lotka--Volterra Competition Model on Time Scales}

In order to formulate a time scales analogue of the 2-species Lotka--Volterra competition model, we recall that a key feature is that, in the absence of a competitor, a species grows logistically. Here, we consider a single species ``logistic growth'' model   of the form
\begin{equation}\label{eq:logsingle}
z^\Delta = -(\ominus r)z^\sigma \left(1-\frac{z}{K}\right),
\end{equation}
for $r, K>0$ with $\ominus z=\frac{-z}{1+\mu(t) z}$. 
Note  that  \eqref{eq:logsingle} is equivalent to $z^\Delta = rz \left(1-\frac{z^\sigma}{K}\right)$ proposed in \cite{Marcia}, as well as the dynamic Beverton--Holt model  $z^\Delta = \alpha z^\sigma \left(1-\frac{z}{K}\right)$ introduced in \cite{BohnerWarth}, for $\alpha=\frac{r}{1+\mu r}=-(\ominus r)$. 
Since $z^\sigma=z+\mu(t) z^\Delta$,   see Remark~\ref{rem:fsigma}, \eqref{eq:logsingle} can be  expressed equivalently as 
$$z^\Delta = \ominus \left( p(t)+f(t)z\right)z,$$
for time dependent functions\footnote{Since $\ominus r=\frac{-r}{1+\mu(t)r}$ depends on $t$ unless $\mu(t)$ is constant for all $t\in \T$, $p,f\colon \T\to \R$.} $p=\ominus r$ and $f=-\frac{\ominus r}{K}$. 
Thus, \eqref{eq:logsingle} as well as its equivalent expressions in \cite{BohnerWarth} and \cite{Marcia} are special cases of the proposed logistic growth models in \cite{AkinBohner}, see also \cite[p.~18]{Bohner2} (see Appendix~\ref{Ap:Marcia}  for details).

If $\T=\R$, then \eqref{eq:logsingle} is $z'=rz\left(1-\frac{z}{K}\right)$, consistent with the classical logistic growth model. If $\T=\Z$, then \eqref{eq:logsingle} is equivalent to the Beverton--Holt model, since $ z_{t+1}-z_t=\Delta z_t=-(\ominus r)z_{t+1}\left(1-\frac{z_t}{K}\right)=\frac{r}{1+r}z_{t+1}\left(1-\frac{z_t}{K}\right)$, can be rewritten, after collecting the terms for $z_{t+1}$, as $z_{t+1}=\frac{K \rho z_t}{K+(\rho-1)z_t}$ with $\rho=r+1$. 

\begin{remark}
Note that if, instead of \eqref{eq:logsingle},  we  had considered 
 $$z^\Delta =r z^\sigma \left(1-\frac{z}{K} \right),$$
it becomes, after rearranging,
$$z^{\sigma}= \frac{z}{1-r\mu(t)  (1- \frac{z}{K})}.$$
This expression shows that  solutions could become negative for sufficiently large $\mu(t)$, resulting in biologically unrealistic values. We would then have to restrict the parameter range as in \cite{BohnerWarth} or the choice of time scales to ensure that the solutions remain nonnegative.

Instead, rearranging \eqref{eq:logsingle},
we have
$$z^{\sigma}= \frac{(1+r \mu(t) ) z}{1+ r \mu(t)  ( \frac{z}{K})}, $$
resulting in  a model that has nonnegative solutions for nonnegative initial conditions for arbitrary positive constants $r, K$ and arbitrary time scales. In fact, \eqref{eq:logsingle} can be solved explicitly by considering the transformation $u=\frac{1}{z}$, which transforms the nonlinear dynamic equation into a linear nonhomogeneous dynamic equation (see Appendix \ref{SolSingleT}).  
\end{remark}
%%%%%%%%%%%%%%%%

Accepting \eqref{eq:logsingle} as a time scales analogue of the Verhulst logistic growth model $z'=rz\left(1-\frac{z}{K}\right)$, we  propose the following dynamic system  as a  two-species competition model on time scales:
\begin{equation}\label{eq:CompT}
\begin{split}
    x^\Delta &=F(x^\sigma, y^\sigma,x, y)= -(\ominus r) x^\sigma \left(1-\frac{x}{K}-\alpha y\right), \qquad \qquad t\in \mathbb{T},\\
    y^\Delta &=G(x^\sigma, y^\sigma, x,y)=-(\ominus s) y^\sigma \left(1-\frac{y}{L}-\beta x\right), \qquad \qquad t\in \mathbb{T}
    \end{split}
\end{equation}
with nonnegative initial conditions $x(t_0), y(t_0)\geq 0$ for some $t_0\in \mathbb{T}$ and model parameters $r, s,  \alpha, \beta, K, L>0$.

If $\mathbb{T}=\mathbb{R}$, then \eqref{eq:CompT} reduces  to \eqref{LVR} for $\widetilde{\alpha}=r\alpha$ and $\widetilde{\beta}=r\beta$, because $x^\sigma = x$, $y^\sigma = y$ and $-(\ominus r)=r, -(\ominus s)=s$. 
Similarly, if $\mathbb{T}=\mathbb{Z}$,   $-(\ominus r)=r/(1+r), -(\ominus s)=s/(1+s)$ and then \eqref{eq:CompT} reduces to \eqref{LVZ}  with\footnote{Alternatively, choosing $\alpha=\frac{\widetilde{\alpha}}{r}$ and $\beta=\frac{\widetilde{\beta}}{r}$ in \eqref{eq:CompT} ensures that for $\T=\R$, \eqref{eq:CompT}  reduces to \eqref{LVR} and for $\T=\Z$, \eqref{eq:CompT} reduces to \eqref{LVZ}.} $\widetilde{\alpha}=r\alpha$ and $\widetilde{\beta}=s\beta$. 
We will  provide further  justification that  \eqref{eq:CompT} represents a time scales analogue of the classical Lotka--Volterra competition model by showing that the dynamics of \eqref{eq:CompT}  are the same as the well-known dynamics for the classical models \eqref{LVR} and \eqref{LVZ}.

%%%
We  first note that \eqref{eq:CompT} can be written equivalently  as 
\begin{align}\label{eq:xy_Delta}
\begin{split}
  x^{\Delta}&=\frac{rx(1-\frac{x}{K}-\alpha y)}{1 +r \mu(t) (\frac{x}{K}+\alpha y)},\\  
  y^{\Delta}&=\frac{sy(1-\frac{y}{L}-\beta x)}{1 +s \mu(t) (\frac{y}{L}+\beta x)}
  \end{split}
\end{align}
 and implies the relation
\begin{align}\label{eq:xy_sigma}
\begin{split}
  x^{\sigma}&=x+\mu x^\Delta= \frac{x(1+r\mu(t) )}{1+r\mu(t) (\frac{x}{K}+\alpha y)},    \\  
  y^{\sigma}&=y+\mu y^\Delta = \frac{y(1+s\mu(t) )}{1+s\mu(t) (\frac{y}{L}+\beta x)},
  \end{split}
\end{align}
where we used that for a differentiable function $f$, $f^\sigma = f+\mu f^\Delta$ (see Appendix Remark~\ref{rem:fsigma}).

To obtain the 
$x$-and $y$-nullclines of \eqref{eq:CompT}, we find curves in the nonnegative $x,y$-plane such that   $x^\Delta = 0$ and $y^\Delta=0$, respectively. For \eqref{eq:CompT}, we obtain  the following functions for the nontrivial nullclines of the $x$-equation
\begin{equation}\label{eq:NCx}
        x=h_{0}(y):=0 \qquad \mbox{and} \quad y=h(x):=\frac{1}{\alpha}\left(1-\frac{x}{ K}\right)
\end{equation}
and, for the $y$-equation,
     \begin{equation}\label{eq:NCy}
        y=k_0(x):= 0 \qquad \mbox{and} \quad y=k(x):=L(1-\beta x).
\end{equation}
The equilibria, i.e., constant solutions, of \eqref{eq:CompT} are: 
\begin{equation}\label{eq:equilbria}
E_0^*=(0,0), \quad E_K^*=(K,0), \quad E_L^*=(0,L), \quad E^*=(x^*,y^*)=\left(\frac{K(\alpha L-1)}{\alpha \beta K L -1}, \frac{L(\beta K-1)}{\alpha \beta K L -1}\right),
\end{equation}
where the coexistence equilibrium $E^*$  is  (biologically) relevant/feasible (i.e., $x^*>0$ and $y^*>0$) if and only if $(\alpha L-1)(\beta K-1)>0$. All four equilibria are consistent with the equilibria obtained for the classical 2-species Lotka--Volterra model \eqref{LVR} and its discrete counterpart \eqref{LVZ}

\section{Dynamic Phase Plane}

We now introduce a method, using what  we refer to as  the {\it dynamic phase plane}, that extends the classical phase plane analysis, as well as the recently developed {\it augmented phase plane} analysis  for discrete planar models \cite{StWo}, to models on time scales. We will then illustrate this dynamic phase plane on the example of the two  species competition model  formulated in Section 2.

\subsection{Root-sets and Root-operators}

For each nullcline of a planar dynamic system on time scales, we define its corresponding {\it Root-set}. 

\begin{definition}\label{def:rootset}[Root-set]
    Consider \eqref{eq:CompT}. 
    
    Let $y=\ell(x)$, for $\ell\colon \R\to \R$, be a nullcline. The time-dependent \underline{Root-set} associated with the nullcline  $y=\ell(x)$ is defined as
    $$\mathcal{R}_{\ell}(t):=\{(x,y)\, \colon\, y+\mu(t) G(x^\sigma, y^\sigma, x,y) = \ell(x+\mu(t) F(x^\sigma, y^\sigma, x,y))\}.$$
   
    If  instead, $x=\widehat{\ell}(y)$ is a nullcline, where  $\widehat{\ell}:\R \to \R$,   then we define the \underline{Root-set} associated with the nullcline $x=\widehat{\ell}(y)$ as
    $$\mathcal{R}_{\widehat{\ell}}(t):=\{(x,y)\, \colon\, x+\mu(t) F(x^\sigma, y^\sigma, x,y) = \widehat{\ell}(y+\mu(t) G(x^\sigma, y^\sigma, x,y))\}.$$
\end{definition}

If the Root-set associated with a nullcline $y=\ell(x)$ can be expressed geometrically as a curve, we also refer to the Root-set as  the ``Root-curve''.
Based on Definition~\ref{def:rootset}, the Root-set associated with a nullcline is the preimage of the nullcline at the time point $\sigma(t)$ for any right-scattered points $t$. In other words,  since $y^\sigma=y+\mu(t)G(x^\sigma, y^\sigma, x, y)$ and $x^\sigma=x+\mu(t)F(x^\sigma, y^\sigma, x, y)$, the Root-set for the nullcline $y=\ell(x)$ are all points such that $y^\sigma=\ell(x^\sigma)$. Thus, at time point $\sigma(t)$, the $y$-value is on the associated nullcline. A key feature and difference from the work in \cite{StWo} is that the Root-set  is  time-dependent  whenever the graininess $\mu(t)$ is not constant.

If $\T=\R$, then, by Definition~\ref{def:rootset}, $\mathcal{R}_\ell(t)$ simply consists of  the points on the nullcline $y=\ell(x)$. Instead, if $\T=\Z$, then Definition~\ref{def:rootset} defines the Root-set of the nullcline $y=\ell(x)$ as 
$$\mathcal{R}_{\ell}(t)=\{(x,y)\, \colon\, y_{t+1} = \ell(x_{t+1})\}.$$
This is consistent with the Definition of Root-sets in \cite{StWo}. In general, if $\T$ is an arbitrary time scale and $t\in \T$ is right-dense, then the Root-set $\mathcal{R}_\ell(t)$, at time $t$,  consists of the nullclines (since $\mu(t)=0$) but if $t\in \T$ is right-scattered, then $\mathcal{R}_\ell(t)$, at time $t$, takes the form 
\begin{equation}\label{eq:Rellrightscat}
\mathcal{R}_{\ell}(t)=\{(x,y)\, \colon\, y^\sigma = \ell(x^\sigma)=\ell(x+\mu(t) x^\Delta)\},
\end{equation}
highlighting the time-dependent nature of the Root-sets.

\begin{definition}[Root-operator]\label{def:Rootop}
    Consider \eqref{eq:CompT} and let $y=\ell(x)$ for $\ell\colon \R\to \R$ be a nullcline. We define the {\it Root-operator associated with $y=\ell(x)$} as
    $$\mathcal{L}_{\ell}(t, x^\sigma, y^\sigma, x,y):= y+\mu(t) G(x^\sigma, y^\sigma, x,y) -\ell(x+\mu(t) F(x^\sigma, y^\sigma, x,y))
    =y^\sigma-\ell(x^\sigma)
    .$$
    If, instead, $x=\widehat{\ell}(y)$ for $\widehat{\ell}\colon \R\to \R$ is a nullcline, then we define the {\it Root-operator associated with $x=\widehat{\ell}(y)$} as
    $$\mathcal{L}_{\widehat{\ell}}(t, x^\sigma, y^\sigma, x,y):= x+\mu(t) F(x^\sigma, y^\sigma, x,y) -\widehat{\ell}(y+\mu(t) G(x^\sigma, y^\sigma, x,y))
    =x^\sigma-\widehat{\ell}(y^\sigma)
    .$$
\end{definition}

In particular, it is the sign of the Root-operator associated with a nullcline that is informative, as the following theorem highlights. 

\begin{theorem}\label{thm:need}
Consider \eqref{eq:CompT}. 
Let $\ell \in C(\R,\R)$ such that $y=\ell(x)$ is a nullcline. Let  $(x(t),y(t))$ be a solution of \eqref{eq:CompT}. If $\mathcal{L}_{\ell}(T,x(T),y(T))>0$, then there exists $\epsilon>0$ such that 
$y(t)>\ell(x(t))$ for $t\in [T,T+\epsilon)$. If $T$ is right-scattered, then $\epsilon>\sigma(T)$ so that $y(t)>\ell(x(t))$ for $t\in [T,\sigma(T)]$. If $\mathcal{L}_{\ell}(T,x(T),y(T))<0$, then there exists $\delta>0$ such that 
$y(t)<\ell(x(t))$ for $t\in [T,T+\delta)$. If $T$ is right-scattered, then $\delta>\sigma(T)$ so that $y(t)<\ell(x(t))$ for $t\in [T,\sigma(T)]$.
\end{theorem}

\begin{proof}
    If $T$ is right-scattered, by Def.~\ref{def:Rootop},  $\mathcal{L}_{\ell}(T,x(T),y(T))>0$ implies that $y^\sigma(T)>\ell(x^\sigma(T))$. Thus, for right-scattered points $T$, $\mathcal{L}_{\ell}(T,x(T),y(T))>0$ implies 
    $y(s)>\ell(x(s))$ for $s\in \{T,\sigma(T)\}$. The same argument can be used if $\mathcal{L}_{\ell}(T,x(T),y(T))<0$, confirming the statement for right-scattered points. 
    
    If instead $T$ is right-dense, then $\mathcal{L}_{\ell}(T,x(T),y(T))>0$ implies that $y(T)>\ell(x(T))$ i.e., $(x(T),y(T))\in \Omega:=\{(x,y)\, \colon \, y>\ell(x)\}$. By the density of real numbers, there exists $\widehat{\epsilon}>0$ such that $U_{\widehat{\epsilon}}(x(T),y(T))\cap \Omega\neq \emptyset$, where $U_{\widehat{\epsilon}}(x(T),y(T))=\{(u,v)\colon (u-x(T))^2+(v-y(T))^2<\widehat{\epsilon}^2\}$. We now argue that there exists a corresponding $\epsilon>0$ such that $(x(t),y(t))\in U_{\widehat{\epsilon}}(x(T),y(T))\cap \Omega\subset \Omega$ for $t\in [T,T+\epsilon)$.
   Since $T$ is right-dense,  there exists $\{t_n\}_n\subset \T$ such that $T<\ldots<t_{n+1}<t_{n}$ with $\lim_{n\to \infty}t_n=T$.  Furthermore, since $(x(t),y(t))$ is a solution of \eqref{eq:CompT}, the components of the solution are delta-differentiable and  therefore, by \cite[Theorem 1.16i)]{Bohner1} continuous and  by \cite[Theorem 1.60]{Bohner1} rd-continuous. That is, $x,y\colon \T\to \R$ are continuous (from the right) at any right-dense point (here $T$). Thus, there exists $M>0$ such that $(x(t_i),y(t_i))\in U_{\widehat{\epsilon}}(x(T),y(T))\cap \Omega$ for all $i>M$. This implies that there exists  
   $\epsilon>0$ such that $y(t)>\ell(x(t))$ for all $t\in [T, T+\epsilon)$.
\end{proof}

 In particular, for \eqref{eq:CompT},   by Definition~\ref{def:Rootop}, 
for the trivial nullclines $x=h_0(y)\equiv 0$  and  $y=k_0(x)\equiv 0$, we obtain
\begin{align}
    \mathcal{L}_{h_0}(t,x,y)=x+\mu(t)F(x^\sigma, y^\sigma, x, y)\stackrel{\eqref{eq:xy_sigma}}{=}\frac{x(1+r\mu(t))}{1+r\mu(t)\left(\frac{x}{K}+\alpha y\right)},\label{eq:Lh0}\\
    \mathcal{L}_{k_0}(t,x,y)=y+\mu(t)G(x^\sigma, y^\sigma, x, y)\stackrel{\eqref{eq:xy_sigma}}{=}\frac{y(1+s\mu(t))}{1+s\mu(t)\left(\frac{y}{L}+\beta x\right)}.\label{eq:Lk0}
\end{align}
Hence, for $x,y\in (0,\infty)^2$, $\mathcal{L}_{h_0}(t,x,y)>0$ and $\mathcal{L}_{k_0}(t,x,y)>0$. Recall, for finite $t$, $\mu(t)<\infty$. Thus, for finite $t$, $\mathcal{L}_{h_0}(t,x,y)=0$,  if and only if $x=0$ and  $\mathcal{L}_{k_0}(t,x,y)=0$ if and only if $y=0$. 
Thus, the corresponding Root-sets consist of the respective axis.

 For \eqref{eq:CompT} and its nontrivial nullclinesin  \eqref{eq:NCx}-\eqref{eq:NCy}, we have, by Definition~\ref{def:Rootop}, 
\begin{equation}\label{eq:LhLk}
    \begin{split}
\mathcal{L}_h(t,x,y)&=\frac{y(1+s\mu(t) )}{1+s\mu(t) (\frac{y}{L}+\beta x)} - \frac{1}{\alpha} + 
\frac{ 1}{\alpha K}\frac{x(1+r\mu(t) )}{1+r\mu(t) (\frac{x}{K}+\alpha y)},  \\
\mathcal{L}_k(t, x,y)&=\frac{y(1+s\mu(t))}{1+s\mu(t) (\frac{y}{L}+\beta x)} - L+L\beta \frac{x(1+r\mu(t) )}{1+r\mu(t) (\frac{x}{K}+\alpha y)}.
    \end{split}
\end{equation}
For a right-dense point, $\mathcal{L}_h(t, x,y)=y(t)-h(x(t))$ and $\mathcal{L}_k(t, x,y)=y(t)-k(x(t))$.

To discuss the dynamics of \eqref{eq:CompT}, we make use of the signs of the Root-operators associated with the nullclines, $\mathcal{L}_h(t,x,y)$ and $\mathcal{L}_k(t,x,y)$  in regions of  the ${\rm int}\mathbb{R}^2_+ =(0,\infty)\times (0,\infty)$. In order to do this, the following observations and notation will be useful.  Applying Def.~\ref{def:Rootop} to \eqref{eq:CompT},  the Root-operators defined in \eqref{eq:LhLk}  can be expressed as rational functions of the graininess $\mu(t)$. More precisely,  
\begin{equation}\label{eq:Lh}
    \mathcal{L}_h(t,x,y)=\frac{a_2(x,y)\mu^2(t)+a_1(x,y)\mu(t)+a_0(x,y)}{\alpha (K + r\mu(t)  x + \alpha  r\mu(t) K y) (L +
      \beta s \mu(t) L x + s \mu(t)  y)},
 \end{equation}     
where
\begin{align}\label{eq:ai}
\begin{split}
a_0(x,y)&=L (x -K+\alpha y K )=\alpha K L(y-h(x)),\\
a_1(x,y)&= \beta s  L x (x-K) + 
  y(s (x-K) + \alpha L (r(x-K)  + sK ) ) +r\alpha^2  y^2K L  ,\\
a_2(x,y)&= \alpha r s  y(
yK (\alpha L-1) + xL  (1 - \beta K)).
\end{split}
\end{align} 
Since the denominator of $\mathcal{L}_h(t,x,y)$ is always positive, its sign is determined by the sign of its numerator, 
$N_a(t,x,y)= a_2(x,y)\mu^2(t)+a_1(x,y)\mu(t)+a_0(x,y).$

\begin{equation}\label{eq:Dai}
\frac{\partial a_1(x,y)}{\partial x}= -s \beta K L + 2s\beta L x + (r \alpha L +s)y
\end{equation}
and 
\begin{equation}\label{eq:Daiy}
\frac{\partial a_1(x,y)}{\partial y}= -K(r \alpha L+s(1-\alpha L))   +  (r \alpha L +s)x   +2r \alpha^2KL y.
\end{equation}

Note that $a_1(x,y)$ is convex in $x$ for fixed $y$ and convex in $y$ for fixed $x$, since
\begin{equation}\label{eq:partials_ai}
\frac{\partial^2 a_1(x,y)}{\partial x^2} = 2s\beta L>0 \quad \mbox{and} \quad \frac{\partial^2 a_1(x,y) }{\partial y^2} =2r \alpha^2  KL>0.
\end{equation}

Similarly, by \eqref{eq:LhLk}, 
\begin{equation}\label{eq:Lk}
\mathcal{L}_k(t,x,y)=\frac{L(b_2(x,y) \mu^2(t)+b_1(x,y)\mu(t)+b_0(x,y))}{{(K + r \mu(t)  x + \alpha r \mu(t) K y) (L +
      \beta s  \mu(t) L x +s  \mu(t)  y)}},
\end{equation}
where
\begin{align}\label{eq:bi}
\begin{split}
b_0(x,y)&=KL(y-k(x)), \\
b_1(x,y)&= L(s x \beta K(y-k(x)) +  rxL(\beta  K-1) +\alpha r y K(y-L)  +rxy), \\
b_2(x,y)&=rs \beta Lx(Lx(\beta K-1)+Ky(1-\alpha L)). 
\end{split}
\end{align}
 Again, since the denominator of $\mathcal{L}_k(t,x,y)$ is always positive, the sign is determined by  its numerator, 
$N_b(t,x,y)= b_2(x,y)\mu^2(t)+b_1(x,y)\mu(t)+b_0(x,y)$.
\begin{equation}\label{eq:Dbi}
\frac{\partial b_1(x,y)}{\partial x}= 
 L^2(r(\beta K -1)- s \beta K  )+ 2s  \beta^2 KL^2 x+ L(r   +  s \beta K ) y
\end{equation}
and
\begin{equation}\label{eq:Dbiy}
\frac{\partial b_1(x,y)}{\partial y}= -\alpha rK L^2 +L(r+s \beta K )x+2\alpha K L r y.
\end{equation}

Note that $b_1$ is convex in $x$ for fixed $y$ and convex in $y$ for fixed  $x$, since
\begin{equation}\label{eq:partials_bi}
 \frac{\partial^2 b_1}{\partial x^2}=2s\beta^2 K L^2 >0   \quad \mbox{and} \quad \frac{\partial^2 b_1}{\partial y^2}=2\alpha  r K L >0.
 \end{equation}

\subsection{Global Analysis of the Time Scales Lotka--Volterra Model}

We now analyze the dynamics of \eqref{eq:CompT} for nonnegative initial conditions and time scales that are unbounded from above. To obtain the global dynamics of the solutions, we use the newly defined Root-sets and Root-operators for nontrivial nullclines.

The following notation will  be used in the proofs that follow.

For any set $\Lambda$, let  $\overline{\Lambda}$ denote the closure of 
 the set $\Lambda$, ${\rm int} \Lambda$, the interior of the set $\Lambda$, and  $\partial \Lambda$, the boundary of the set $\Lambda$.
 Furthermore, we define the time scales intervals $[a,b]_{\T}=[a,b]\cap \T=\{t\in \T\, \colon\, a\leq t\leq b\}$. The semi-open intervals $(a,b]_{\T}$ and $[a,b)_{\T}$ as well as the open interval $(a,b)_{\T}$ are defined accordingly.

The following results are basic but crucial observations for solutions of \eqref{eq:CompT}. The first proposition ensures that solutions of \eqref{eq:CompT} are biologically relevant, supporting the claim that \eqref{eq:CompT} is biologically meaningful to represent a two-species competition model.  Instead of classical techniques, we provide a proof using the Root-operators associated with  the trivial nullclines $y=h_0(x)$ and $y=k_0(x)$.

\begin{proposition}\label{prop:nonneg}
    Consider \eqref{eq:CompT}. 
    \begin{itemize}
        \item[a)] The interior of the first quadrant is positively invariant, i.e., if $x(T)>0$ and $y(T)>0$,  then $x(t)>0$ and  $y(t)>0$,  for all $t\in [T,\infty)_{\T}$.   
    \item[b)] The axes in the first quadrant are positively invariant, i.e.,
    if $x(T)=0$,  and $y(T)\geq 0$, then $x(t)=0$ and $y(t)\geq 0$ for all $t\geq T, t\in \mathbb{T}$  and   if $y(T)=0$,  and $x(T)\geq 0$, then $y(t)=0$ and $x(t)\geq 0$ for all $t\in [T,\infty)_{\T}$.
    \end{itemize}
\end{proposition}

\begin{proof}
    \noindent {\it a)} \, By \eqref{eq:Lh0}, $\mathcal{L}_{h_0}(t,x,y)>0$ for $x>0$, and by \eqref{eq:Lk0}, $\mathcal{L}_{k_0}(t,x,y)>0$ for $y>0$. Thus, for $x_0, y_0>0$,    $\mathcal{L}_{h_0}(t,x_0,y_0)>0$ and  $\mathcal{L}_{k_0}(t,x_0,y_0)>0$. By Theorem~\ref{thm:need}, $x(t), y(t)>0$ for all $t\in (t_0,\infty)_\T$.  

\noindent {\it b)}\,  Assume that there exists $T$ such that $x(T)=0$ and $y(T)\geq 0$. To show that $x(t)=0$ for all $t\in [T, \infty)_\T$, we proceed by contradiction. First, note that $x(t)\equiv 0$ and by Proposition~\ref{prop:1D_model},  $y(t)=\frac{e_s(t,T)Ly(T)}{L+y(T)(e_s(t,T)-1)}$ is a solution of \eqref{eq:CompT}.   Furthermore, by 
\cite[Theorem~8.16]{Bohner1}, 
there exists $\alpha>0$ such that there exists a unique solution for $[T-\alpha, T+\alpha]_\T$, where, if $T$ is right-scattered, the solution is unique for $[T-\alpha, \sigma(T)]_{\T}$.
Since  $x(t)\equiv 0$ and by Proposition~\ref{prop:1D_model},  $y(t)=\frac{e_s(t,T)Ly(T)}{L+y(T)(e_s(t,T)-1)}$ is a solution, 
these two solutions must be the same on their interval of existence,  $[T,T+\beta]_\T$ for $\beta\geq  \max\{\alpha, \sigma(T)\}$. 
Since this holds for all $T\in \T$, it implies that any solution for which $x(T)=0$, 
$x(t)=0$ for all $t\in [T, \infty)_\T$ and 
$y(t)=\frac{e_s(t,T)Ly(T)}{L+y(T)(e_s(t,T)-1)}$, 
for all $t\in [T, \infty)_\T$.  
\end{proof}

The nonnegativity of the solutions formulated in Proposition~\ref{prop:nonneg} allows us to prove the biologically relevant property that  solutions are bounded, as stated in the following theorem with a proof in Appendix~\ref{A:Proofbnd}.

\begin{theorem}\label{thm:bound}
    Consider \eqref{eq:CompT}.  Solutions with nonnegative initial conditions are bounded.
\end{theorem}
%%%%%%%%%%%%%%%%%%%%%%%%%%%%%%%%%%%%%%%%

We next provide  the following crucial result that will assist us in discussing the global dynamics. For the proof, we define 
$$\mu_m:=\liminf_{t \to \infty, t \in \T} \mu(t), \qquad \qquad \mu_M:=\limsup_{t \to \infty, t \in \T} \mu(t).$$

\begin{proposition}\label{prop:convergence}
Consider model \eqref{eq:CompT}.   If  there exists $T \in \T$ such that  both components are component-wise monotone for all $t\in [T, \infty)_\T$, then the solution converges to an equilibrium as $t\to \infty$\footnote{We say $f\colon \T\to \R$ converges to $\overline{f}\in \R$ as $t\to \infty$ if, for any sequence $\{t_i\}_{0}^\infty$ with $t_i<t_{i+1}$ for all $i\in \mathbb{N}$ and $\lim_{i\to \infty}t_i=\infty$, $\lim_{i\to \infty}f(t_i)=\overline{f}$.}. 
\end{proposition}

\begin{proof}
By Theorem~\ref{thm:bound}, all solutions are bounded.  Since $\T$ is unbounded from above, and by assumption, solutions  are  eventually monotonic,  each component must converge to a finite value. Assume that $(x(t),y(t))\to (\widetilde{x},\widetilde{y})$  as $t \to \infty$.  To show that the solution converges to an equilibrium, we show that   $(x^{\Delta},y^{\Delta})\to(0,0)$ as $t \to \infty$ and hence $(\widetilde{x},\widetilde{y})$ is an equilibrium.

First note that if $\T$ is eventually right-dense, then the result follows directly from the classical theory of differential equations.

Next, consider the case,   when $\mu_m=\infty$.  Since $(x,y)$ converges to $(\widetilde{x},\widetilde{y})$,  
\begin{align*}
\widetilde{x}=\lim_{t\to\infty}x^\sigma(t)\stackrel{\eqref{eq:xy_sigma}}{=}\lim_{t\to\infty}\frac{x(t)(1+r \mu(t))}{1+r\mu(t)\left(\frac{x(t)}{K}+\alpha y(t)\right)}=\lim_{t\to\infty}\frac{x(t)\left(\frac{1}{\mu(t)}+r\right)}{\frac{1}{\mu(t)}+r\left(\frac{x(t)}{K}+\alpha y(t)\right)}
=\frac{\widetilde{x}}{\frac{\widetilde{x}}{K}+\alpha \widetilde{y}}.
\end{align*}
Similarly, 
$\widetilde{y} =\lim_{t\to \infty}y^\sigma(t)
=\frac{\widetilde{y}}{\frac{\widetilde{y}}{L}+\beta \widetilde{x}}$. 
Solving these two equations for $(\widetilde{x}, \widetilde{y})$ gives the equilibrium values found in \eqref{eq:equilbria}.

We distinguish between two cases:  {\it a)}  $\mu_M>0$ and {\it b)} $\mu_M=0$. In fact, we will show the second part more generally for any $0\leq \mu_m\leq \mu_M<\infty$.

Case a): Assume $\mu_M>0$. 
Since $\mu_M = \limsup_{t\to \infty}\mu(t)$, there exists $\{t_i\}\subset \T$ such that $\lim_{i\to \infty}t_i=\infty$ and $\lim_{i\to \infty}\mu(t_i)=\mu_M$.  Assume first that $\mu_M<\infty$.  Then,  
\begin{align*}
\widetilde{x}=\lim_{t\to\infty}x^\sigma(t)
=\lim_{i\to \infty} x^\sigma(t_i)
\stackrel{\eqref{eq:xy_sigma}}{=}\lim_{i\to\infty}\frac{x(t_i)(1+r \mu(t_i))}{1+r\mu(t_i)\left(\frac{x(t_i)}{K}+\alpha y(t_i)\right)}=\frac{\widetilde{x}(1+r \mu_M)}{1+r\mu_M\left(\frac{\widetilde{x}}{K}+\alpha \widetilde{y}\right)}.
\end{align*}
Therefore, 
\begin{equation}\label{eq:must1}
\widetilde{x}=\frac{\widetilde{x}(1+r \mu_M)}{1+r\mu_M\left(\frac{\widetilde{x}}{K}+\alpha \widetilde{y}\right)}.
\end{equation}
Similarly, %for $\widetilde{y}$, 
\begin{equation}\label{eq:must2}
\widetilde{y}=\frac{\widetilde{y}(1+s \mu_M)}{1+s\mu_M\left(\frac{\widetilde{y}}{L}+\beta \widetilde{x}\right)}.
\end{equation}
Solving \eqref{eq:must1} and \eqref{eq:must2}, assuming that $\mu_M>0$, $(\widetilde{x}, \widetilde{y})$ must be an equilibrium as provided in \eqref{eq:equilbria}.   If, on the other hand, $\mu_M$ is infinite, then there exists $\{t_i\}\subset \T$  such that $\lim_{i\to \infty}t_i=\infty$, and then, 
$$\widetilde{x}=\lim_{i\to \infty}\frac{x(t_i)(1+r \mu(t_i))}{1+r\mu(t_i)\left(\frac{x(t_i)}{K}+\alpha y(t_i)\right)}=
\lim_{i\to \infty}\frac{x(t_i)\left(\frac{1}{\mu(t_i)}+r \right)}{\frac{1}{\mu(t_i)}+r\left(\frac{x(t_i)}{K}+\alpha y(t_i)\right)}=\frac{\widetilde{x}}{\frac{\widetilde{x}}{K}+\alpha \widetilde{y}}$$
and similarly, 
$$\widetilde{y}=\frac{\widetilde{y}}{\frac{\widetilde{y}}{L}+\beta \widetilde{x}},$$
which again has only an equilibrium as solution. 

Case b) Assume now $\mu_M=0$. This implies that $\mu_m=\mu_M=0$ so that $\lim_{t\to \infty}\mu(t)=0$. Then, by \eqref{eq:xy_Delta}, 
\begin{equation*}
\lim_{t\to \infty} x^\Delta (t)
=r\widetilde{x}\left(1-\frac{\widetilde{x}}{K}-\alpha \widetilde{y}\right):=\epsilon_x
\end{equation*}
and similarly, 
\begin{equation*}
\lim_{t\to \infty} y^\Delta (t)
=s\widetilde{y}\left(1-\frac{\widetilde{y}}{L}-\beta \widetilde{x}\right):=\epsilon_y.
\end{equation*}
For the sake of contradiction, suppose that $(\widetilde{x}, \widetilde{y})$ is not an equilibrium. Then, $\epsilon_x, \epsilon_y$ cannot both be zero. If $\epsilon_x\neq 0$, we may assume (without loss of generality) that $\epsilon_x>0$. Then, since $\lim_{t\to \infty}x^\Delta=\epsilon_x>0$, there exists $T$ such that $x^\Delta(t)>\frac{\epsilon_x}{2}$ for all $t\geq T$. Then, by Lemma~\ref{lem:convergence}, $\lim_{t\to \infty}x(t)=\infty$. Note that if $\epsilon_x<0$, then, by a similar argument, $\lim_{t\to \infty}x(t)=-\infty$. In either case, the solution is unbounded, violating Theorem~\ref{thm:bound}. A similar argument applied to the case when $\epsilon_y\neq 0$ completes the contradiction. 
\end{proof}

\subsubsection{Global Dynamics: No interior equilibrium}

Assume that  the nullclines $y=h(x)$ and $y=k(x)$ do not intersect in $\R^+ =(0,\infty)$ so that there is no interior equilibrium. This occurs when 
\begin{enumerate}
    \item[(I)] $\alpha L >1$ and $\beta K<1$ or 
    \item[(II)] $\alpha L<1$ and $\beta K>1$.  
\end{enumerate}
We provide the analysis for (I). In this case, $h(x)<k(x)$ for all $x\in (0,\beta^{-1})$ with $h(x)>0$ for $x\in (0,K)$ and $k(x)>0$ for $x\in (0,\beta^{-1})$. The analysis for (II) is similar. We consider strict inequalities but the same analysis can be completed  if  one of the inequalities is replaced by an equality.

\begin{lemma}\label{lem:signLhLkCase1}
    Consider \eqref{eq:CompT} with $\alpha L >1$ and $\beta K<1$.  Then, the following hold:
    \begin{itemize}
        \item[a)]  $\mathcal{L}_h(t,x,y)>0$ for all  \ $x>0$, $y>\max\{0,h(x)\}$, and $t\in \T$. 
\item[{\it b)}]  $\mathcal{L}_k(t,x,y)<0$ for all  $x>0$, $0<y<k(x)$, and $t\in \T$. 
    \end{itemize}
    
\end{lemma}

\begin{proof}
{\it a)}  First, recall that the sign of $\mathcal{L}_h(t,x,y)$ is determined by the signs of $a_i(x,y), i=0,1,2$, given in \eqref{eq:ai}.   Let $x>0$. If $y>\max\{0, h(x)\}$, then, $a_0(x,y)>0$.   Furthermore,  $a_2(x,y)>0$, since  $L\alpha>1$ and $\beta K<1$. 
 
It therefore suffices to show that $a_1(x,y)>0$. 
If $x\geq K$, then $a_1(x,y)>0$. If 
$0<x<K$, then since  $y>h(x)=\frac{1}{\alpha}-\frac{x}{\alpha K}$, (see \eqref{eq:NCx}) and  using that \eqref{eq:Daiy} is linear in $y$ with positive coefficient, we have 
\begin{align*}
\frac{\partial a_1(x,y)}{\partial y}\geq \frac{\partial a_1(x,y)}{\partial y}\bigg|_{y=h(x)}&=
-K(r \alpha L+s(1-\alpha L))   +  (r \alpha L +s)x   +2r \alpha^2KL h(x)\\
% &=s (x-K) + \alpha L ((x-K) r + K s)+2r\alpha^2KL h(x) \\
&= sx+Ks(\alpha L-1)+\alpha rL(K-x)>0.
\end{align*}
This implies that $a_1(x,y)$ is increasing in $y$. Thus, for $0<x<K$ and $y>h(x)$,  
\begin{equation}\label{eq:sign_a1}
\begin{split}
a_1(x,y)&>a_1(x,h(x)) =\frac{s (K - x) (
K (\alpha L - 1) + x (1
-\alpha \beta L  K  
) )}{\alpha K}\\
&
\stackrel{\beta K<1}{>}
\frac{s (K - x) (
K (\alpha L - 1) + x (1
-\alpha  L 
) )}{\alpha K}=\frac{s (K - x)^2 
 (\alpha L - 1)}{\alpha K}> 0,
\end{split}
 \end{equation}
where the last inequality holds because $\alpha L>1$.

\noindent
 {\it b)} The sign of $\mathcal{L}_k(t,x,y)$ is determined by the signs of $b_i(x,y), i=0,1,2$, defined  in \refeq{eq:bi}.

Now, since in this case, $0<y<k(x)$,   $b_0(x,y)<0$.  Also, since $\beta K<1$ and $\alpha L>1$, it follows that $b_2(x,y)<0$. Thus, to complete the claim, it suffices to show that $b_1(x,y)<0$ for $0<y<k(x)$. Note that if  $0<y<k(x)$, then $y<k(0)=L$, because $k(x)$ is decreasing in $x$.  We can therefore focus on  $0<y<L$ and $y<k(x)$.

By  \eqref{eq:partials_bi},  for any fixed $x$, $b_1(x,y)$  defined in \eqref{eq:bi} is a convex function of $y$, and hence the maximum must occur  on the  boundary of the region 
$\{(x,y)\in [0,\infty)^2 : y \leq k(x)\}$.

\noindent
(i) On the line segment $x=0$ and $0 \leq y\leq L$,
$b_1(0,y)=\alpha ryL(y-L)\leq 0$.

\noindent
(ii) On the line segment $y=0$ and  $0\leq x\leq \frac{1}{\beta}$, $b_1(x,0)= L^2x(r( \beta K - 1)+ s \beta K (\beta x - 1))\leq 0$. 

\noindent
(iii) On the line segment  $y=k(x)$ for $0\leq x\leq \frac{1}{\beta}$,
$$b_1(x,k(x))= r x \beta L^2(K(1+\alpha L(\beta x-1))-x)$$
with
$$\frac{d b_1(x,k(x))}{dx} = 2rL^2  \left(-1+\beta\left( \alpha \beta KLx-\frac{3}{2}\alpha KL +
\frac{1}{2}K+x\right) \right)$$ and
$$ \frac{d^2 b_1(x,k(x))}{dx^2}=2 r  \beta L^2  ( \alpha \beta K L+1)>0.$$

Therefore, $b_1(x,k(x))$ is convex and the maximum is attained at an end point, shown to be non-positive at the end points  in (i) and (ii) above. Note also that  $b_1(x,y)=0$ if and only if $(x,y)\in \left\{(0,L), (0,0), (\frac{1}{\beta},0)\right\}$. The result follows.

\end{proof}

As in the classical representation of a phase plane, we subdivide the positive cone into regions separated by nullclines. In this case, we therefore obtain the representation in Fig.~\ref{fig:case0} with the middle region $\Omega_2$ and the two other regions located above, $\Omega_3$, and below, $\Omega_1$. More precisely, 
\begin{align*}
\Omega_1&=\{(x,y)\in (0,\infty)^2 \colon \, y<h(x)\},\\
    \Omega_2 &:= \{(x,y)\in (0,\infty)^2\colon  \max\{0,h(x)\}\leq y\leq \max\{0,k(x)\},\\ 
    \Omega_3 &:= \{(x,y)\in (0,\infty)^2\colon \max\{0, k(x)\}<y\}.
\end{align*}

The classical representation is then extended using the information provided in Lemma~\ref{lem:signLhLkCase1} that gives the signs of the Root-operators associated with the nullclines.

\begin{figure}[h!]
    \centering
    \includegraphics[width=0.5\linewidth]{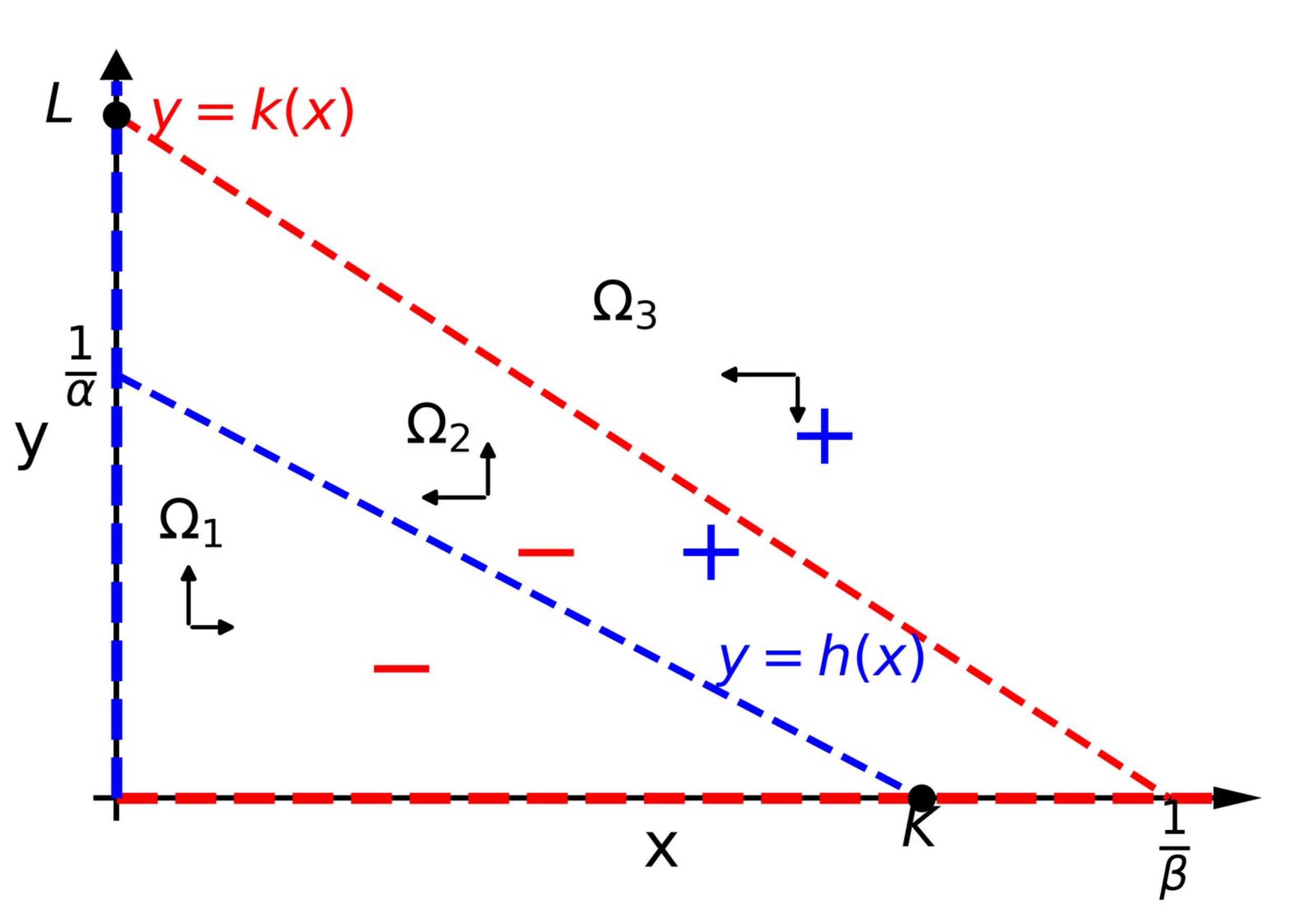}
    \caption{Regions considered with generic nullclines $y=h(x)$ (blue dashed line) and $y=k(x)$  (red dashed line), in the case where $\alpha L>1$ and $\beta K<1$. The colored signs relate to the sign of the corresponding Root-operators $\mathcal{L}_h$ (in blue) and $\mathcal{L}_k$ (in red)  in the regions,  $\Omega_i$, $i\in \{1,2,3\}$. The signs are based on the results in Lemma~\ref{lem:signLhLkCase2}.
   The arrows represent the sign of $x^\Delta$ (horizontal) and $y^\Delta$ (vertical). Each region exhibits component-wise monotonicity.   }
    % change to red and blue color
    \label{fig:case0}
\end{figure}

\begin{theorem}\label{thm:y_only_gas}
     Consider \eqref{eq:CompT} with $\alpha L >1$ and $\beta K<1$. Then, the region   $\Omega_2$  is positively invariant. Moreover, all solutions with positive initial conditions converge to the equilibrium $E^*_L=(0, L)$. 
\end{theorem}

\begin{proof}
To show that $\Omega_2$ is positively invariant, let $(x(T),y(T))\in \Omega_2$ for some $T\in \T$. First, we  discuss the boundary points  of $\Omega_2$: 
a) \, $\left\{(x,y)\colon y=0,\, x\in \left[K, \frac{1}{\beta}\right]\right\}$, 
   b)\,  $\{(x,y)\colon y=h(x), \, x\in (0, K)\}$, 
c)\, $\left\{(x,y)\colon x=0, \, y\in \left[\frac{1}{\alpha}, L\right]\right\}$, 
d)\, $\left\{(x,y)\colon y=k(x), \, x\in \left(0,  \frac{1}{\beta}\right)\right\}$.

a) Assume first that there exists $T$ such that $y(T)=\max\{0,h(x(T))\}=0$ and $K\leq x(T)\leq \frac{1}{\beta}$. Clearly, if $x(T)=K$, then the solution remains at this equilibrium, confirming that the solution remains in $\Omega_2$. If $K<x(T)\leq \frac{1}{\beta}$,  by \eqref{eq:ai}, $a_0(x(T),0)\geq 0$, $a_1(x(T),0)>0$, and $a_2(x(T),0)>0$ so that $\mathcal{L}_h(T,x(T),0)\geq 0$ with equality iff  $a_0(x(T),0)=0$ and $\mu(T)=0$. Furthermore, by \eqref{eq:LhLk}, $\mathcal{L}_k(T,x(T),0)=-L+L\beta \frac{x(T)(1+\mu(T) r)}{1+\mu(T) r\frac{x(T)}{K}}=-K\left(\frac{x(T)(1+\mu(T) r)}{1+\mu(T) r\frac{x(T)}{K}}\right)$. Note that by \eqref{eq:xy_sigma}, $\frac{x(T)(1+\mu(T) r)}{1+\mu(T) r\frac{x(T)}{K}}$ is the evaluation of $x^\sigma(T)$ when $y(T)=0$. Since solutions are nonnegative, we know that $x^\sigma(T)\geq 0$ and by \eqref{eq:xy_Delta}, since $x(T)>K$, $x^\Delta(T)<0$ so that $x^\sigma(T)\leq x(T)$. Thus, $k(x^\sigma(T))> k(x(T))\geq 0$ and therefore $\mathcal{L}_k(T,x(T),y(T))<0$. By Theorem~\ref{thm:need}, there exists $\epsilon>0$ such that $(x(t),y(t))$ remains above the line $y=\max\{0,h(x)\}$ and below the line $y=k(x)$ for all $t\in (T,T+\epsilon)_\T$.  Hence, $(x(t),y(t))\in \Omega_2$ for all $t\in (T,T+\epsilon)_\T\neq \emptyset$.

b) We continue the boundary discussion by assuming that there exists $T$ such that $x(T)\in (0, K)$ and $y(T)=h(x(T))$. The latter implies that $x^\sigma(T)=x(T)$. 
By the definition of the Root-operators in Def.~\ref{def:Rootop}, we have 
$$\mathcal{L}_h(T,x(T),y(T))=y^\sigma(T) - h(x^\sigma(T))=y^\sigma(T) - h(x(T))=y(T)+\mu(T)y^\Delta(T)-h(x(T))$$
and since $y^\Delta (T)>0$ for $y=h(x)$ and $x\in (0,K)$, $\mathcal{L}_h(T,x(T),y(T))\geq 0$ with equality iff $\mu(T)=0$. Using Def.~\ref{def:Rootop}, we also have
\begin{align*}
\mathcal{L}_k(T,x(T),y(T))&=y^\sigma(T) - k(x^\sigma(T))=y^\sigma(T) - k(x(T))=y(T)+\mu(T)y^\Delta(T)-k(x(T))\\
&=h(x(T))+\mu(T)y^\Delta(T)-k(x(T))\\
&\stackrel{\eqref{eq:xy_Delta}}{=}h(x(T))-k(x(T))+\mu(T)\frac{sy(T)(k(x(T))-y(T)}{L+s\mu(T)(y(T)+\beta L x(T))}\\
&=h(x(T))-k(x(T))+\mu(T)\frac{sh(x(T))(k(x(T))-h(x(T))}{L+s\mu(T)(y(T)+\beta L x(T))}\\
&=-(k(x(T))-h(x(T)))\left\{1-\mu(T)\frac{sy(T)}{L+s\mu(T)(y(T)+\beta L x(T))}\right\}<0,
\end{align*}
because $k(x(T))>h(x(T))$ and $\frac{\mu(T) sy(T)}{s\mu(T)y(T)+L(1+\mu(T)\beta s  x(T))}<1$. By Theorem~\ref{thm:need}, this implies the existence of $\epsilon>0$ such that $(x(t),y(t))\in \Omega_2$ for $t\in (T,T+\epsilon)_\T\neq \emptyset$. 

c) Assume now that there exits $T$ such that $x(T)=0$ and $\frac{1}{\alpha}\leq y(T)\leq L$. Note that if $y(T)=L$,  then this is an equilibrium and $(x(t),y(t))=(0,L)\in \Omega_2$ for all $t\in [T,\infty)_\T$. Thus, consider $y(T)\in \left[\frac{1}{\alpha}, L\right)$. Since $y(T)\geq \frac{1}{\alpha}$,  by \eqref{eq:ai}, $a_0(0,y(T))\geq 0$ with equality iff $y(T)=\frac{1}{\alpha}$, 
\begin{eqnarray*}
a_1(0,y(T))&=r\alpha^2y^2(T) K L +y(T)(-Ks-\alpha L r K +sK\alpha L)\\
&=r\alpha y(T) K L (\alpha y(T)-1) +yKs \alpha \left(L-\frac{1}{\alpha}\right)>0,
\end{eqnarray*}
because $y(T)\geq \frac{1}{\alpha}$ and $L>\frac{1}{\alpha}$. By the same reasoning, $a_2(0,y(T))>0$. 

By \eqref{eq:bi}, $b_0(0,y(T))<0$ and $b_1(0,y(T))<0$ for $y(T)<L$, and $b_2(0,y(T))=0$. Hence, by applying Theorem~\ref{thm:need} (twice), there exists $\epsilon>0$ such that $(x(t),y(t))\in \Omega_2$ for $t\in (T,T+\epsilon)_\T\neq \emptyset$. 

d) Lastly, let us assume that there exists $T$ such that $x(T) \in  \left(0,\frac{1}{\beta}\right)$ and $y(T)=k(x(T))$. Then $y^\sigma(T)=y(T)=k(x(T))$ and, by Def.~\ref{def:Rootop}, 
$$\mathcal{L}_k(T,x(T),y(T))=y^\sigma(T)-k(x^\sigma(T))=k(x(T))-k(x(T)+\mu(T)x^\Delta(T)).$$
Since $y(T)=k(x(T))>h(x(T))$, $x^\sigma(T)\leq x(T)$ with equality iff $\mu(T)=0$. Given that $y=k(x)$ is a decreasing function, $k(x(T))\geq k(x^\sigma(T))$ so that $\mathcal{L}_k(T,x(T),y(T))\leq 0$ with equality iff $\mu(T)=0$. Similarly, 
\begin{align*}
    \mathcal{L}_h(T,x(T),y(T))&=y^\sigma(T)-h(x^\sigma(T))=y(T)-h(x(T)+\mu(T)x^\Delta (T))\\
    &=k(x(T))-\frac{1}{\alpha}+\frac{1}{\alpha K}x(T)+\frac{1}{\alpha K}\mu(T) x^\Delta (T)\\
    &=k(x(T))-h(x(T))+\frac{1}{\alpha K}\mu(T) x^\Delta (T)\\
    &\stackrel{\eqref{eq:xy_Delta}}{=}k(x(T))-h(x(T))+\frac{1}{\alpha K}\mu(T) \frac{rx(T)\alpha (h(x(T))-y(T))}{1+r\mu(T)\left(\frac{x(T)}{K}+\alpha y(T)\right)}\\
    &=(k(x(T))-h(x(T)))\left\{1- \frac{r\mu(T)\frac{x(T)}{K}}{1+r\mu(T)\left(\frac{x(T)}{K}+\alpha y(T)\right)}\right\}\geq 0
\end{align*}
with equality iff $\mu(T)=0$. Since $\mathcal{L}_k(T,x(T),y(T))\leq 0$ and $\mathcal{L}_h(T,x(T),y(T))\geq 0$, we apply Theorem~\ref{thm:need} (twice), to confirm that there exits $\epsilon>0$ such that $(x(t),y(t))\in \Omega_2$ for $t\in (T,T+\epsilon)_\T\neq \emptyset$. 
This completes the boundary cases of $\Omega_2$. 

 Now assume that there exists $T$ such that $(x(T),y(T))$ is in the interior of $\Omega_2$. By Lemma~\ref{lem:signLhLkCase1}a), $\mathcal{L}_h(T,x(T),y(T))>0$. Therefore, by Theorem~\ref{thm:need}, there exists $\epsilon_1>0$ such that $y(t)>h(x(t))$ for $t\in [T, T+\epsilon_1]_\T$, where $T+\epsilon_1\geq \sigma(T)$. By Lemma~\ref{lem:signLhLkCase1}b), $\mathcal{L}_k(T,x,y)<0$. Therefore,  by Theorem~\ref{thm:need}, there exists $\epsilon_2>0$ such that $y(t)<k(x(t))$ for $t\in [T, T+\epsilon_2]_\T$, where $T+\epsilon_2\geq \sigma(T)$. Choosing $\epsilon=\min\{\epsilon_1, \epsilon_2\}$ confirms that the solution remains in $\Omega_2$. 

Next, we show that all solutions with positive initial conditions converge to $E^*_L$. Since $\Omega_2$ is positively invariant,   and $x^\Delta>0$ and $y^\Delta <0$, whenever $(x,y)\in \Omega_2$, the component-wise monotonicity implies that, by Proposition~\ref{prop:convergence}, that $(x,y)$ converges to an equilibrium in  $\Omega_2$, and $E^*_L$ is the only possibility due to the direction of monotonicity. It therefore suffices to show that for any $(x(0),y(0))\notin \Omega_2$, there exists $T$ such that $(x(T),y(T))\in \Omega_2$.

Suppose there exists $T\in \T$ such that $y(T)=h(x(T))$, then $x^\Delta(T)=0$ and $y^\Delta(T)>0$.
 If $T$ is right-scattered, then $y^\Delta(T)>0$ implies that  $y^\sigma(T)>y(T)=h(x(T))=h(x^\sigma(T))$. Furthermore, since $y(T)=h(x(T))<k(x(T))$, by Lemma~\ref{lem:signLhLkCase1}b), $\mathcal{L}_k((T),x(T),y(T))<0$, so that $y^\sigma(T)<k(x^\sigma(T))$, confirming that $(x^\sigma(T), y^\sigma(T))\in \Omega_2$. 
Instead, if $T$ is right-dense and $x, y$ are rd-continuous, $y^\Delta(T)>0$ and $x^\Delta(T)=0$ imply that there exists $\epsilon_2>0$ such that $y(t)>y(T)$ and $y(t)>h(x(t))$ for $t\in (T, T+\epsilon_2]_\T$. Since $y(T)=h(x(T))<k(x(T))$, by Lemma~\ref{lem:signLhLkCase1}b), $\mathcal{L}_k(T,x(T),y(T))<0$ which implies by Theorem~\ref{thm:need},
that there exists $\epsilon_3>0$ such that $y(t)<k(x(t))$ for $t\in (T, T+\epsilon_3]_\T$, where $T+\epsilon_3\geq \sigma(T)$. Thus,  there exists $ \overline{\epsilon}:=\min\{\epsilon_3, \epsilon_2\}$ such that $(x(t),y(t))\in \Omega_2$ for $t\in [T, T+\overline{\epsilon}]_\T$.

If  there exists $T\in \T$ such that $y(T)=k(x(T))$, then $y^\Delta(T)=0$ and $x^\Delta(T)<0$. If $T$ is right-scattered, then  $y^\sigma(T)=y(T)=k(x(T))<k(x^\sigma(T))$ since $y=k(x)$ is a decreasing function. Furthermore, by Lemma~\ref{lem:signLhLkCase1}a), $\mathcal{L}_h((T),x(T),y(T))>0$, so that $y^\sigma(T)>h(x^\sigma(T))$, confirming that $(x^\sigma(T), y^\sigma(T))\in \Omega_2$. 
Instead, if $T$ is right-dense, then there exists $\epsilon_1>0$ such that $x(t)<x(T)$ and, since $y, x, h$ are rd-continuous, $y(t)<k(x(t))$ for $t\in (T, T+\epsilon_1]_\T\neq \emptyset$. Since, by Lemma~\ref{lem:signLhLkCase1}a), $\mathcal{L}_h(T,x(T),y(T))>0$, 
by Theorem~\ref{thm:need}, there exists $\epsilon_2>0$ such that $y(t)>h(x(t))$ for $t\in (T, T+\epsilon_2]_\T\neq \emptyset$. Thus, there exists $\overline{\epsilon}=\min\{\epsilon_1, \epsilon_2, \epsilon_3\}$ such that $(x(t),y(t))\in \Omega_2$ for $t\in (T, T+\overline{\epsilon}]_\T\neq \emptyset$.

Recall that $\Omega_1=\{(x,y)\in (0,\infty)^2 \colon \, y<h(x)\}$, see also Fig.~\ref{fig:case0}. 
Let $(x(t_0),y(t_0))\in \Omega_1$.  Either i) $(x(t),y(t))\in \overline{\Omega}_1$, where $\overline{\Omega}_1$ is the closure of $\Omega_1$, for all  $t\in (t_0,\infty)_\T$ or ii) there exists $T>0$ such that $(x(T),y(T))\notin \Omega_1$. 
i) If $(x(t),y(t))\in \overline{\Omega}_1$ for all $t\in (t_0,\infty)_\T$, then, by Proposition~\ref{prop:convergence}, $(x,y)$ converges to one of the  equilibria in $\overline{\Omega}_1$, either $(0,0)$ or $(K,0)$. However, since $x^\Delta, y^\Delta>0$ for $(x,y)\in \Omega_1$, neither of these equilibria can be approached. Thus, there must exist (a first) $T>t_0$ such that $(x(T),y(T))\notin \Omega_1$, i.e., $y(T)\geq h(x(T))$. If $y(T)=h(x(T))$, then we proved  above  that $(x,y)\in \Omega_2$ for all   $t\in (T,\infty)_\T$. If $y(T)>h(x(T))$, then by Lemma~\ref{lem:signLhLkCase1}b), $y(T)<k(x(T))$, so that $(x(T),y(T))\in \Omega_2$ and, by the above argument, $(x(t),y(t))\in \Omega_2$ for all $t\in (T,\infty)_\T$  and the solution converges to $E^*_L$.

Finally, recall that 
$\Omega_3:=\{(x,y)\in (0,\infty)^2\, \colon\, y>k(x)\}$. 
Let $(x(t_0),y(t_0))\in \Omega_3$.  Either i) $(x(t),y(t))\in \overline{\Omega}_3$ for all $t\in (t_0,\infty)_\T$ or ii) there exists $T>0$ such that $(x(T),y(T))\notin \Omega_3$. If $(x(t),y(t))\in \overline{\Omega}_3$ for all $t\in (t_0,\infty)_\T$, then, due to the component-wise monotonicity in $\Omega_3$, by Proposition~\ref{prop:convergence}, $(x,y)$ converges to an equilibrium in  $\overline{\Omega}_3$. The only equilibrium in  $\overline{\Omega}_3$ is $E^*_L$, as we  claim. If instead there exists (a first) $T>t_0$ such that $(x(T),y(T))\notin \Omega_3$, then $y(T)\leq k(x(T))$. In the case of equality, the above discussion implies that $(x(t),y(t))\in \Omega_2$ for all $t\in (T,\infty)_\T$. If $y(T)>k(x(T))$, then since $(x(t),y(t))\in \Omega_3$ for all $t\in (t_0,T)_\T$, $y(t)>h(x(t))$ for all $t\in (t_0,T)_\T$. In particular, for $t_1$ such that $\sigma(t_1)=T$. By  Lemma~\ref{lem:signLhLkCase1}a), $y^\sigma(t_1)=y(T)>h(x(T))$. This implies that $(x^\sigma(T), y^\sigma(T))\in \Omega_2$, so that by the above argument, $(x(t),y(t))\in \Omega_2$ for all $t\in [T,\infty)_\T$ and the solution converges to $E^*_L$. 
\end{proof}

\begin{example}
An example using the quantum time scale $\T=2^{\mathbb{N}_0}$ is provided in Fig.~\ref{fig:Example1_C1_g} that illustrates  the time dependence of the Root-sets.  As $t$ increases in the time scale, \sout{$\T=2^{\mathbb{N}_0}$} from $t=2^0=1$ to $t=2^3$, the corresponding Root-curves move further and further away  from the nullcline also demonstrating  that iterations can enter the positively invariant region $\Omega_2$ within one iteration, at some time point $t>0$,  despite being far away.   Fig.~\ref{fig:Example1_C1_g} also highlights a property indicated in Fig.~\ref{fig:case0}, that the Root-curves of the nullcline $y=k(x)$ (red dashed line) remain above this nullcline, so that any region in $[0,\infty)^2$ below $y=k(x)$ has a negative Root-operator value. Similarly, since all Root-sets remain below the (blue dashed) nullcline $y=h(x)$, the corresponding Root-operator does not change sign in any region above the line $y=h(x)$.

\begin{figure}[h!]
    \centering
    \includegraphics[width=0.85\linewidth]{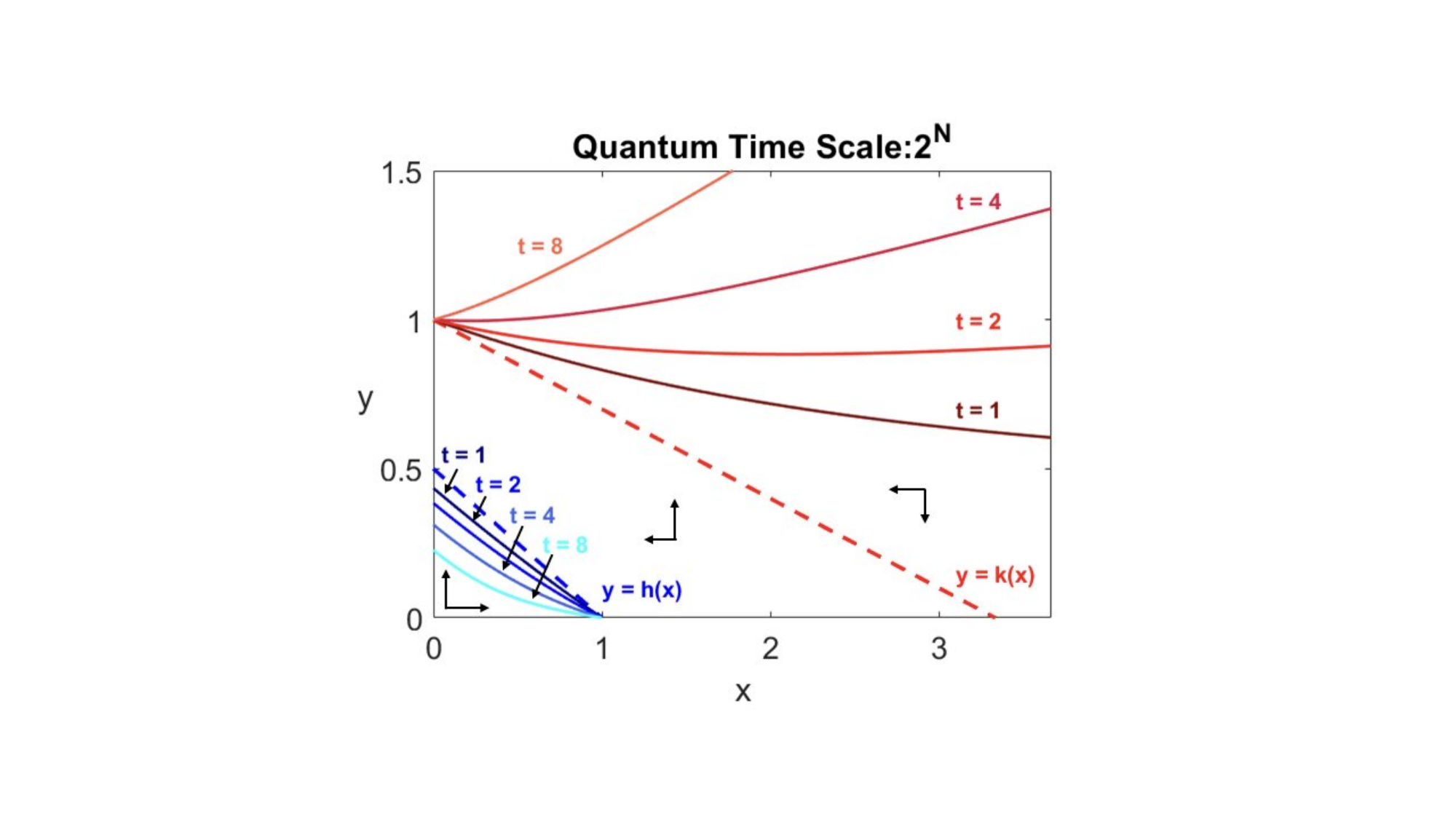}
    \caption{Dynamic  Phase Plane of \eqref{eq:CompT} for  the quantum time scale $2^{\mathbb{N}_0}$ including the (time-dependent) Root-curves for specific $t$-values. The  (dashed)  nontrivial $y$-nullcline and its (solid) Root-set $\mathcal{R}_k$ are displayed in shades of red. The (dashed)  nontrivial $x$-nullcline and its (solid) Root-set $\mathcal{R}_h$ are in shades of blue. The parameter values are: $\alpha  = 2$, $\beta=0.3$, $r=0.5$, $s=0.3$ and $L=K=1$.}
    \label{fig:Example1_C1_g}
\end{figure}

\end{example}

\begin{example}
 The panels of Fig.~\ref{fig:Example1_C1} contain the first three iterations of the solution for the initial condition $(x(1),y(1))=(2,1)$ at time values $t_0=2^0=1$, $\sigma(t_0)=2^1=2$, and $\sigma^2(t_0)=2^2=4$. At each  of these time values, the  sign of the Root-operator for that particular time point determines whether the next iterate (i.e., the solution at time $\sigma(t)$) is above or below the corresponding nullcline. The Root-curves that are illustrated and the iterate of the sample orbit that is highlighted by a purple star share the same time value. Thus, the positioning of the star relative to the Root-curves and associated nullclines determines the positioning of the next iterate. For example,  consider the initial condition when time $t_0=2^0=1$ at $(x(2^0),y(2^0))=(2,1)$,  (see the purple  star in Fig.~\ref{fig:Example1_C1}a). To determine the position of the next iterate $(x(2^1),y(2^1))$, we include the Root-operator that changes signs at   its associated Root-set.  Since $(x(2^0),y(2^0))$ is in a region where the Root-operators associated with $y=k(x)$ (red dashed line) and $y=h(x)$ (blue dashed line) are positive are both positive, (see the red and the blue ``+" signs), the next iterate $(x(2^1),y(2^1))$ remains above these red and blue nullclines, respectively. However, at time $t=2^1$,  (see Fig.~\ref{fig:Example1_C1}b), the Root-sets change and now $(x(2^1),y(2^1))$ is in a region of a red ``--" sign and a blue ``+" sign, indicating that the next iterate $(x(2^2),y(2^2))$ lies below the red nullcline and above the blue nullcline, see purples star in Fig.~\ref{fig:Example1_C1}c). Thus, $(x(2^2),y(2^2))$ enters the positively invariant region and  converges, by Theorem~\ref{thm:y_only_gas}, to $(0,L)$ as $t\to \infty$.

\begin{figure}[h!]
    \centering
\includegraphics[scale=0.25]{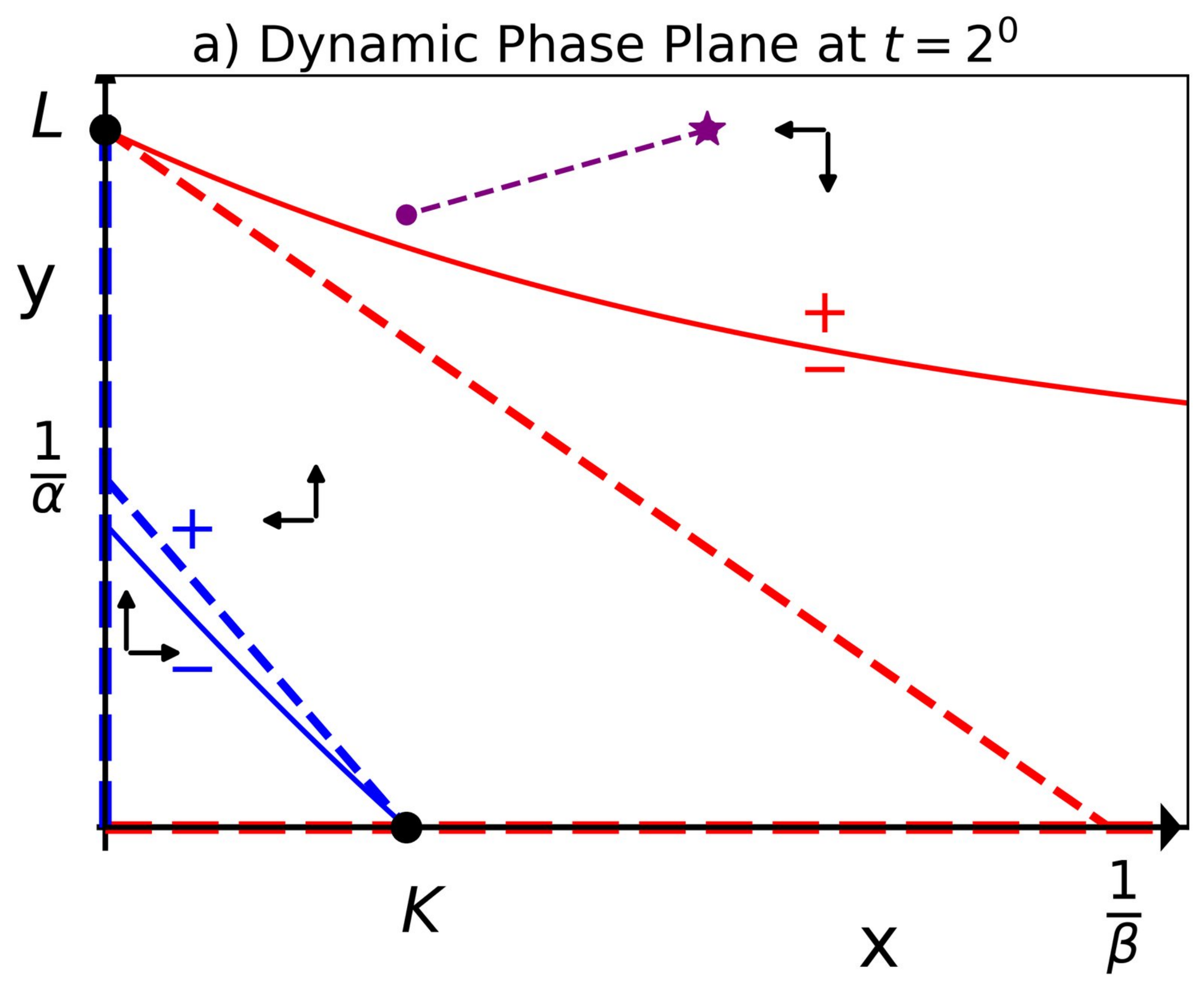}
\includegraphics[scale=0.25]{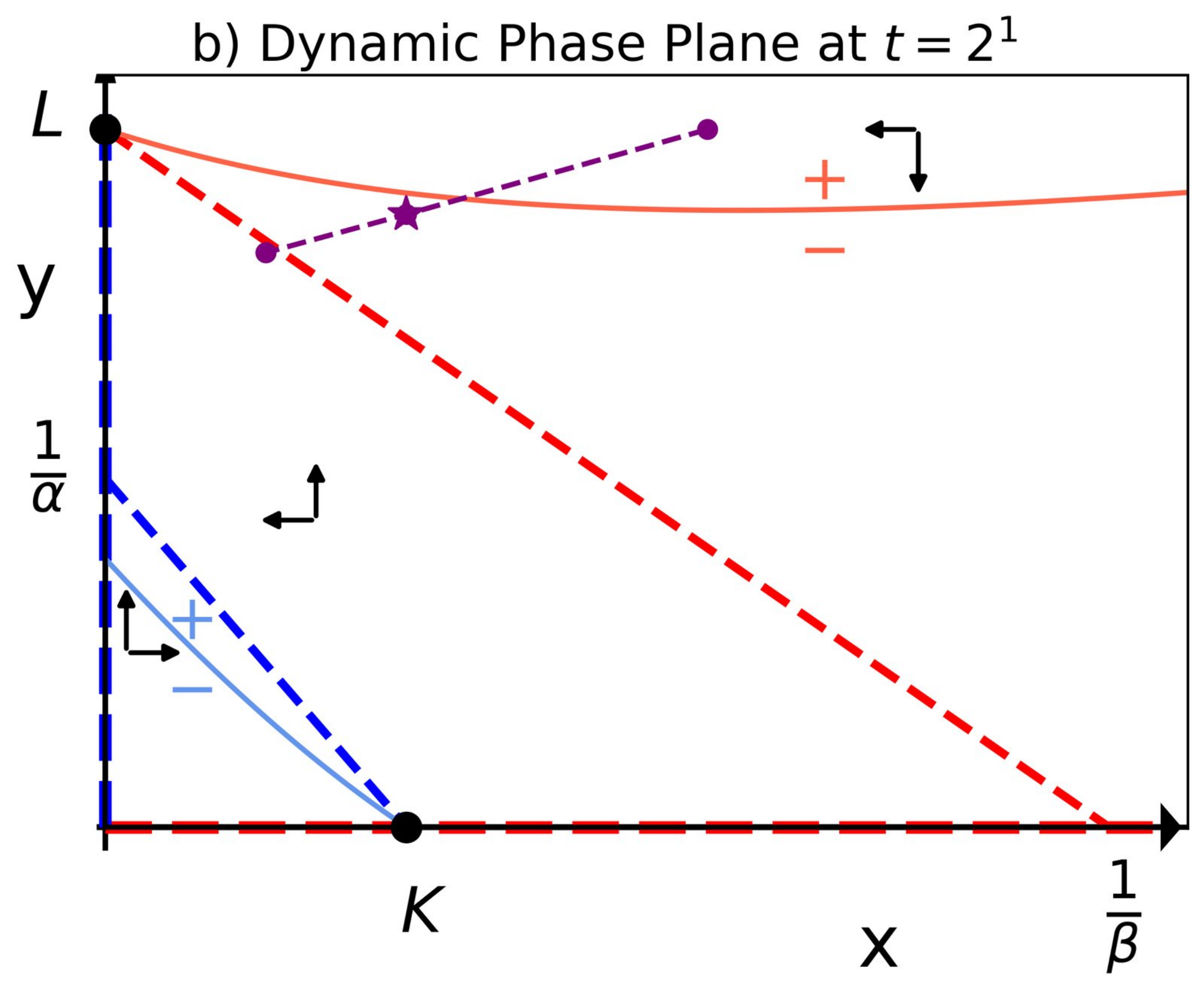}
\includegraphics[scale=0.25]{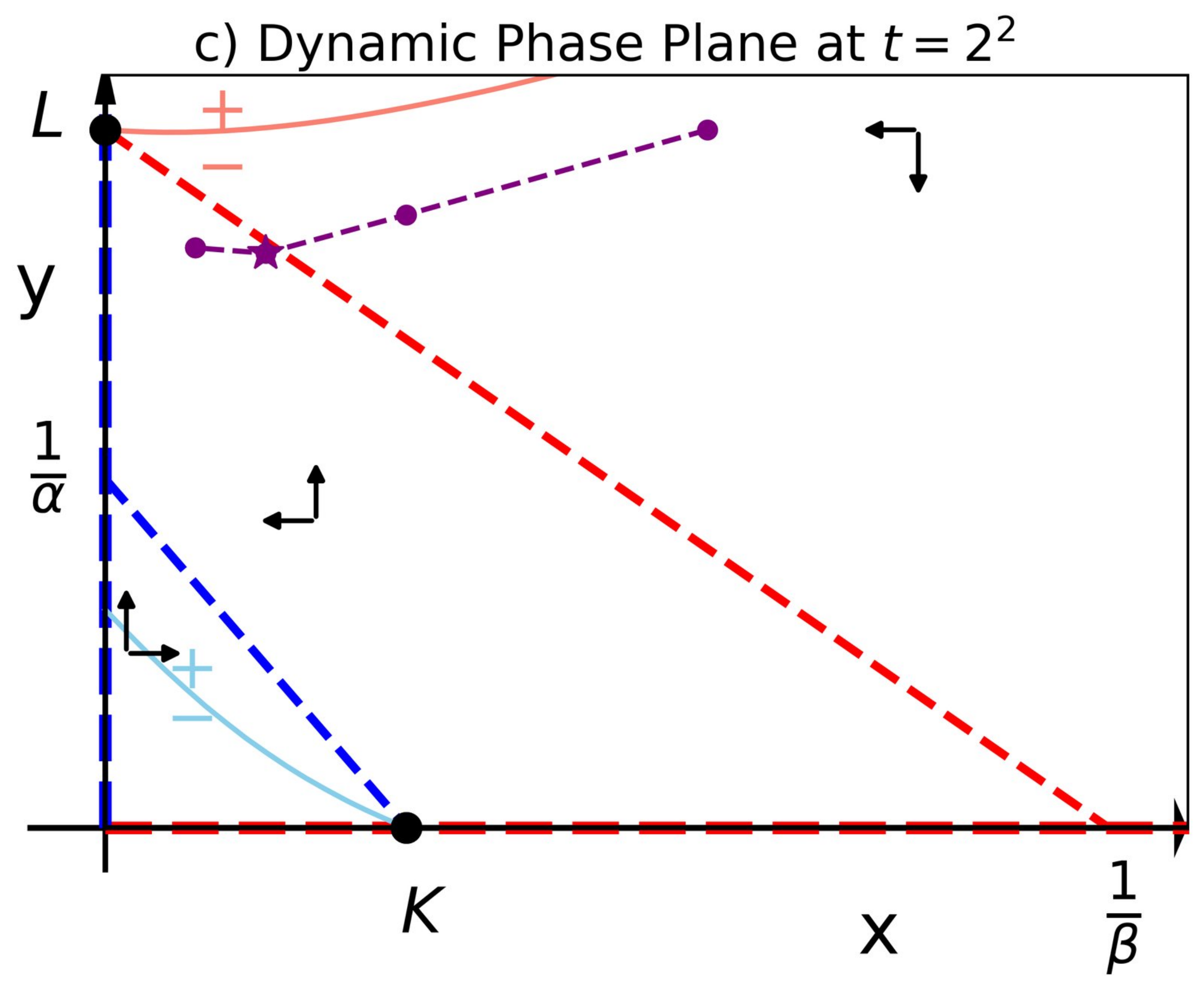}
    \caption{ Dynamic Phase Plane for $t\in\{2^0, 2^1, 2^2\}$ and the first three iterations of a (purple) solution starting at $(x(2^0),y(2^0))=(2,1)$ for parameter values as in Fig.~\ref{fig:Example1_C1_g}.  The Root-sets,  $\mathcal{R}_h$ and $\mathcal{R}_k$, are solid in the color associated with the same colored (dashed) nullcline;  red for $y=k(x)$ and blue for $y=h(x)$. Only the Root-set (at time $T$) relevant to determine the next iterate ($T+1$) based on the positioning of the current iterate $(x(T),y(T))$, indicated by a star, are illustrated. Since $(x(2^0),y(2^0))$ is in a region with a red "+", the next iterate $(x(2^1),y(2^1))$ remains above the red nullcline. However, since $(x(2^1),y(2^1))$ is below the Root-curve associated with the red nullcline, at time $t=2^1$, the next iterate $(x(2^2),y(2^2))$ is mapped below the red nullcline. Since the Root-operator associated with $y=h(x)$ is positive at the point $(x(2^1),y(2^1))$, the next iterate remains above the blue nullcline. Hence, $(x(2^2),y(2^2))$ is the first iterate of this orbit being mapped into the positively invariant region  bounded by the two nontrivial nullclines in the nonnegative cone. 
    }
    \label{fig:Example1_C1}
\end{figure}

\end{example}

\begin{remark}
    If $\T=\R$, then Fig.~\ref{fig:case0} already describes  the dynamic phase plane because for right-dense points $t\in \T$, the Root-curves are identical to the nullclines. Hence, the Root-operator associated with $y=k(x)$ is positive in regions $\Omega_3$ and negative in regions $\Omega_1\cup \Omega_2$. Similarly, the Root-operator associated with $y=h(x)$ is negative in $\Omega_1$ and positive in regions $\Omega_2\cup \Omega_3$. Thus, by Theorem~\ref{thm:need}, if $(x(T),y(T))\in \Omega_3$, there exists $\epsilon>0$ such that for $t\in [T,T+\epsilon)$, $(x(t),y(t))\in \Omega_3$. This however does not imply that the solution cannot cross the associated nullcline, as $\epsilon$ depends on $(x(T),y(T))$, but is simply a consequence of the density of real numbers and the continuity of solutions.  
\end{remark}

\subsubsection{Global Dynamics: An interior equilibrium exists}

 We now assume that  $y=h(x)$ and $y=k(x)$   intersect in $\R^+$ $=(0,\infty)$, in which case there exists a unique coexistence equilibrium $E^*=(x^*,y^*)$ given in \eqref{eq:equilbria}.  This occurs when either
\begin{enumerate}
 \item[(III)] $\alpha L<1$ and $\beta K<1$, or 
    \item[(IV)] $\alpha L>1$ and $\beta K>1$.    
\end{enumerate}
Here, strict inequality is necessary for all of the inequalities  in order to guarantee an interior equilibrium. 

 The following result is used together with the Dynamic Phase Plane to determine the global dynamics in each of the cases (III) and (IV).

\begin{proposition}\label{Prop:boxes}
   Define $$\mathcal{B}_0:=\{(x,y)\in [0,\infty)^2 \, \colon \, 0\leq x\leq x^*, \,   \,0 \leq y\leq y^*\}$$ and $$\mathcal{B}_1:=\{(x,y)\in (0,\infty)^2 \, \colon \, x\geq x^*,  \,    \, y\geq y^*\}.$$ 
   If $(x,y)\in \mathcal{B}_0$,  then $(x^\sigma, y^\sigma) \in \mathcal{B}_1$ if and only if $(x,y)=(x^*,y^*)$. 
\end{proposition}

\begin{proof}
 The result clearly holds for any right-dense points, since  $x^\sigma=x$ and $y^\sigma=y$. Thus, we now consider right-scattered points $t\in \T$ and  proceed using proof by contradiction.  Suppose that  $(x (T),y(T))\in \mathcal{B}_0$  and 
 $(x^\sigma(T), y^\sigma(T))\in \mathcal{B}_1$ for some right-scattered $T\in \T$. This implies that $x^\sigma(T)\geq x^*$ and $y^\sigma(T)\geq y^*$, but $(x(T),y (T))\neq(x^*,y^*)$. 
Then,  define $m(x(T),y(T)):=1+\mu (T) r\left(\frac{x(T)}{K}+\alpha y(T)\right)$.  It follows that  
\begin{align*}
0 & \leq x^\sigma(T) - x^*\stackrel{\eqref{eq:xy_sigma}}{=}\frac{x(T)(1+r\mu(T))}{m(x(T),y(T))}-\frac{x^*(1+r\mu(T))}{m(x^*,y^*)}\\
&=(1+\mu(T)r)\frac{(x(T)-x^*)+\alpha r\mu(T) (x(T)y^*-x^*y(T))}{m(x(T),y(T))m(x^*,y^*)}.
\end{align*}
Since $m(x(T),y(T))>0$, this implies that $(x(T)-x^*)+\alpha r \mu(T) (x(T)y^*-x^*y(T))\geq 0$, or, equivalently, 
\begin{equation}\label{xbigxstar}
  x(T)y^*\geq x^*y(T) +\frac{x^*-x(T)}{\alpha r \mu(T)}.
\end{equation}
Similarly, define
$n(x(T),y(T)):=1+s\mu(T)\left(\frac{y(T)}{L}+\beta x(T)\right)$. Then, 
\begin{align*}
    0&\leq y^\sigma (T) - y^* \stackrel{\eqref{eq:xy_sigma}} {=}\frac{y(T)(1+s\mu(T))}{n(x(T),y(T))}-\frac{y^*(1+s\mu(T))}{n(x^*,y^*)}\\
    &=(1+s\mu (T))\frac{(y(T)-y^*)+\beta s \mu(T) (y(T)x^*-y^*x(T))}{n(x(T),y(T))n(x^*,y^*)}.
\end{align*}
Since $n(x(T),y(T))>0$, this implies $(y(T)-y^*)+\beta s\mu(T) (y(T)x^*-y^*x(T))\geq 0$, or equivalently, 
\begin{equation}\label{ybigystar}
x^*y(T)\geq y^*x(T)+\frac{y^*-y(T)}{\beta s \mu(T)}.
\end{equation}
Taking \eqref{xbigxstar} and \eqref{ybigystar} together yields
$$
x(T)y^*\geq  x^*y(T) +\frac{x^*-x(T)}{\alpha r \mu(T)}\geq y^*x(T)+\frac{y^*-y(T)}{\beta s \mu(T)}+ \frac{x^*-x(T)}{\alpha r \mu(T) }\geq x(T)y^*,
$$
and since  $(x (T),y(T))\in\mathcal{B}_0$, this results in a contradiction unless $x(T)=x^*$ and $y(T)=y^*$.

\end{proof}

\vspace{3mm}

\textbf{ (III) \, $\alpha L<1$ and $\beta K<1$}

\vspace{3mm}

\noindent In this case,  $k(x)<h(x)$ for all $x\in (0,x^*)$ and $k(x)>h(x)$ for all $x \in (x^*,\beta^{-1})$.  Recall that $h(x)>0$ for $x\in (0,K)$ but $h(x)<0$ for all $x>K$ and  $k(x)>0$ for $x\in (0,\beta^{-1})$, but $k(x)<0$, for all $x>\beta^{-1}$.

 \begin{lemma}\label{lem:signLhLkCase3}
     Consider \eqref{eq:CompT} with $\alpha L <1$ and $\beta K<1$.  Then, the following hold:
     \begin{itemize}
         \item[a)]  $\mathcal{L}_h(t,x,y)<0$ for all $(x,y)\neq (x^*,y^*)$,  $0<x\leq x^*$, $y^*\leq  y<h(x)$, and $t\in \T$. 
\item[b)] $\mathcal{L}_h(t,x,y)>0$ for all $(x,y)\neq (x^*,y^*)$,   $x\geq x^*$, $\max\{0,h(x)\}<y\leq y^*$
and $t\in \T$.   
\item[c)]   $\mathcal{L}_k(t,x,y)>0$ for all \, $0<x\leq x^*$, $y>k(x)$, and $t\in \T$.  
\item[d)]  $\mathcal{L}_k(t,x,y)<0$ for all  $x^*\leq x<\beta^{-1}$, $0<y<k(x)$, and $t\in \T$.  
    \end{itemize}
\end{lemma}

 \begin{proof}  
%Let $\alpha L<1$ and $\beta K<1$.
{\it a)} Assume $0<x\leq x^*$, $y^*\leq  y<h(x)$, and $(x,y)\neq (x^*,y^*)$.  To show that  $\mathcal{L}_h(t,x,y)<0$, it suffices to show that  $a_i, i=0,1,2$, given in \eqref{eq:ai}, are nonpositive and not all zero. 

Since $y<h(x)$,  $a_0(x,y)<0$ and  
since $\alpha L<1$ and $\beta K<1$,  $a_2(x,y)<0$. By \eqref{eq:partials_ai}, for fixed $x$, $a_1(x,y)$ is convex in $y$ and for  fixed $y$, $a_1(x,y)$ is convex in $x$. Therefore,   the maximum of  $a_1(x,y)$ occurs  on the boundary of the  region under consideration.  Therefore, to show $a_1(x,y)\leq 0$, we need only consider points in the sets:

 \begin{align*}
    \Upsilon_1&:=\{(x,y) \, \colon \, y=y^*, \,  0\leq x\leq x^*\}, \\
 \Upsilon_2&:=\{(x,y)\, \colon \,0\leq x\leq x^*, \,  y=h(x)\}, \quad \mbox{and}\\
 \Upsilon_3&:=\{(x,y)\, \colon \, x=0, \, y^* \leq y \leq \alpha^{-1}. \}
 \end{align*}
For $(x,y)\in \Upsilon_1$, 
since $a_1(x,y^*)$ is convex in $x$, we have $a_1(x,y^*)\leq \max\{a_1(0,y^*), a_1(x^*,y^*)\}$. Since 
$a_1(x^*,y^*)=0$ and 
$$a_1(0,y^*)=K y^* (-s(1-\alpha L) -r \alpha L (1- \alpha y^*))<0$$
since $y^*<L$ and $\alpha L<1$. Thus, $a_1(x,y)\leq 0$ for $(x,y)\in \Upsilon_1$. 

%\end{align*}

For $(x,y)\in \Upsilon_2$, 
$$
a_1(x,h(x))=\frac{-s (K - x) (K - x + \alpha K L (\beta x-1))}{\alpha K}.$$
This is negative as long as $K-x+\alpha K L(\beta x-1)>0$. 
Since 
$$\frac{d}{dx}( K - x + \alpha K L (\beta x-1))=\alpha\beta K L-1<0$$
and substituting $x=x^*$, we have 
$K-x^*-\alpha K L(\beta\,  x^* -1)=0$. 
Thus, $a_1(x,h(x))\leq 0$ for $(x,y)\in \Upsilon_2$, with equality if and only if $(x,y)=(x^*,y^*)$.

\noindent
For $(x,y)\in \Upsilon_3$, 
$$ a_1(0,y)= y K(\alpha^2rLy+s(\alpha L-1)-\alpha r L)<0$$   
since $(\alpha^2rLy+s(\alpha L-1)-\alpha r L)$ is  an increasing function of $y$ 
  and substituting $y=\alpha^{-1}$,
  $$(\alpha^2rLy+s(\alpha L-1)-\alpha r L)= s(\alpha L-1)<0.$$
Thus, also for $(x,y)\in \Upsilon_3$, $a_1(x,y)\leq 0$, completing the claim for a).

{\it b)} Assume   $x\geq x^*$, $\max\{0,h(x)\}<y\leq y^*$, and $(x,y)\neq (x^*,y^*)$. To show that   $\mathcal{L}_h(t,x,y)>0$, it suffices to show that  $a_i(x,y), i=0,1,2$, given in \eqref{eq:ai}, are nonnegative and not all zero. Since $y>h(x)$,  $a_0(x,y)>0$. 
Since $yK (\alpha L-1)+xL(1-\beta K)$ is strictly decreasing in $y$ for $\alpha L<1$ and strictly increasing in $x$ for $\beta L<1$, 
$$yK (\alpha L-1)+xL(1-\beta K)\leq y^*K (\alpha L-1)+x^*L(1-\beta K)=0.$$ Thus, by the definition of $a_2(x,y)$ in \eqref{eq:ai}, $a_2(x,y)\leq 0$ for $x\geq x^*$ and $\max\{0,h(x)\}<y\leq y^*$ with equality only if $(x,y)=(x^*,y^*)$.

It is left to show that $a_1(x,y)\geq 0$. If $x\geq K$,  by \eqref{eq:ai},  $a_1(x,y)>0$.  
If $x^*\leq x<K$ and $y>\max\{0, h(x)\}$, we first show that $a_1(x,y)$ is  strictly increasing as a function of $y$, i.e., $\frac{\partial a_1(x,y)}{\partial y}>0$. We  note that, by \eqref{eq:Daiy}, $\frac{\partial a_1}{\partial y}$ is strictly  increasing in $y$, i.e., $\frac{\partial^2 a_1(x,y)}{\partial y^2}>0$. Since $y>\max\{0, h(x)\}=h(x)$ for $x<K$, 
\begin{align*}
\frac{\partial a_1(x,y)}{\partial y}\geq \frac{\partial a_1(x,y)}{\partial y}\Bigg|_{y=h(x)}&=
-K(r \alpha L+s(1-\alpha L))   +  (r \alpha L +s)x   +2r \alpha^2KL h(x), \\
% &=s (x-K) + \alpha L ((x-K) r + K s)+2r\alpha^2KL h(x) \\
%&= sx+Ks(\alpha L-1)+\alpha rL(K-x)
&= \alpha LK (r+s)-sK+x(s-r\alpha L).
\end{align*}
If $s\leq r \alpha L$, since $x<K$, we further simply 
$$\frac{\partial a_1(x,y)}{\partial y}\geq\frac{\partial a_1(x,y)}{\partial y}\Bigg|_{y=h(x)}\geq  \alpha L K(r+s)-sK+K(s-r\alpha L)=s\alpha KL>0.$$
If $s>r\alpha L$,  since $x\geq x^*$, we further simplify
\begin{align*}
     \frac{\partial a_1(x,y)}{\partial y}&\geq \frac{\partial a_1(x,y)}{\partial y}\Bigg|_{y=h(x)}\geq  \alpha L K (r+s)-sK+x^*(s-r\alpha L)\\
    & 
    = s\left( x^*-K(1- \alpha L )\right) +r\alpha L (K-x^*) \\
    & 
    = s\left( \frac{K(1-\alpha L)}{1-\alpha \beta K L}-K(1- \alpha L )\right) +r\alpha L (K-x^*) >0,
\end{align*}
because $x^*<K$ and $\alpha \beta K L <1$ so that $ \frac{K(1-\alpha L)}{1-\alpha \beta K L}-K(1- \alpha L )>0$.

Therefore, since $a_1(x,y)$ is strictly increasing in $y$ for fixed $x$, the  minimum occurs along the line $y=h(x)$ for $x^*\leq x< K$, where 
\begin{align*}
       a_1(x,h(x)) & = \frac{-s(K-x)(K(1-\alpha L)-x(1-\alpha  \beta K L))}{\alpha K}. 
\end{align*}
It is left to show that $K(1-\alpha L)-x(1-\alpha  \beta K L)\leq 0$ to conclude that $a_1(x,h(x))\geq 0$. Note that $K(1-\alpha L)-x(1-\alpha  \beta K L)$  is decreasing in $x$ since $\alpha L<1$ and $\beta K<1$. Hence,  
$$K(1-\alpha L)+x(-1+\alpha  \beta K L) \geq 
K(1-\alpha L)+x^*(-1+\alpha  \beta K L)=0,
$$
completing the claim for b).

{\it c)} Assume  $0<x\leq x^*$    and $y>k(x)$. To show that $\mathcal{L}_k(t,x,y)>0$, it suffices to show that $b_i(x,y)$,  $i\in \{1,2,3\}$, defined in \eqref{eq:partials_bi}, are all nonnegative and not all 0. Since $y>k(x)$, $b_0(x,y)>0$.

By \eqref{eq:Dbiy}, 
$\frac{\partial b_1(x,y)}{\partial x}$ is an increasing function of $y$  and since $y>k(x)$,  we have 
$$\frac{\partial b_1(x,y)}{\partial x}>\frac{\partial b_1(x,k(x))}{\partial x}=\beta L^2(s x \beta K+r(K-x))>0.
$$
Thus $b_1$ is increasing in $x$, so that the minimum of $b_1(x,y)$ occurs, i) for $y\geq L$ at $x=0$ and, ii) for $y<L$ the minimum occurs at $0\leq  x \leq x^*$ and $y= k(x)$.

\noindent
For i) $b_1(x,y)=b_1(0,y)=\alpha rKLy(y-L)\geq 0.$

\noindent
For ii), $0\leq x\leq x^*$, 
\begin{align*}
b_1(x,k(x))&= \beta rL^2 x(K(1-\alpha L) -x(1-\alpha \beta K L )) \\
&>   \beta r L^2 x(K(1-\alpha L) -x(1-\alpha L ))  \\
&= \beta r L^2 x(K-x)(1-\alpha L)>0.
\end{align*}
Therefore, $b_1(x,y)>0$ is the region considered in c).

It remains to show that $b_2(x,y)\geq 0$, in this region. From \eqref{eq:bi}, since $\alpha L<1$, $b_2(x,y)$ is increasing in $y$
and, hence,
\begin{align*}
    b_2(x,y)&\geq b_2(x,k(x))= rs\beta L x(Lx(\beta K-1)+k(x)K(1-\alpha L)\\
    &=\beta L^2 r s x\left\{x(\alpha\beta K L-1)+K(1-\alpha L)\right\}\\
    &\geq \beta L^2 r s x\left\{x^*(\alpha\beta K L-1)+K(1-\alpha L)\right\}=0,
\end{align*}
completing the proof of c).

{\it d)} Assume that  $x^*\leq x \leq \beta^{-1}$ and $0<y<k(x)$. To show that $\mathcal{L}_k(t,x,y)<0$, it suffices to show that $b_i(x,y)$, given in \eqref{eq:bi}, for $i\in \{0,1,2\}$ are nonpostive and not all zero. Since $y<k(x)$, $b_0(x,y)<0$. 

Next we show that $b_1(x,y) \leq 0$ for  $x^*\leq x\leq \beta^{-1}$ and $0<y<k(x)$. 
By \eqref{eq:partials_bi}, for fixed $x$,
$b_1(x,y)$ is a convex function of $y$  and for fixed $y$, $b_1(x,y)$ is a convex function of $x$. Therefore,   the maximum of $b_1(x,y)$  occurs  on the boundary of the region considered in this case, i.e., 
\begin{align*}
    \Upsilon_4&:=\{(x,y)\, : \,    x^*\leq x\leq \beta^{-1} , \, y=0 \},\\
    \Upsilon_5&:=\{(x,y)\, : \,    x^*\leq x\leq \beta^{-1} , \, y=k(x)\},\\
    \Upsilon_6&:=\{(x,y)\, : \,    x=x*, \, 0\leq y \leq y^*\}.
\end{align*}

\noindent
For $(x,y)\in \Upsilon_4$, $b_1(x,y)=b_1(x,0)=xL^2(r(\beta K-1)+s\beta K(\beta x-1))<0$.

% For $(x,y)\in \Omega_1$, 
% $$b_1(x,0)=L^2 x ((\beta K-1) r + \beta K s (\beta x-1))<0$$ since $\beta K<1$ and $x<\beta^{-1}$. 

\noindent
For $(x,y)\in \Upsilon_5$, $y=k(x)$ and 
\begin{align*} b_1(x,k(x))&=\beta L^2 r x (K (1-\alpha L)- x (1 - \alpha \beta  K L ))\\
&\leq \beta L^2 r x (K (1-\alpha L)- x^* (1 - \alpha \beta  K L ))=0.
\end{align*}

\noindent
For $(x,y)\in \Upsilon_6$, $x=x^*$ and
\begin{align*}
b_1(x^*,y)&=\frac{-K L [\beta s K (1 - \alpha L)  +r (1 - \alpha \beta K L) (1-\alpha L +\alpha y )] (L (1-\beta K)- y(1-\alpha\beta K L ))}{(1 - \alpha \beta K L)^2}.
 \end{align*}

Our assumptions imply that the  factor in the square brackets is positive. Furthermore, the  last factor in the numerator, in round brackets, is also positive because $y\leq y^*(<L)$ and, therefore, 
 \begin{equation}\label{eq:sign}
L (1-\beta K)- y(1-\alpha\beta K L )> L (1-\beta K)- y(1-\beta K  )=(L-y)(1-\beta K)>0.
\end{equation}
 Therefore $b_1(x^*,y)<0$, and hence
 $b_1(x,y)\leq 0$ in the region under consideration in d).

Lastly, $b_2(x,y)$, defined in \eqref{eq:bi} is an increasing function of $y$.  
Thus, 
\begin{align*}
b_2(x,y) &\leq  r s \beta Lx(Lx^*(\beta K-1)+ Ky (1-\alpha L)))\\
&=r s \beta Lx\left(\frac{ L(1-K\beta)-y(1- \alpha \beta K L ) }{\alpha \beta K L -1} \right)\stackrel{\eqref{eq:sign}}{<}0.
\end{align*}

This completes the  proof of part d) and hence the entire Lemma.
\end{proof}

A generic image of the nullclines and the respective regions  in Lemma~\ref{lem:signLhLkCase3} is provided in Fig.~\ref{fig:regions2}, where regions $\mathcal{S}_i, \, i=1,2,\dots,6$ and $\mathcal{B}_i, i=0,1$ are shown.  These regions are also crucial for the proof of global dynamics as formulated in Theorem~\ref{thm:Estar_GAS}. Regions $\mathcal{B}_i, i=0,1$  are defined in  Proposition~\ref{Prop:boxes} and  
    \begin {align*}
\mathcal{S}_1 &:=\{(x,y)\, : \, 0<x<x^*, \, \,  y^*<y<k(x) \},\\   
\mathcal{S}_2 &:=\{(x,y)\, : \, 0<x<x^*, \, \,  k(x)<y<h(x)  \},\\
\mathcal{S}_3 &:=\{(x,y)\, : \, 0<x<x^*, \, \,  y>h(x) \},\\
\mathcal{S}_4 &:=\{(x,y)\, : \, x^*<x<K, \, \,  0<y<h(x)\},\\
\mathcal{S}_5 &:=\{(x,y)\, : \, x^*<x<\frac{1}{\beta}, \, \,  \max\{0,h(x)\} <y<k(x)\},\\
\mathcal{S}_6 &:=\{(x,y)\, : \, x^*<x, \, \,  \max\{0,k(x)\}<y<y^*\}\},\\
\mathcal{S}_{i_T} &:=\{(x,y)\in \mathcal{\overline{S}}_{i} \,: \, x>0, \,  y>0, \, (x,y)\neq (x^*,y^*)\}, \ i=1,2,\dots,6,
    \end{align*}
where $\overline{S}_i$ is the closure of $S_i$ for $i\in \{1, \ldots, 6\}$.

\begin{figure}[h!]
    \centering
   \includegraphics[ width=0.65\linewidth]{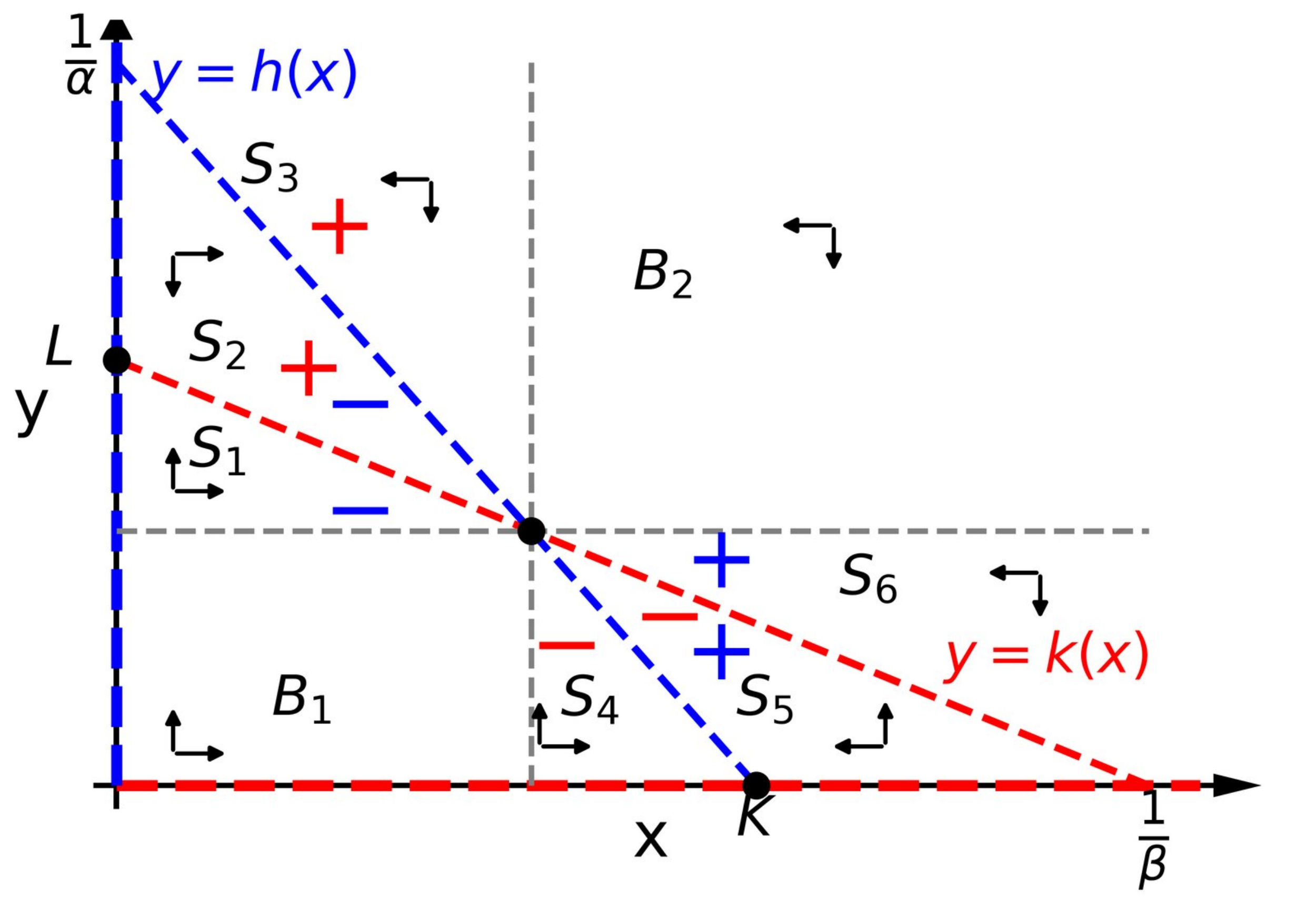}
    \caption{Regions considered with generic nullclines $y=h(x)$ and $y=k(x)$, in the case where $\alpha L<1$ and $\beta K<1$. The colored signs relate to the sign of the corresponding Root-operators $\mathcal{L}_h$ (in blue) and $\mathcal{L}_k$ (in orange). The signs are based on the results in Lemma~\ref{lem:signLhLkCase3}.
  The arrows  represent  the signs of $\Delta x$ (horizontal) and $\Delta y$ (vertical). Each region exhibits component-wise monotonicity.  }
    \label{fig:regions2}
\end{figure}

\begin{lemma}\label{lem:invarianceCase1new}
Consider \eqref{eq:CompT}  with $\alpha L <1$ and $\beta K<1$. The region $\mathcal{S}_{2_T}$ is positively invariant. The region 
$\mathcal{S}_{5_T}$
is positively invariant.  
\end{lemma}

\begin{proof}
    
    If $(x(t_0),y(t_0))\in \mathcal{S}_{2}$, then by Lemma~\ref{lem:signLhLkCase3}a) and c), $\mathcal{L}_h(t_0,x(t_0),y(t_0))<0$ and $\mathcal{L}_k(t_0,x(t_0),y(t_0))>0$. By Theorem~\ref{thm:need}, there exits $\epsilon>0$ such that $y(t)<h(x(t))$ and $y(t)>k(x(t))$ for $t\in (t_0, t_0+\epsilon]_\T$ with $t_0+\epsilon\geq \sigma(t_0)$. That is, $(x(t),y(t))\in \mathcal{S}_2$ for $t\in [t_0, t_0+\epsilon]_\T$. To complete the claim that $\mathcal{S}_{2_T}$ is positively invariant, we now consider the boundaries of this region within the positive quadrant. 
    If there exists $t_0\in \T$ such that $x (t_0)\in (0,x^*)$ and $y(t_0)=h(x(t_0))$. Since $h(x)>k(x)$ for $x\in (0,x^*)$, $G(x,h(x))<h(x)$ and since $F(x,h(x))=x$, so that       
    \begin{align*}
    \mathcal{L}_h(t_0,x (t_0),y(t_0))&=G(x (t_0),h(x(t_0)))-h(F(x(t_0),h(x(t_0))))\\
    &=G(x(t_0),h(x(t_0)))-h(x(t_0))
    <h(x(t_0))-h(x(t_0))=0.
    \end{align*}
Thus, by Theorem~\ref{thm:need}, there exists $\epsilon>0$ such that $y(t)<h(x(t))$ for $t\in (t_0,t_0+\epsilon]_\T$, where $t_0+\epsilon\geq \sigma(t_0)$. 
By Lemma~\ref{lem:signLhLkCase3}c), $\mathcal{L}_k(t,x,y)>0$ for $y>k(x)$ so that, by Theorem~\ref{thm:need}, there exists $\delta>0$ such that $y(t)>k(x(t))$ for $t\in (t_0,t_0+\delta]_\T$ with $t_0+\delta\geq \sigma(t_0)$. Thus, choosing $\overline{\epsilon}=\min\{\epsilon, \delta\}$, $(x(t),y(t))\in \mathcal{S}_{2}$ for $t\in (t_0,t_0+\overline{\epsilon}]_\T$. 

 If there exists $t_0\in \T$ such that $x(t_0)\in (0,x^*)$ and $y(t_0)=k(x(t_0))$. Since $h(x)>k(x)$ for $x\in (0,x^*)$, $F(x(t_0),k(x(t_0)))>x(t_0)$ and $G(x(t_0),k(x(t_0)))=k(x(t_0))$. Furthermore, since $y=k(x)$ is decreasing, $k(F(x(t_0),k(x(t_0)))<k(x(t_0))$. We therefore have   
 \begin{align*}
 \mathcal{L}_k(t_0,x t_0),k(x(t_0)))&=G(x (t_0),k(x(t_0)))-k(F(x(t_0),k(x(t_0))))\\
 &>k(x (t_0))-k(x(t_0))=0.
 \end{align*}
Thus, by Theorem~\ref{thm:need}, there exists $\epsilon>0$ such that $y(t)>k(x(t))$ for $t\in (t_0,t_0+\epsilon]_\T$, where $t_0+\epsilon\geq \sigma(t_0)$. 
By Lemma~\ref{lem:signLhLkCase3}a), $\mathcal{L}_h(t,x,y)>0$ for $y=k(x)<h(x)$ so that, by Theorem~\ref{thm:need}, there exists $\delta>0$ such that $y(t)<h(x(t))$ for $t\in (t_0,t_0+\delta]_\T$ with $t_0+\delta\geq \sigma(t_0)$. Thus, choosing $\overline{\epsilon}=\min\{\epsilon, \delta\}$, $(x(t),y(t))\in \mathcal{S}_{2}$ for $t\in (t_0,t_0+\overline{\epsilon}]_\T$, confirming the claim that $\mathcal{S}_{2T}$ is positively invariant.

A similar argument can be made to prove the positive invariance of  $\mathcal{S}_{5_T}$. 

\end{proof}

\begin{theorem}\label{thm:Estar_GAS}
     Consider \eqref{eq:CompT} and assume that 
      $\alpha L <1$ and $\beta K<1$. Then,  there exists a unique interior equilibrium  $E^*=(x^*,y^*)$,  where $x^*,y^*$ are given in \eqref{eq:equilbria}. $E^*$ is  asymptotically stable and attracts all solutions with positive initial conditions. %Both $E^*_K=(K,0)$ and $E^*_L=(0,L)$ repell all solutions with positive initial conditions,  and $E_0^*$ is a repellor.
\end{theorem}

\begin{proof} 
   
   We  show that any solution starting in the positive quadrant either enters $\overline{\mathcal{S}}_{2}$  and  converges to $E^*$, or enters $\overline{\mathcal{S}}_5$ and converges to $E^*$, unless it remains in $\mathcal{B}_0$ or $\mathcal{B}_1$ and converges to $E^*$. 
{\it i)} Assume that  $(x(t_0),y(t_0))  \in  \mathcal{B}_0$. Either $(x(t),y(t))\in  \mathcal{B}_0$ for all $t\in [t_0,\infty)_\T$ or there exists $T>t_0$ such that $(x(T),y(T))\in (0,\infty)^2\backslash \mathcal{B}_0$. Note that in the former case,  by the component-wise monotonicity of $\Delta x, \Delta y$ and Proposition~\ref{prop:convergence}, $(x,y)$ must converge to an equilibrium in 
$\mathcal{B}_0$.  Since the only reachable equilibrium is  $E^*$, if the solution remains in 
$\mathcal{B}_0$, it must converge to $E^*$. 
In the latter case,  by  Proposition~\ref{Prop:boxes}, there exists $T>t_0$ such that
$(x(T),y(T))\in (0,\infty)^2\backslash \{\mathcal{B}_0 \cup \mathcal{B}_1\}$.

{\it ii)}  Arguing as in {\it i)}, if $(x(t_0),y(t_0))\in \mathcal{B}_1$, then either $(x(t),y(t))\in  \mathcal{B}_1$ for all $t\in [t_0,\infty)_\T$, in which case it must converge to $E^*$, or there exists $T>t_0$ such that $(x(T),y(T))\in  \bigcup_{i=1}^6 \mathcal{S}_{i_T} \cup \mathcal{B}_0$, so that  again, it suffices to know what happens to solutions with
$(x(t_0),y(t_))\in (0,\infty)^2\backslash \{\mathcal{B}_0 \cup \mathcal{B}_1\}=\mathcal{S}_1\cup \mathcal{S}_{2_T}\cup \mathcal{S}_3\cup \mathcal{S}_4\cup \mathcal{S}_{5_T}\cup \mathcal{S}_6$.

Thus, it suffices to show that for $(x(t_0),y(t_0))\in \mathcal{S}_1\cup \mathcal{S}_{2_T}\cup \mathcal{S}_3\cup \mathcal{S}_4\cup \mathcal{S}_{5_T}\cup \mathcal{S}_6$, the solution converges to $E^*$.

{\it iii)} Assume that  $(x(t_0),y(t_0))\in \mathcal{S}_{1}$. The  solution cannot remain in this region, since  based on the component-wise monotonicity,  the solution would have to converge to an equilibrium in $\mathcal{S}_{1_T}$ and none of the equilibria are reachable. Since  $\mathcal{L}_k(t,x,y)>0$ in this region  (see Lemma~\ref{lem:signLhLkCase3}c), by Theorem~\ref{thm:need},
there exist $T>t_0$ such that $(x(T),y(T))\in \mathcal{S}_{2_T}$. This region is by Lemma~\ref{lem:invarianceCase1new} positively invariant.

{\it iv)} Assume that  $(x(t_0),y(t_0))\in \mathcal{S}_{2_T}$. By Lemma~\ref{lem:invarianceCase1new}, the solution remains in $\mathcal{S}_{2_T}$.  
By Proposition~\ref{prop:convergence}, solutions must converge in $\mathcal{S}_{2_T}$ to an equilibrium and the only reachable equilibrium  is $E^*$. 

{\it v)} Assume that  $(x(t_0),y(t_0))\in \mathcal{S}_{3}$. Based on the component-wise monotonicity in $\mathcal{S}_{3}$ and by Proposition~\ref{prop:convergence},  the solution either converges to $E^*$
or since $\mathcal{L}_h(t,x,y)<0$ by Lemma~\ref{lem:signLhLkCase3}a), 
 must enter $\mathcal{S}_{2_T}$ and case {\it iv)} applies. 

{\it vi)} Assume that  $(x(t_0),y(t_0))\in \mathcal{S}_{4}$. 
The  solution cannot remain in this region, because the component-wise monotonicity  implies that the solution must converge to an equilibrium but the direction field indicates that no equilibrium is reachable. By Lemma~\ref{lem:signLhLkCase3}b) and d), 
$\mathcal{L}_h(t_0,x,y)>0$  in   $\mathcal{S}_{4}$. Thus, there exists $T>t_0$ such that $(x(T),y(T))\in \mathcal{S}_{5_T}$.  This region is by Lemma~\ref{lem:invarianceCase1new} positively invariant.

{\it vii)} Assume that  $(x(t_0),y(t_0))\in 
\mathcal{S}_{5_T}$. By Lemma~\ref{lem:invarianceCase1new}, this region is positively invariant.  By the direction field and Proposition~\ref{prop:convergence}, the solution  must  converge to $E^*$. 

{\it viii)} Assume that  $(x(t_0),y(t_0))\in \mathcal{S}_{6}$. By the  component-wise monotonicity,  
the solution either converges to $E^*$
or since $\mathcal{L}_k(t,x,y)<0$ in this region (see Lemma~\ref{lem:signLhLkCase3}d),  the solution 
 must enter $\mathcal{S}_{5_T}$. 
This completes the proof.
\end{proof}

An example of the dynamic (augmented) phase portrait for the time scale $\T=2^{\mathbb{N}_0}$ is provided in Fig.~\ref{fig:example3_coexistence_g} with the incorporation of the first few Root-curves (solid lines). We notice once more that the Root-curves are moving away from their associated nullclines as time increases, similar to the  behavior displayed in Fig.~\ref{fig:Example1_C1_g}. However,  the Root-curves associated with the nontrivial nullclines all intersect at the interior equilibrium.  

\begin{figure}[h!]
    \centering
    \includegraphics[width=0.85\linewidth]{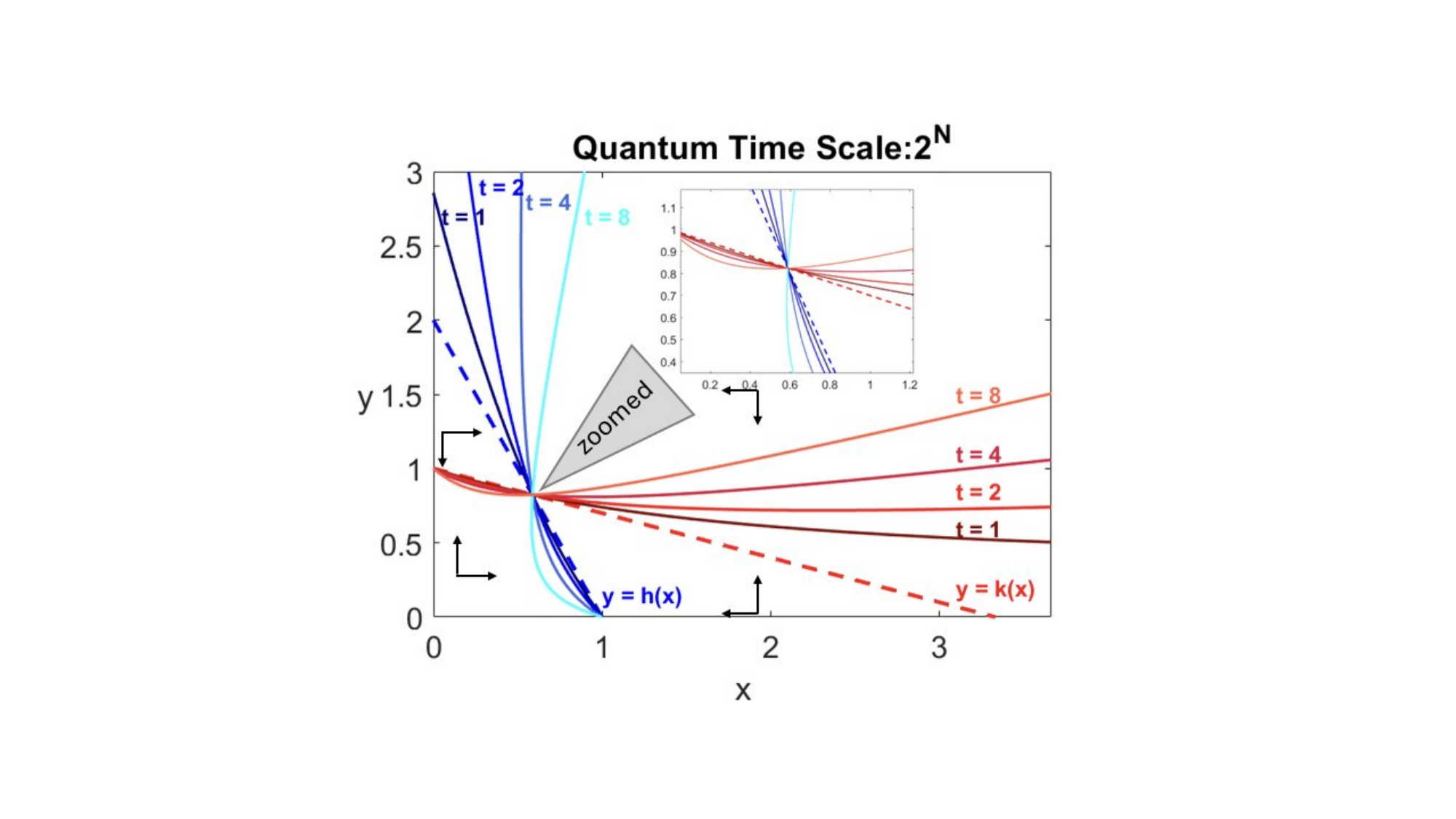}
    \caption{Dynamic Phase Plane of \eqref{eq:CompT} on the quantum time scale $2^{\mathbb{N}_0}$ including the (time-dependent) Root-curves for the specific $t$-values $t\in \{2^0, 2^1, 2^2, 2^3\}$. The  (dashed) $y$-nullcline and the (solid) Root-operator $\mathcal{R}_k$ are displayed in shades of red and the (dashed) $X$-nullcline and the (solid) Root-operator $\mathcal{R}_h$ are in shades of blue. The parameter values are: $\alpha  = 0.5$, $\beta=0.3$, $r=0.5$, $s=0.3$ and $L=K=1$.  }
    \label{fig:example3_coexistence_g}
\end{figure}

\begin{example} 
Fig.~\ref{fig:example3_coexistence} illustrates the first three iterates of the (purple) orbit initialized at $(x(2^0),y(2^0))=(2,1)$ at time $t_0=2^0=1$, including the associated Root-curves determining the positioning of the next iterate with respect to the corresponding nullcline (dashed curve). The iterate at time $T$, for which the Root-curve at time $T$ is relevant to, is indicated by a star symbol. More precisely, the dynamic phase plane at time $t=2^0=1$ includes the (dashed) nullclines and the color-coded associated Root-curves at time $t=1$.  Since the initial value $(x(2^0),y(2^0))=(2,1)$, see purple star in a), is in a region where the Root-operators associated with $y=h(x)$ and $y=k(x)$ are both positive (indicated by the red and blue "+" signs), the next iterate $(x(2^1),y(2^1))$ remains above both nontrivial nullclines,  see purple dot. Fig.~\ref{fig:example3_coexistence}b) illustrates the dynamic phase plane at time $t=2^1=2$. Since $(x(2^1),y(2^1))$, indicated by a purple star, is above both Root-curves where the Root-operators are positive, the next iterate $(x(2^2),y(2^2))$ remains above both nullclines. However, at time $t=2^2$, see c), we observe that $(x(2^2),y(2^2))$ (purple star) is in a region where the Root-operator associated with $y=k(x)$ is negative, but the Root-operator associated with $y=h(x)$ is positive. This implies that the next iterate $(x(2^3),y(2^3))$  must enter the (positively invariant) region bounded by the $x$-axis and the two nontrivial nullclines. Based on the component-wise monotonicity, the solution converges, by Theorem~\ref{thm:Estar_GAS}, to the coexistence equilibrium,   $E^*$.

\begin{figure}[h!]
    \centering
\includegraphics[scale=0.25]{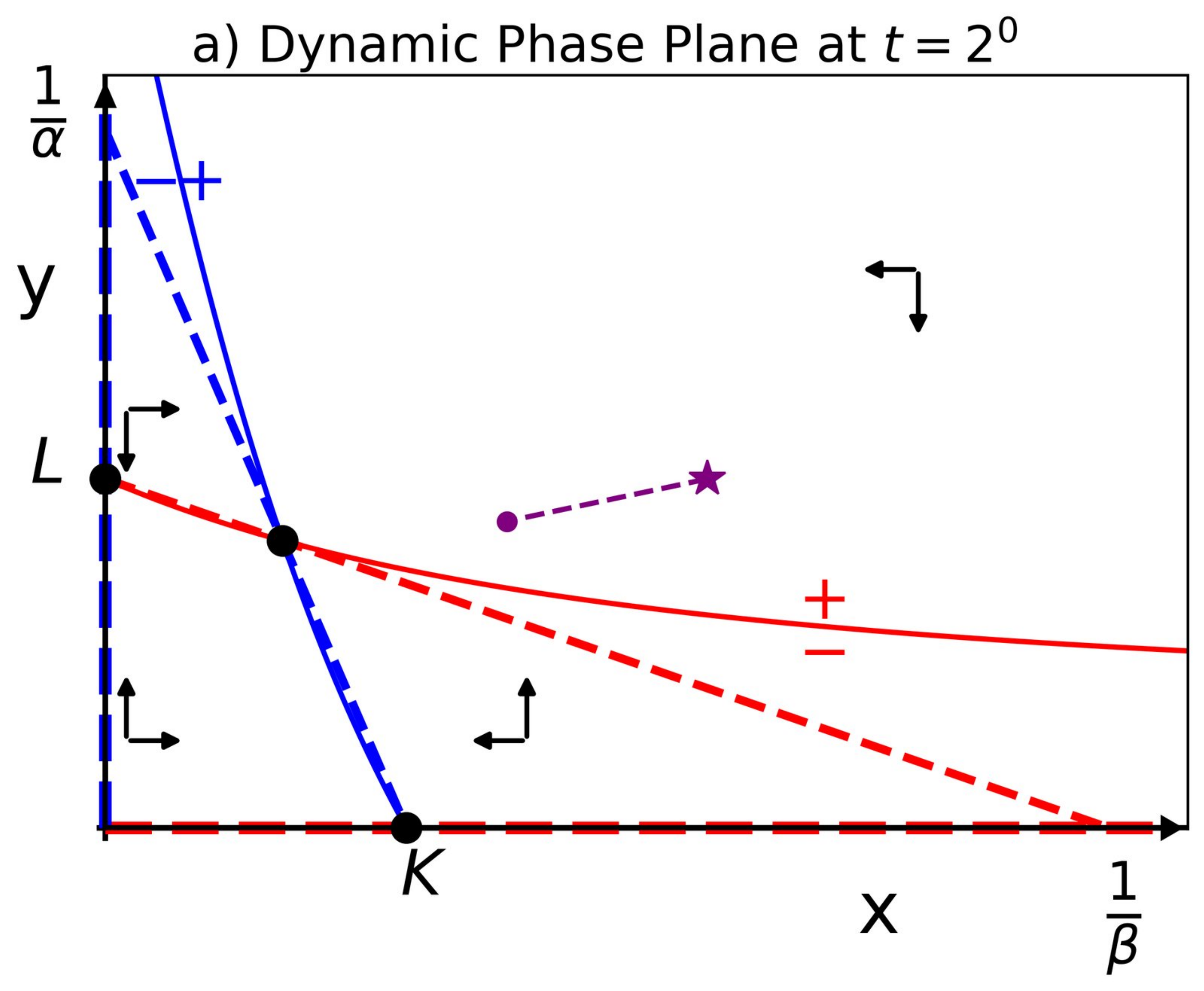}
\includegraphics[scale=0.25]{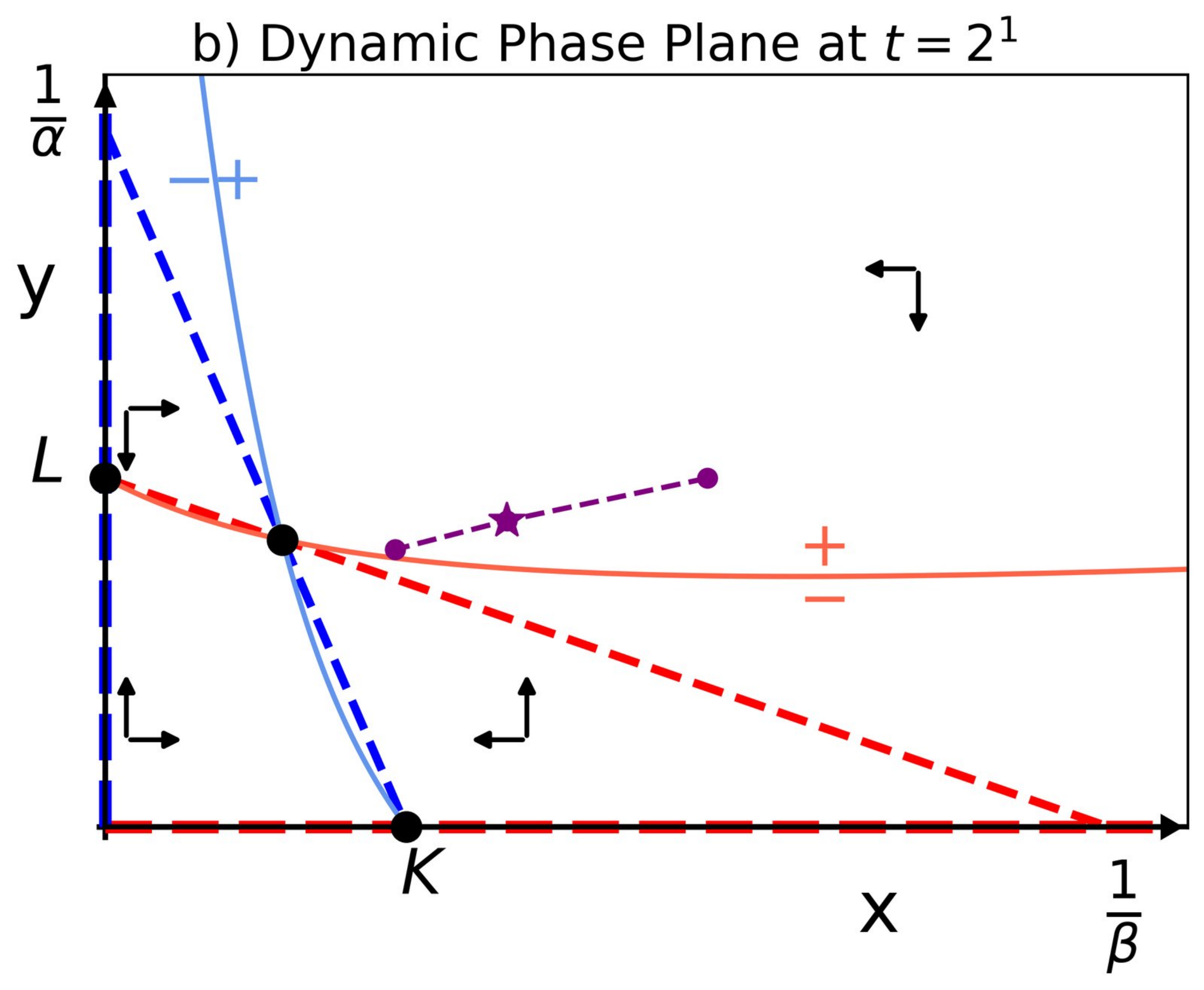}
\includegraphics[scale=0.25]{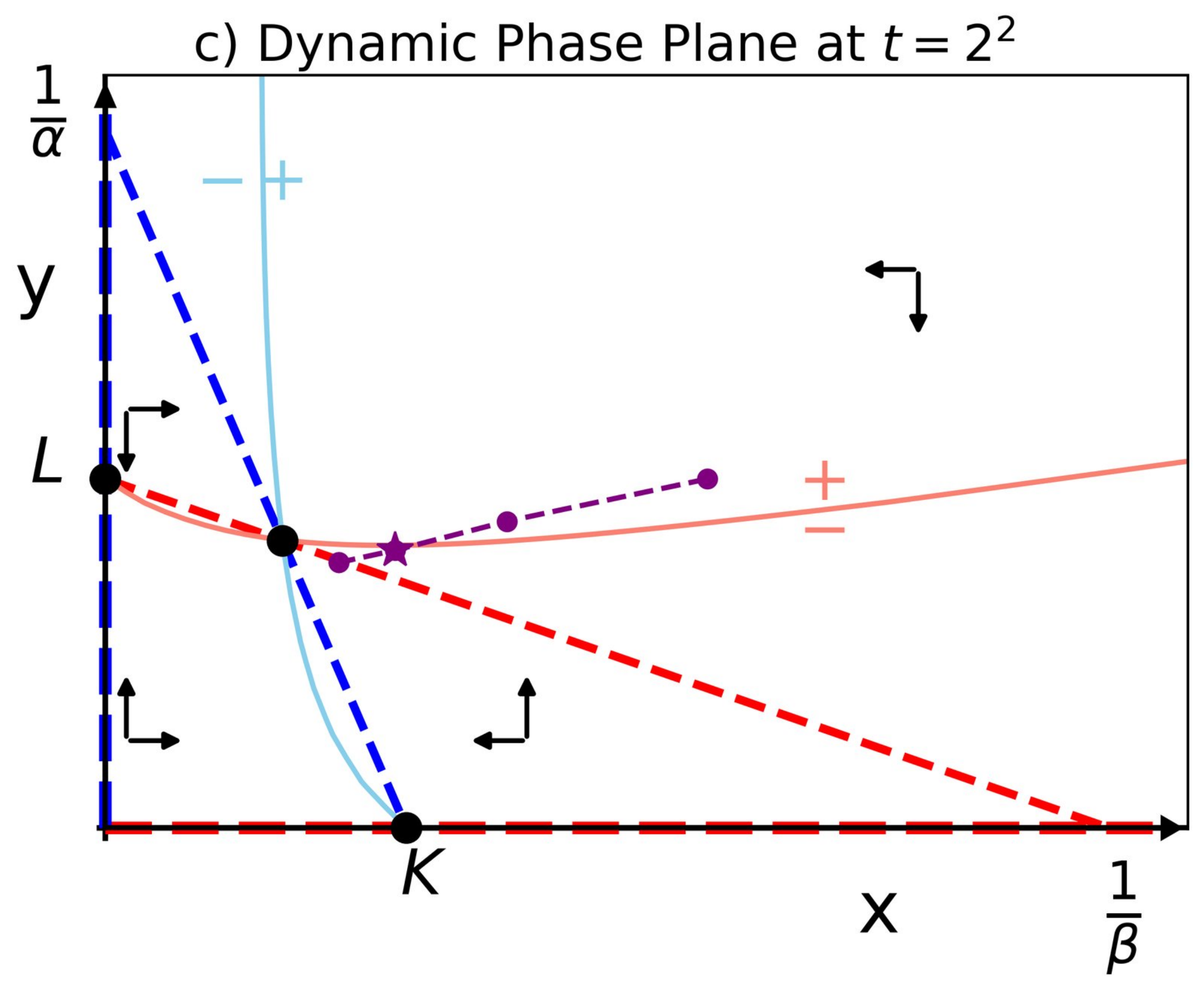}
    \caption{  a) Dynamic Phase Plane of \eqref{eq:CompT} for $\T=2^{\mathbb{N}_0}$ with $\alpha L, \beta K<1$ and the same parameter values as in Fig.~\ref{fig:example3_coexistence_g}. Root-curves are  solid curves, nullclines are dashed.  Curves associated with the $y$-equation are in (shades of) red and those associated with the $x$-equation are in (shades of) blue. b)--d) The first three iterations of a solution initialized at $t_0=2^0=1$ at $(x(2^0),y(2^0))=(2,1)$. The respective time-dependent Root-curve applies to the corresponding time iterate (purple star). The solution represented is in the positively invariant region within 3 iterations and then converges by Theorem~\ref{thm:Estar_GAS} to the coexistence equilibrium, $E^*$.}
    \label{fig:example3_coexistence}
\end{figure}
\end{example}

\begin{example}
Consider $\T=\{1\} \cup_n [2^{2n-1}, 2^{2n}] = \{1\}\cup [2,4]\cup [8,16] \cup \ldots$. At any point within the continuous interval, i.e., $t\in [2^{2n-1},2^{2n})$, the Root-curves are identical to the corresponding nullclines. However, at the right-scattered points $2^{2n}$, the root-curves may differ. The dynamic (augmented) phase plane for the first few $t\in \T$ is shown in Fig.~\ref{fig:Example1_C2}, where we used the same parameter values as in Fig.~\ref{fig:example3_coexistence}. The figure also includes a sample orbit initialized at $t_0=1$ with value $(x(t_0),y(t_0))=(1,2)$. The solution at time $t$, for which the dynamic phase plane is illustrated, is highlighted by a star. That is, for the dynamic phase plane at time $t=1$, the initial value is highlighted by a star. Since the initial value $(x(t_0),y(t_0))$ is in a region where both Root-operators are positive, the next iterate at $t=2$ remains above both nullclines. Since $t$ is right-dense for all $t\in [2,4)$, the Root-curves are identical to their associated nullclines and the dynamic phase plane is identical to the classical phase plane. However, at the right-scattered point $t=4$ (see purple star in c), the Root-curves no longer are identical to their respective nullclines. Regardless, the point $(x(4),y(4))$ is in a region where both Root-operators are still positive so that the next iterate $(x(8),y(8))$ remains above both nullclines.

\begin{figure}[h!]
    \centering
\includegraphics[scale=0.25]{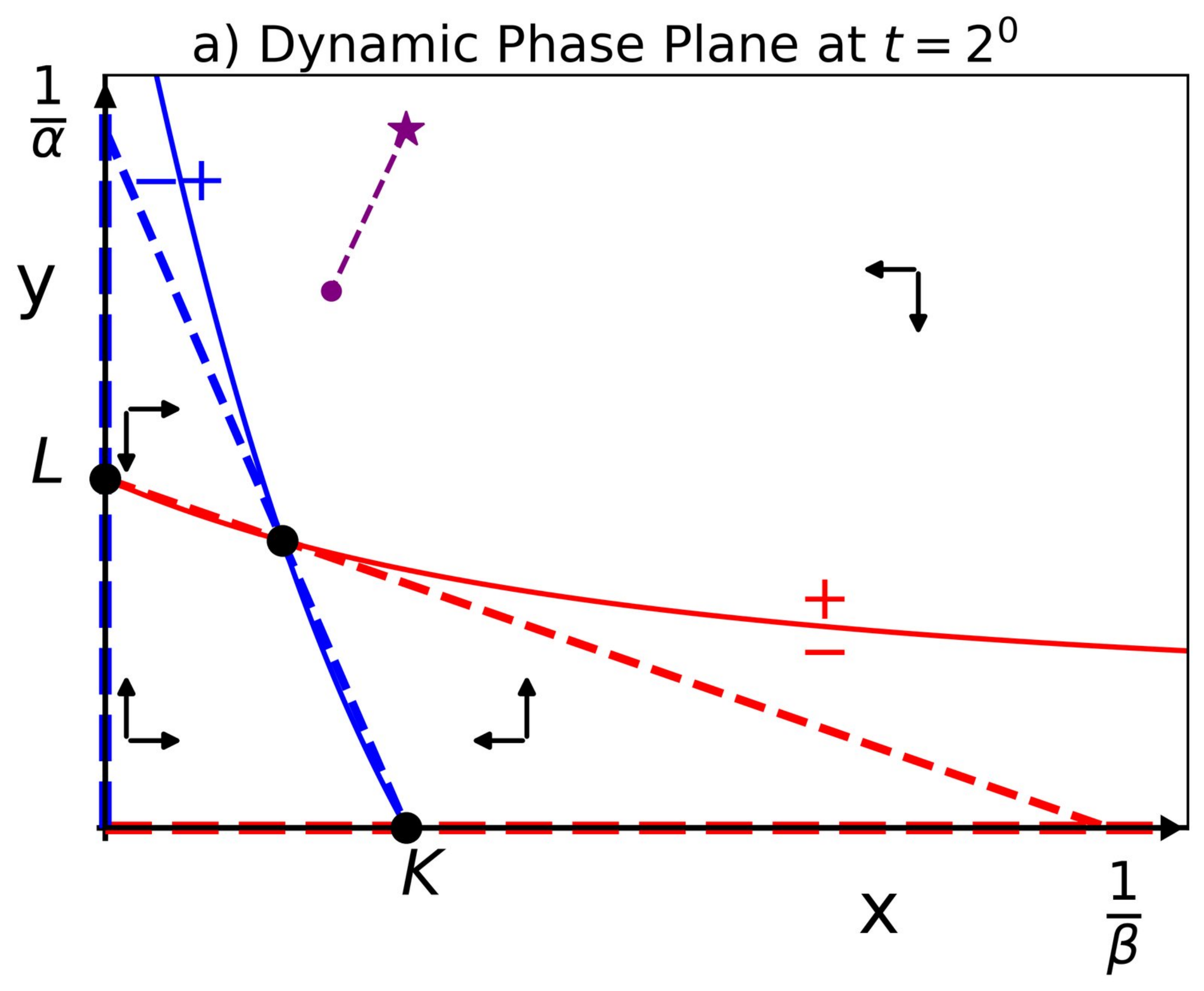}
\includegraphics[scale=0.25]{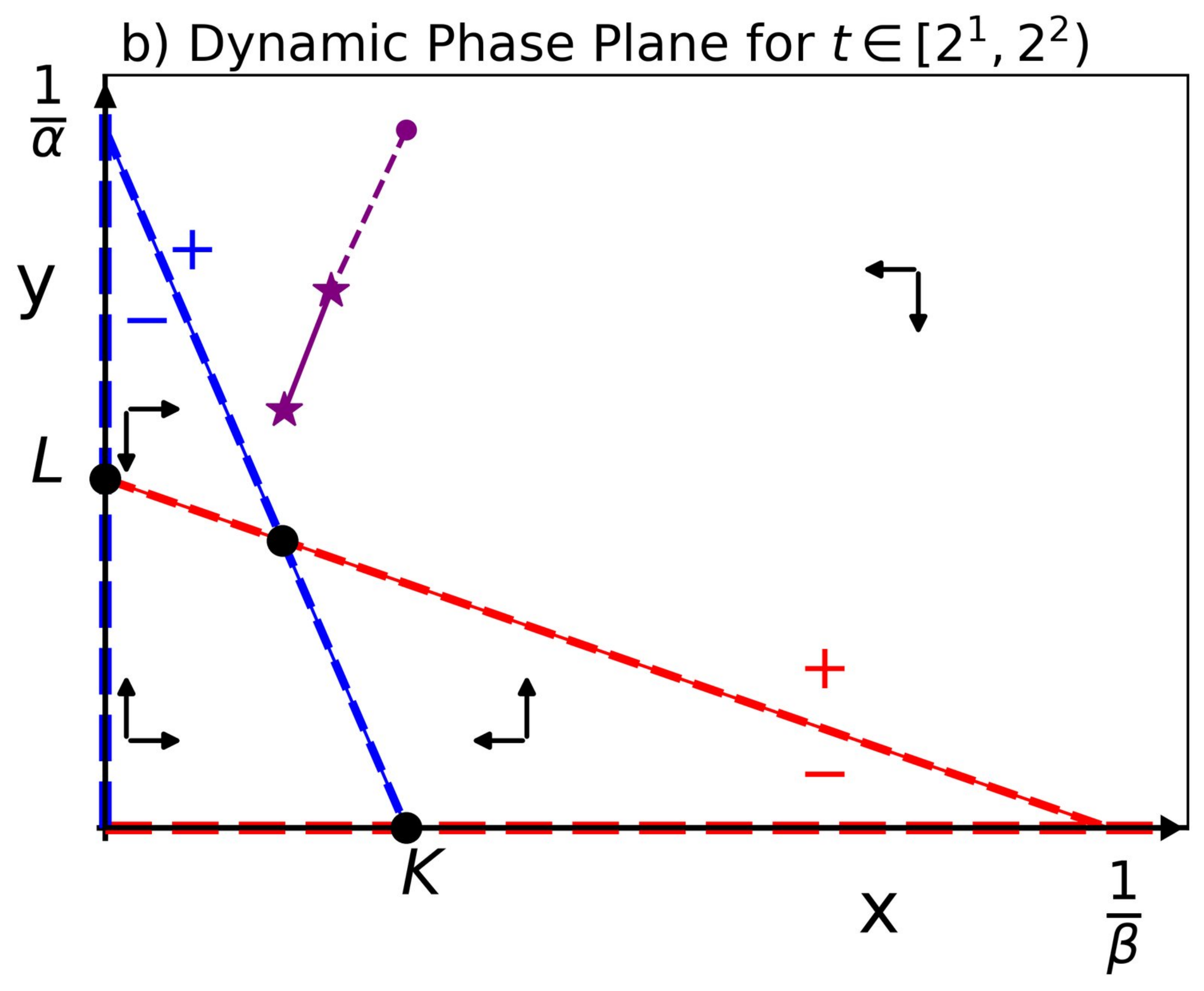}
\includegraphics[scale=0.25]{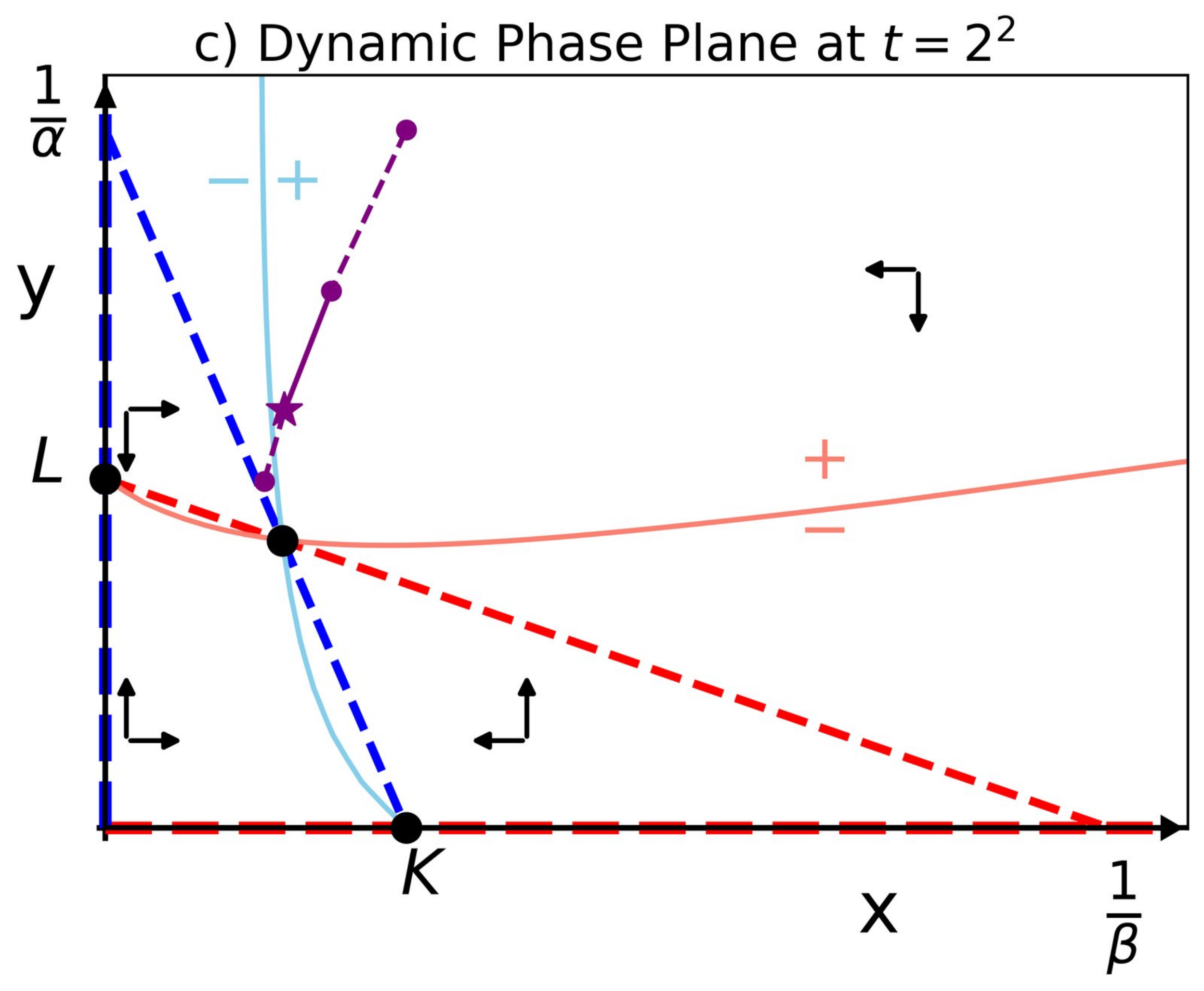}
    \caption{Dynamic Phase Portrait for the first few time points of the time scale $\T=\{1\}\cup [2, 4]\cup  [8, 16] \ldots$ including a sample orbit initialized at $t_0=1$ with value $(x(1),y(1))=(1,2)$. The parameter values are the same as for Fig.~\ref{fig:example3_coexistence_g}. Purple colored continuous (solid) parts refer to the time points in intervals, where the dynamic equation on time scales reduces to a differential equation. Purple colored points obtained by applying the dynamic equation to a right-scattered point are highlighted by isolated dots. The solution value at $t$ corresponding to the same time as the dynamic phase plane is highlighted as purple star.}
    \label{fig:Example1_C2}
\end{figure}
\end{example}

\vspace{3mm}

\textbf{ (IV) \, $\alpha L>1$ and $\beta K>1$}

\vspace{3mm}

\noindent In this case,  $h(x)<k(x)$ for all $x\in (0,x^*)$ and $h(x)>k(x)$ for all $x \in (x^*,K)$. Recall that $h(x)>0$ for $x\in (0,K)$ and $k(x)>0$ for $x\in (0,\beta^{-1})\subset(0,K)$, but $k(x)<0$, for all $x>\beta^{-1}$.

 \begin{lemma}\label{lem:signLhLkCase2}
     Consider \eqref{eq:CompT} with $\alpha L >1$ and $\beta K>1$.  Then, the following hold:
     \begin{itemize}
         \item[a)]  $\mathcal{L}_h(t,x,y)>0$ for all  $0<x\leq x^*$, $y>h(x)$, and $t\in \T$. 
\item[b)] $\mathcal{L}_h(t,x,y)<0$ for all  $x^*\leq x<K$, $
%\max\{0,k(x)\}
0<y<h(x)$, and $t\in \T$.    
\item[c)]  $\mathcal{L}_k(t,x,y)<0$ for all \, $0<x <x^*$, $y^*\leq y<k(x)$, and $t\in \T$.  
\item[d)]  $\mathcal{L}_k(t,x,y)>0$ for all $x^*\leq x<K$,  $\max\{0,k(x)\}<y\leq y^* $, and $t\in \T$.  
    \end{itemize}
\end{lemma}

\begin{proof}
Let $\alpha L>1$ and $\beta K>1$.
    \begin{itemize}
    \item[{\it a)}]  Assume  $0<x\leq x^*$ and $y>h(x)$.  Since $\alpha L>1$ and $\beta K>1$, $k(x)>h(x)$. To prove $\mathcal{L}_h(t,x,y)>0$, by \eqref{eq:Lh}, it suffices to show that $a_i(x,y)$, given in \eqref{eq:ai}, are nonnegative and not all zero, for $i\in \{0,1,2\}$. 
    
Since $y>h(x)$, $a_0(x,y)>0$. Furthermore, 
since $\alpha L>1$ and $y>h(x)$, it follows that
    \begin{align*}
  a_2(x,y)&>  \alpha r s y (h(x)K(\alpha L-1) + x L  (1 - \beta K))\\
    &=r s y (\alpha KL(1-\beta x)+x-K)=r s y \alpha K(k(x)-h(x))> 0.
 \end{align*}
    It remains to show that $a_1(x,y)\geq 0$. Assume first that  $y>k(x)$, and note that $k(x)>h(x)$ for $0<x\leq x^*$. Then,  
    \begin{align*}
        a_1(x,y)&>\beta Lsx(x-K)+(s(x-K)+\alpha L((x-K)r+Ks))y+\alpha^2KLryh(x)\\
         &> (x-K)(\beta L s x+sy +\alpha rLy ) + s\alpha LKh(x) + \alpha rLy (K-x)\\
         &=(x-K)(\beta L s x+sy +\alpha rLy ) + s L(K-x) + \alpha rLy (K-x)\\
         &=s(K-x)(L(1-\beta x)-y) = s(K-x)(k(x)-y)>0.
    \end{align*}
Now, assume that  $h(x)<y\leq k(x)$,
     \begin{align*}
        a_1(x,y)&>\beta Lsx(x-K)+(s(x-K)+\alpha L((x-K)r+Ks))y+\alpha^2KLryh(x)\\
         &> (x-K)(\beta L s x+sy +\alpha rLy ) + s\alpha LKh(x) + \alpha rLy (K-x)\\
         &\geq (x-K)(\beta L s x+sk(x) +\alpha rLy ) + s L(K-x) + \alpha rLy (K-x)\\
         &=(x-K)(\beta L s x+sL(1-\beta x)-sL)=0.
    \end{align*}
    This completes the proof of case {\it a)}. 
    %Sabrina: I am here
   
    \item[{\it b)}]  Assume that   $x^*\leq x<K$ and  $0<y<h(x)$. To show that $\mathcal{L}_h(t,x,y)$ given in \eqref{eq:Lh} is negative, it suffices to show that  $a_i(x,y), i=0,1,2$, given in \eqref{eq:ai} are all non-positive and not all zero. 
    First, $a_0(x,y)<0$, since $y<h(x)$.     
Next, consider $a_1(x,y)$. From \eqref{eq:partials_ai}, it follows that for fixed $x$, $a_1(x,y)$ is a convex function of $y$ and for fixed $y$,  $a_1(x,y)$ is a convex function of $x$.  Hence, it suffices to show that $a_1(x,y)\leq 0$ on the sets
$$\Lambda_1:=\{(x,y)\, : \, x=x^*, \, 0\leq y\leq y^*\}, $$
$$\Lambda_2:=\{ (x,y) \, : \, y=0, \, x^*\leq x \leq K \}, $$ 
and $$\Lambda_3:=\{ (x,y)\, : \, 
 x^*\leq x \leq K, \, y = h(x)\} .$$

Let $(x,y)\in \Lambda_1$.
Since $a_1(x,y)$ is convex in   $y$ for fixed $x^*$, and $y^*=h(x^*)$, \\ $a_1(x^*,y)\leq 
\max\{a_1(x^*,0), a_1(x^*,h(x^*))\}$. Since $a_1(x^*,0)=\beta sL x^*(x^*-K)<0$ 
     and $a_1(x^*,h(x^*))=a_1(x^*,y^*)=0$,  $a_1(x^*,y)\leq 0$ for $(x,y)\in  \Lambda_1$.

Next, let  $(x,y)\in \Lambda_2$. Setting $y=0$ in $a_1(x,y)$ given in \eqref{eq:ai}, $a_1(x,0)=sx\beta L(x-K) \leq 0$, for $x\leq K$    with equality if and only if $x=K$.

   Finally, let  $(x,y)\in \Lambda_3$. Substituting $y=h(x)$ in $a_1(x,y)$ defined  in \eqref{eq:ai} yields 
\begin{align*}
a_1(x,h(x))&=\frac{s(K-x)}{\alpha K}((1-\alpha \beta LK)x  -K(1-\alpha L))\leq 0,
\end{align*}
since, $((1-\alpha \beta LK)x  -K(1-\alpha L))$ is decreasing in $x$, so that it attains its maximum value  at $x^*$ for  $x\in \Lambda_3$ and $a_1(x^*,h(x^*))=a_1(x^*,y^*)=0$. 
 Therefore, $a_1(x,y)\leq 0$ for all  $(x,y)\in \Lambda_1\bigcup \Lambda_2 \bigcup \Lambda_3$ with equality if and only if $x=K$ or $x=x^*$.

Lastly, we show   that $a_2(x,y)\leq 0$ for $x^*\leq x\leq  K$ and $0\leq y \leq  h(x)$, where $a_2(x,y)$ is defined in \eqref{eq:ai}.
Since, $\alpha L>1$ and $y \leq h(x)$, we have
$$yK (\alpha L-1) + xL  (1 - \beta K)\leq h(x)K (\alpha L-1) + xL  (1 - \beta K))= (1-\alpha\beta K L)x+K(\alpha L-1),$$ where the right hand side is a decreasing function that attains its maximum for points in $\Lambda_3$ at $x=x^*$.
Evaluating this expression at $x=x^*$, using the value given in \eqref{eq:equilbria}, it follows that 
 $(1-\alpha \beta KL)x^*+K(\alpha L-1)=0.$ 
 Hence, $a_2(x,y)\leq 0$ for all $(x,y)\in\Lambda_3$ with equality if and only $x=x^*$ and $y=h(x^*)=y^*$. 
     This completes the proof of part {\it b)}.

 \item[{\it c)}]  Let $0<x<x^*$  and $y^*\leq  y<k(x)$. To show that $\mathcal{L}_k(t,x,y)$ given in \eqref{eq:Lk} is negative, it suffices to show that 
 $b_i(x,y)$ for $i=0,1,2$, given in \eqref{eq:bi}, are all non-positive and not all zero. 
    First,  $b_0(x,y)<0$, since $y<k(x)$.
%%%%%%%%%%%%%%%%%%%%%%%%%

Next we show that $b_1(x,y)\leq 0$ for $y^*\leq y<k(x)$ and $0<x<x^*$. 
By \eqref{eq:partials_bi}, for fixed $x$,
$b_1(x,y)$ is a convex function of $y$  and for fixed $y$, $b_1(x,y)$ is a convex function of $x$. Therefore,   the maximum of $b_1(x,y)$  occurs  on the boundary of the region considered in this case, i.e., 
for $\displaystyle{(x,y)\in \bigcup_{i=4}^6 \overline{\Lambda}_i}$ where 
\begin{align*}
   \Lambda_4&:=\{(x,y)\colon \,  x=0, \, y^*\leq y<L\},\\
  \Lambda_5&:=\{(x,y)\colon \,  0<x<x^*, \, y=y^*\},  \\
 \Lambda_6&:=\{(x,y)\colon  \, 0<x<x^*, \, y=k(x)\}.
\end{align*}
 Let $(x,y)\in \overline{\Lambda}_4$. Then,
 $b_1(0,y) = \alpha r y K  L(y-L)\leq 0$, with equality if and only if $y=L$.  

 For $(x,y)\in \overline{\Lambda}_5$, $b_1(x,y^*)$ is convex in $x$, by \eqref{eq:partials_bi}, so that its maximum occurs at an end point, that is, $b_1(x,y^*)\leq \max\{b_1(0,y^*), b_1(x^*,y^*)\}$. Since $b_1(0,y^*)=\alpha r y^* K  L(y^*-L)< 0$
and $b_1(x^*,y^*)=0$,  it follows that $b_1(x,y^*)\leq 0$ for $(x,y)\in \overline{\Lambda}_5$ with equality if and only if $(x,y)=(x^*,y^*)$. 

For $(x,y)\in \overline{\Lambda}_6$, $b_2(x,y) =b_2(x,k(x))=\beta L^2 r sx \phi_2(x)$ where  $\phi_2(x)=K(1-\alpha L)+x(\alpha \beta  K L - 1)$, that is, since $\alpha \beta K L >1$, a  strictly increasing function of  $x$ and hence, $\phi_2(x)< \phi_2(x^*)=0$ for $x<x^*$, so that $b_2(x,k(x))\leq 0$ for $(x,y) \in\overline{\Lambda}_6$.

It remains to show that $b_2(x,y)\leq 0$ for $0<x<x^*$ and $y^*\leq y<k(x)$.
By its definition in \eqref{eq:bi},  since $\alpha L>1$, $b_2(x,y)$ is a strictly decreasing function of $y$ and so for $y\geq y^*$, 
\begin{align*}
b_2(x,y)&\leq b_2(x,y^*)=rsx\beta L(Lx(\beta K-1)+Ky^*(1-\alpha L))\\
    &=rsx\beta L\left(\frac{L(\beta K -1)}{(\alpha \beta K L-1)}\right)\phi_2(x),
\end{align*}
where $\phi_2(x):=K(1-\alpha L)+x(\alpha \beta KL-1)$.  Since $\phi_2(x)$ is strictly increasing in $x$ for $0\leq x \leq x^*$, $\phi_2(x)\leq \phi_2(x^*)=0.$ Thus, $b_2(x,y)\leq 0$  with quality if and only if $x=x^*$ and $y=y^*$.

This completes the proof of part {\it c)}.

\item[{\it d)}]   
    Assume $(x,y)\neq(x^*,y^*),$  $x^*\leq x<K$ and $\max\{0, k(x)\}<y\leq y^*$. To show  that $\mathcal{L}_k(t,x,y)$ given in \eqref{eq:Lk} is positive,  it suffices to show that $b_i(x,y)$, for $i=0,1,2$,   defined in \eqref{eq:bi}, are nonnegative and not all zero.  Since $\max\{0,k(x)\}<y $, $b_0(x,y)> 0$.

For $x\geq x^*$ and $\max\{0,h(x)\}<y\leq  y^*$, since $\mbox{sign}(b_2(x,y))=\mbox{sign}(Lx(\beta K-1)+Ky(1-\alpha L))$ and the latter is increasing in $x$ and decreasing in $y$, we have 
\begin{align*}
b_2(x,y)\geq 
&rsx \beta L(Lx^*(\beta K-1)+Ky^*(1-\alpha L))=0.
\end{align*}

It remains to show that $b_1(x,y)\geq 0$ for $x^*\leq x <K$ and $\max\{0, k(x)\}<y\leq y^*$. Since,
 $$\frac{\partial b_1(x,y) }{\partial x}=L(s\beta K(y - L(-\beta x + 1)) + s x 
 \beta^2 K L + r L (K \beta - 1) + r y) 
 $$
 is linear in $y$ with positive slope,  
and since we are assuming that $y>\max\{k(x),0\}$, 
$$\frac{\partial b_1(x,y)}{\partial x}>\frac{\partial b_1(x,k(x))}{\partial x}=\beta L^2 (r(K-x) + K \beta s x)>0\quad \mbox{for} \quad x<K.$$
Thus, $b_1(x,y)$ is increasing in $x$. 
Furthermore, since  $y>k(x)=L(1-\beta x)$ implies $x>\frac{1}{\beta L}(L-y)$, we have
\begin{align*}
b_1(x,y)&\geq b_1\left(\frac{1}{\beta L}(L-y),y\right)=\frac{r (L - y) (L(\beta K-1)-(\alpha \beta K L - 1 )y)}{\beta}\\
&\geq \frac{r (L - y) (L(\beta K-1)-(\alpha \beta K L - 1 )y^*)}{\beta}=0.
\end{align*}
Hence, $b_1(x,y)\geq 0$, in this case. 
\end{itemize}
This completes the proof of part d), and hence the Lemma is proven.
\end{proof}

Similar to Fig.~\ref{fig:regions2}, we introduce regions of component-wise monotonicity and visualize them in Fig.~\ref{fig:Case2}. Regions $\mathcal{B}_i, i=0,1$  are defined in  Proposition~\ref{Prop:boxes} and  
    \begin {align*}
\mathcal{R}_1 &:=\{(x,y)\, : \, 0<x<x^*, \, \,  y^*<y<h(x) \},\\   
\mathcal{R}_2 &:=\{(x,y)\, : \, 0<x<x^*, \, \,  h(x)<y<k(x) \},\\
\mathcal{R}_3 &:=\{(x,y)\, : \, 0<x<x^*, \, \,  y>k(x) \},\\
\mathcal{R}_4 &:=\{(x,y)\, : \, x^*<x<\frac{1}{\beta}, \, \,  0<y<k(x)\},\\
\mathcal{R}_5 &:=\{(x,y)\, : \, x^*<x<K, \, \,  \max\{0,k(x)\}<y<h(x) \},\\
\mathcal{R}_6 &:=\{(x,y)\, : \, x^*<x, \, \,  \max\{0,h(x)\} <y<y^*\},\\
\mathcal{R}_{i_T} &:=\{(x,y)\in \mathcal{\overline{R}}_{i} \,: \, x>0, \,  y>0, \, (x,y)\neq (x^*,y^*)\}, \ i=1,2,\dots,6.
    \end{align*}

\begin{figure}[h!]
    \centering
    \includegraphics[ width=0.65\linewidth]{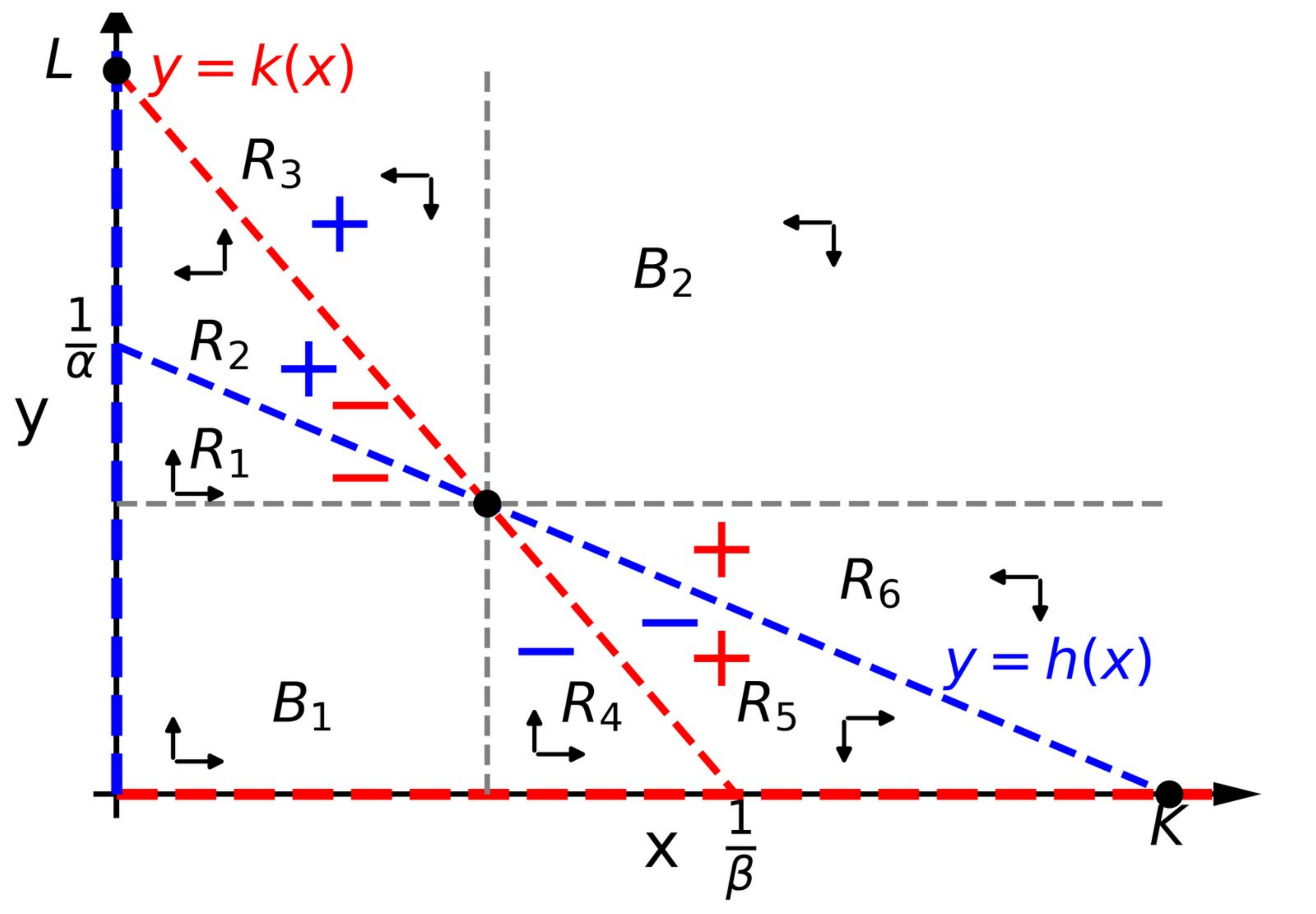}
    \caption{Regions considered with generic nullclines $y=h(x)$ and $y=k(x)$, in the case where $\alpha L>1$ and $\beta K>1$. The colored signs relate to the sign of the corresponding Root-operators $\mathcal{L}_h$ (in blue) and $\mathcal{L}_k$ (in orange). The signs are based on the results in Lemma~\ref{lem:signLhLkCase2}.
   The arrows represent the sign of $x^\Delta$ (horizontal) and $y^\Delta$ (vertical). Each region exhibits component-wise monotonicity.}
    \label{fig:Case2}
\end{figure}

Following the proof of Lemma~\ref{lem:invarianceCase1new}, we can similarly show for this case that the same regions are once more positively invariant. Due to the similarity, the proof is omitted.

\begin{lemma}\label{lem:invarianceCase1}
   Consider \eqref{eq:CompT}  with $\alpha L >1$ and $\beta K>1$. The region $\mathcal{R}_{2_T}$ is positively invariant. The region $\mathcal{R}_{5_T}$ is positively invariant.  
\end{lemma}

Based on Fig.~\ref{fig:Case2} and following the similar arguments as in the proof of Theorem~\ref{thm:Estar_GAS}, we obtain the following result. 

\begin{theorem}\label{thm:Estar_saddle}
     Consider \eqref{eq:CompT}  with positive initial conditions. Assume $\alpha L >1$ and $\beta K>1$. Then, solutions converge either to $E^*=(x^*,y^*)$, $E^*_K$, or $E^*_L$. No solution converges to $E_0^*$. 
\end{theorem}

Note that even though the dynamic phase plane gives insights into the global dynamics, we may not be able to pre-determine the equilibrium that solutions converge to as this may depend on the time point the orbit is initiated at. For example, consider \eqref{eq:CompT} with $\beta=\alpha=0.7$, $r=0.5$, $s=0.3$, $K=2$, and $L=\frac{10}{3}$. 
If we initiate the orbit at $t=2^0=1$ with value of $(x(1),y(1))=(2,0.95)$, then the orbit converges to $(0,L)$ implying the extinction of competitor $x$, see  Fig.~\ref{specialcase}a). However, if we had initiated the orbit at $t=2^7=128$, then the same initial value $(x(128),y(128))=(2,0.95)$ would converge to $(K,0)$ implying the extinction of competitor $y$ instead, see Fig.~\ref{specialcase}b).  The different behavior of their limiting behavior, despite the same initial value $(2,0.95)$, is due to the time-dependence of the Root-curves.

\begin{figure}[h!]
    \centering
    \includegraphics[width=0.4\linewidth]{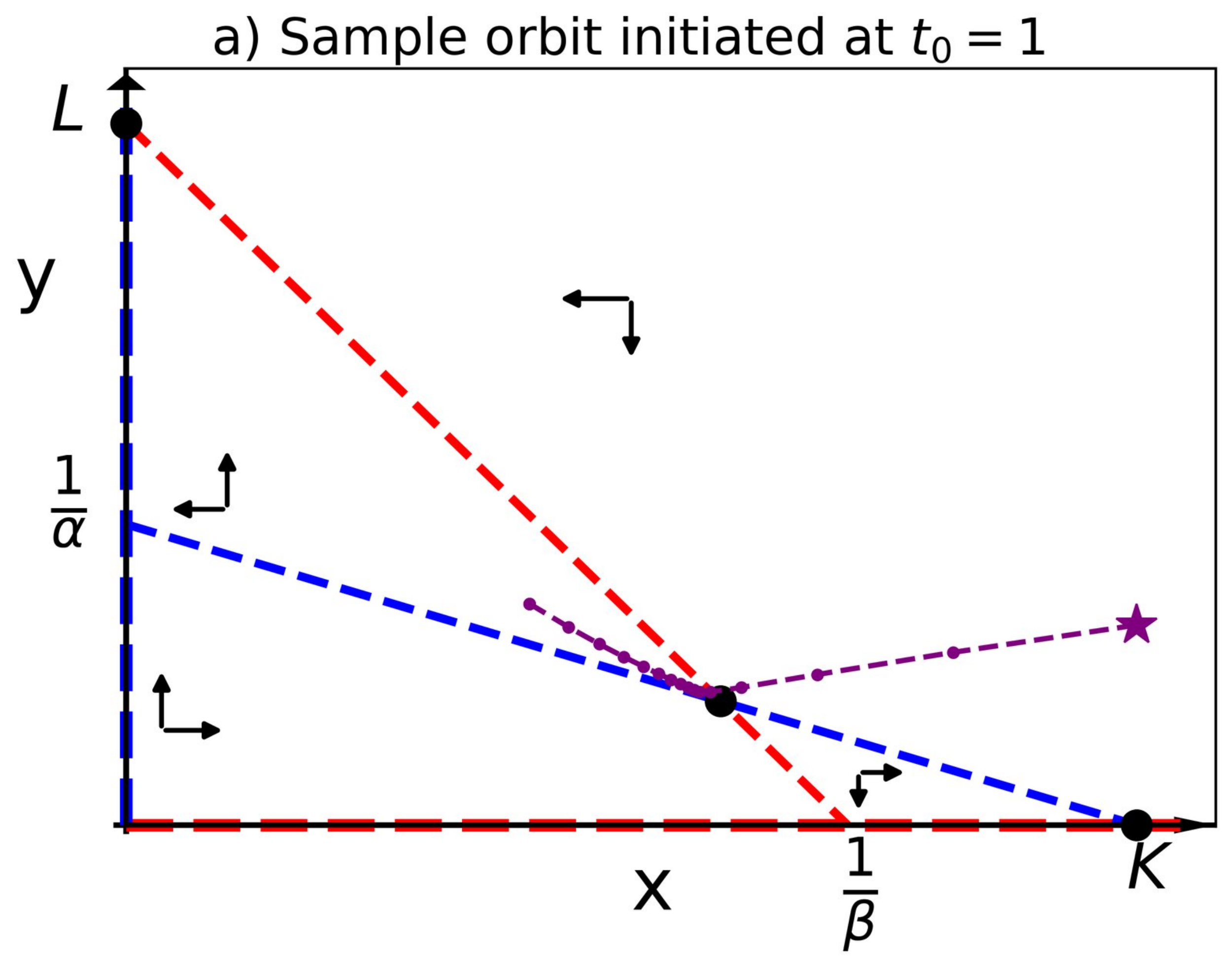}
    \includegraphics[width=0.4\linewidth]{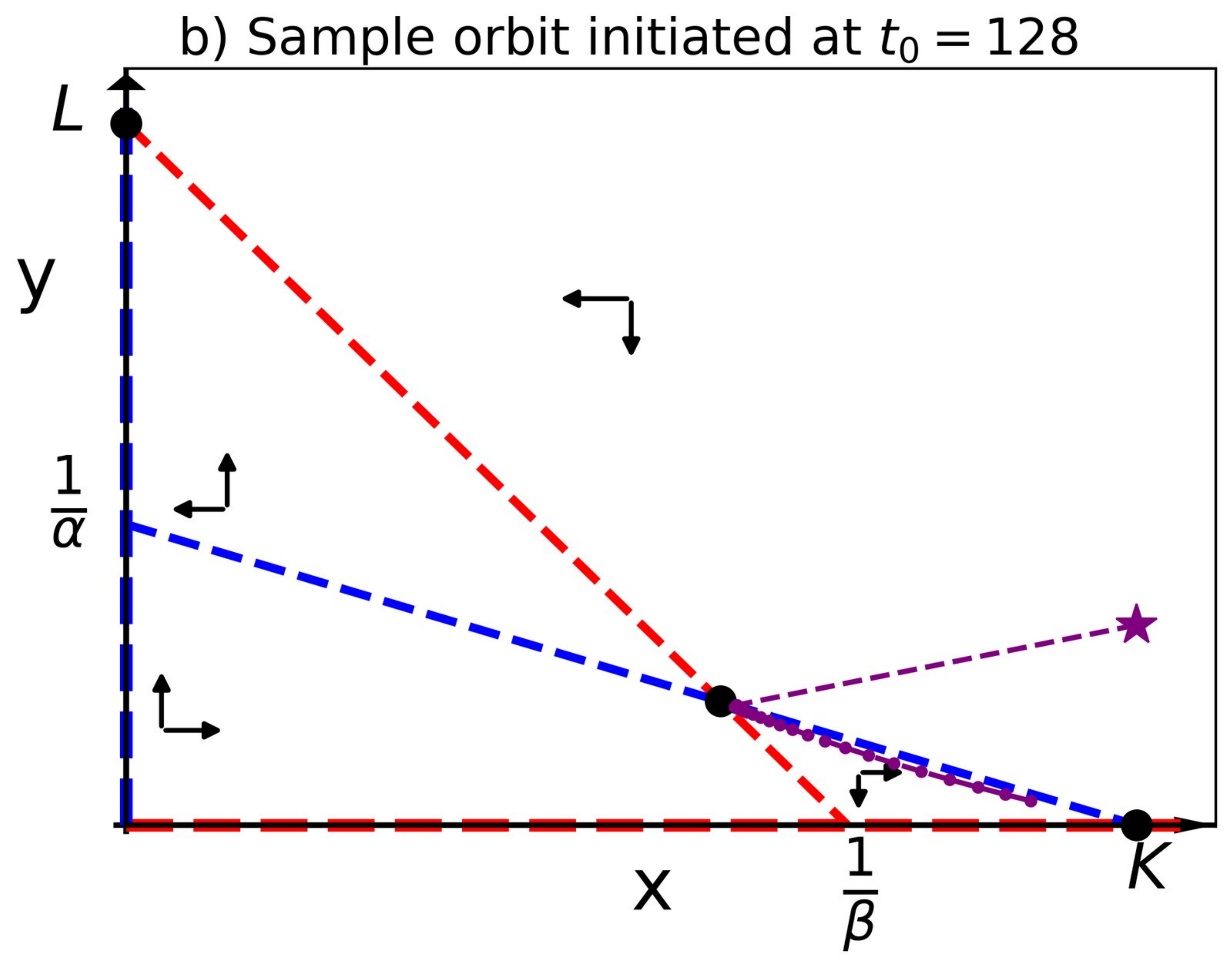}
    \caption{Demonstration that the limiting behavior of an orbit with a fixed initial value may depend on the initial time $t_0$. In a),  $t_0=2^0=1$ and the solution enters $\mathcal{R}_{2T}$ and converges, by Theorem~\ref{thm:Estar_GAS}, to $(0,L)$. In b), $t_0=2^7=128$. Now, the solution enters the positively invariant region $\mathcal{R}_{5T}$ instead and converges, by Theorem~\ref{thm:Estar_GAS}, to $(K,0)$.    }
    \label{specialcase}
\end{figure}

\vspace{3mm}

\subsubsection{Global Dynamics: Infinitely many interior equilibria}

In the case that $\alpha L=1$ and $\beta K=1$,  $h(x)=k(x)$ for all $x$ and all points on  the line segment $y=h(x)$, $0<x<K$ are interior equilibrium points.

\begin{lemma}\label{lem:degen}
   Consider \eqref{eq:CompT} with $\alpha L =1$ and $\beta K=1$.  Then,  
   \begin{itemize}
       \item[a)] $\mathcal{L}_h(t,x,y)=\mathcal{L}_k(t,x,y)$ and
       \item[b)] The Root-sets are independent of $t\in\mathbb{T}$  and satisfy  $\mathcal{R}_h(t)=\mathcal{R}_k(t)=
       \{(x,y):y=h(x)\}=\{(x,y):y=k(x)\}$, i.e., the nullclines and Root-sets all coincide.
   \end{itemize}
\end{lemma}
\begin{proof}
Since $h(x)=k(x)$, by Definition~\ref{def:Rootop}, a) follows immediately. 
To prove b), we note that by \eqref{eq:LhLk}, 
\begin{equation}\label{eq:Lhdegen}
\begin{split}
\mathcal{L}_h(t, x,y)&=\frac{(\beta x + \alpha y-1) (1 + \beta \mu s x + \alpha \mu r y)}{\alpha (1 + \beta \mu r x + \alpha \mu r y) (1 + \beta \mu s x + 
   \alpha \mu s y)},
\end{split}
\end{equation}
which is zero if and only if $\beta x + \alpha y=1$, or, equivalently, $y=\frac{1-\beta x}{\alpha}=h(x)=k(x)$, confirming the claim in b). 
\end{proof}

\begin{lemma}\label{lem:bothsame}
    Consider \eqref{eq:CompT} with $\alpha L =1$ and $\beta K=1$.  Then,  the following hold:
    \begin{itemize}
        \item[a)] $\mathcal{L}_h(t,x,y)>0$ for $y>h(x)$,  $x>0$, and $t\in \T$. 
        \item[b)] $\mathcal{L}_h(t,x,y)<0$ for $y<h(x)$, $x>0$, and $t\in \T$.
    \item[c)]
$\mathcal{L}_k(t,x,y)>0$ for $y>k(x)$, $x>0$, and $t\in \T$.
\item[d)] $\mathcal{L}_k(t,x,y)<0$ for $y<k(x)$,  $x>0$, and $t\in \T$.
\end{itemize}
\end{lemma}

\begin{proof}
    By Lemma~\ref{lem:degen}, $\mathcal{L}_h=\mathcal{L}_k$. Thus, {\it c)} follows immediately from {\it a)} and {\it d)} follows from {\it b)}.  
By \eqref{eq:Lhdegen}, $\mathcal{L}_h(t,x,y)>0$ if and only if $\beta x + \alpha y -1>0$, i.e., $y>\frac{1-\beta x}{\alpha}=h(x)$, confirming {\it a)}. Reversing the inequality proves b).
\end{proof}

\begin{theorem}
     Consider \eqref{eq:CompT} with $\alpha L =1$ and $\beta K=1$. Then any point on $y=h(x)$ for $x\in [0,K]$ is a stable equilibrium. 
\end{theorem}

\begin{proof}
     By \eqref{eq:xy_Delta},
     it follows immediately that for $K=\frac{1}{\beta}$ and $L=\frac{1}{\alpha}$, $x^\Delta = 0$ and $y^\Delta = 0$ if $\alpha y+\beta x=1$. Thus, any point on the nullcline $y=h(x)=k(x)$ is an equilibrium. 

     Furthermore, by Lemma~\ref{lem:bothsame} and Lemma~\ref{lem:degen}, if $y( t_0)>h(x(t_0))$  for some $t_0\in \T$, then $y(t)\geq h(x(t))$ for all $t \in [t_0, \infty)_\T$. Similarly, if $y( t_0)<h(x( t_0))$, then $y(t)\leq h(x(t))$ for all $t \in  [t_0, \infty)_\T$. 
     
       Select an arbitrary equilibrium point $(x^*, h(x^*))$ and   any open set $U$ containing $(x^*, h(x^*))$. 
 There exists $\epsilon>0$ such that  $U_{\epsilon}=\{(x,y)\in (0,\infty)^2 \, \colon \, \|(x,y)-(x^*, h(x^*))\|_{\infty} < \epsilon\} \subset U$.  Take the rectangle $V\subseteq U_{\epsilon}$ such that both its upper left corner and its lower right corner are on the nullcline $y=h(x)$.  Note that if  $x^*=0$ or $x^*=K$, select $U$ containing $(x^*, h(x^*))$, open relative to $[0,\infty]^2$. If $x^*=0$,
     let  $U_{\epsilon}=\{(x,y)\in (0,\infty)\times [0,\infty)\, \colon \, \|(x,y)-(K,0)\|_{\infty} < \epsilon\} \subset U$ and take the rectangle $V\subseteq U_\epsilon$ with its upper left corner on the nullcline and if $x^*=0$, let $U_{\epsilon}=\{(x,y)\in [0,\infty)\times (0,\infty)\, \colon \, \|(x,y)-(0,L)\|_{\infty} < \epsilon\} \subset U$ and take a rectangle $V\subseteq U_\epsilon$ with its lower right corner on the nullcline. (If $K=L$, take $V=U_{\epsilon}$.) From the direction of the component-wise monotonicity, the rectangle $V$ is positively invariant in all cases.
\end{proof}

\section{Conclusion}

In this work, we introduced the time scales analogue of the classical Lotka--Volterra two species competition model. In the special case when the time scale is the continuous time domain, our model is consistent with the well-studied Lotka--Volterra competition model. In the case when the time scale is the discrete time domain, our model is consistent with the discrete analogue of the Lotka--Volterra competition model, also known as Leslie--Gower model. 

We further introduced a dynamic phase plane analysis, which extends the augmented phase plane analysis introduced in \cite{StWo} for the special case of the discrete time domain. Using the dynamic phase plane analysis, we proved the global dynamics of the time scales Lotka--Volterra competition model. More precisely, we showed that if $\alpha L-1$ and $\beta K-1$ have opposite signs, then competitive exclusion applies and only one species persists. More precisely, if $\alpha L>1$ and $\beta K<1$, then competitor $y$ is the sole survivor. Instead, if $\alpha L<1$ and $\beta K>1$, then competitor $x$ is the sole survivor. If $\alpha L-1$ and $\beta K-1$ are both positive, then the system exhibits bistability. Instead, if both are negative, then solutions converge to the unique coexistence equilibrium. In the special case where both terms are zero, we used our dynamic phase plane analysis to show that any point on the nullcline $y=h(x)$ for $x\in [0,K]$ is a stable equilibrium. Thus, we were able to conclude the global dynamics of \eqref{eq:CompT} that turned out to be identical to the conditions known for the continuous and discrete case. This supports the claim that our introduced model \eqref{eq:CompT} is indeed the time scales analogue of the popular two species Lotka--Volterra competition model. 

Note that our dynamic phase plane analysis relies on the signs of the {\it Root-operators}, as defined in Section 2.2, rather than the precise location of the points in the {\it Root-set}. This is particularly useful in the study of time scales models with a non-constant graininess. In that case, the {\it Root-sets}, associated with a nullcline, change in time and are therefore hard to locate in general. Instead, we used that a {\it Root-set}, at any time $t$, was either above or below its corresponding (time-independent) nullcline, allowing us to conclude the sign of the {\it Root-operator} in a large region of the state space. This ultimately led us to provide time-independent information and conclude the global dynamics. Note that the here introduced dynamic phase plane analysis is identical to the classical phase plane analysis in the special case when $\T=\R$, as in this case, the {\it Root-sets} are the same as the nullclines.

Future work may include the application of the introduced method of a dynamic phase plane analysis to other planar time scales models. We remark that the method can also be used to identify regions of oscillatory behavior, similar to the discussion in \cite{StWo} for the special case of $\T=\Z$. To study the stability of periodic orbits, one can also apply the method to compositions of a planar map. Another option is the extension of the method to higher dimensions, albeit their visualization becomes harder beyond three dimensions. Interesting would also be if there exist properties of the dynamic phase plane that  identify chaotic regions of a system.

{\color{black}We also remark that for practicality, if we consider a dynamic equation of the general form $x^\Delta = f(x)$ one might face difficulties actually calculating the solution at a  point $\widehat{t}\in \T$. Take for example an initial condition $x(\widehat{t}_0)=x_0\in \R$ and consider a sequence of isolated points $\{\widehat{t}_j\}_0^\infty \in \T$  such that $\widehat{t}_j<\widehat{t}_{j+1}$ for all $j\in \mathbb{N}$ and $\lim_{j\to \infty}\widehat{t}_j=\widehat{t}$. Then the recursive approach  $x(\widehat{t}_{j+1})=x(\widehat{t}_j)+\mu(\widehat{t}_j)f(x(\widehat{t}_j))$ is failing to provide the value of $x(\widehat{t})$. 
It is an interesting question to explore applications of such a time scale and its dynamic process and numerical methods to implement (approximations) of solutions.

}

\section*{Acknowledgments}
This research was supported in part by grants from
the NSF (DMS-2235451) and Simons Foundation (MPS-NITMB-00005320) to the NSF-Simons
National Institute for Theory and Mathematics in Biology (NITMB). 
GSKW was partially supported by a Natural Science and Engineering Council of Canada (NSERC) Discovery grant (RGPIN-2022-05067).

\appendix
\renewcommand{\thesection}{\Alph{section}}
\renewcommand{\thesubsection}{\Alph{section}.\arabic{subsection}}
\renewcommand{\theequation}{A\arabic{equation}}
\setcounter{equation}{0}

\setcounter{theorem}{0}

\renewcommand{\thetheorem}{A\arabic{theorem}}

\section{Appendix: Preliminaries of Time Scales Theory} \label{sec:appendix}

For the reader unfamiliar with time scales, we provide basic definitions and theorems of time scales and refer the interested reader to \cite{Bohner1} and \cite{Bohner2}. 
A time scale $\mathbb{T}$ is a closed nonempty subset of $\mathbb{R}$. For $t\in \mathbb{T}$, the forward jump operator $\sigma:\mathbb{T}\to \mathbb{T}$ is defined by $\sigma(t)=\inf\{s \in \mathbb{T}:  s \, > t\}$ and for $f\colon \T\to \R$, we define   $f^{\sigma}:=f(\sigma(t))$.   Similarly, a backward jump operator $\rho \colon \mathbb{T} \rightarrow \mathbb{T}$ is defined by $\rho(t)=\sup \{ s \in \mathbb{T} \colon s<t\}$. We define the graininess (or ``stepsize") function $\mu \colon \mathbb{T} \rightarrow \mathbb{R}_0^+$ by $\mu(t)=\sigma(t)-t$. 
 If $\sigma(t)>t$, then we say that $t$ is right-scattered  and if $\sigma(t)=0$ we say that $t$ is right dense. If $\rho(t)<t$,  we say that $t$ is left-scattered and if $\rho(t)=0$, we say that $t$ is left dense.  

A function $p:\mathbb{T} \to \mathbb{R}$ is called rd-continuous provided $p$ is continuous at $t$ for all right-dense points $t$ and the left-sided limit exists for all left-dense points $t$. The set of real-valued rd-continuous functions $f \colon \mathbb{T} \rightarrow \mathbb{R}$ is denoted by $C_{\rm rd} = C_{\rm rd}(\mathbb{T}) = C_{\rm rd}(\mathbb{T}, \mathbb{R})$. We say that $p$ is regressive if for all $t \in \mathbb{T}$, $1+\mu(t)p(t) \neq 0$. The set of real-valued  regressive and rd-continuous functions is denoted by $\mathcal{R}=\mathcal{R}(\mathbb{T})=\mathcal{R}(\mathbb{T},\mathbb{R})$. We denote the set of functions $f\in \mathcal{R}$ such that $1+\mu(t)p(t)>0$ for all $t\in \mathbb{T}^\kappa$, where $\mathbb{T}^\kappa=\T$ if $\T$ is unbounded and $\mathbb{T}^\kappa=\T\backslash\{T_M\}$ if $\max \T=T_M$, by $\mathcal{R}^+$ and refer to it as ``positively regressive''.

The ``circle plus" addition on $\mathcal{R}$ is defined by $\left(p \oplus q\right)(t) =p(t) + q(t) + \mu(t) p(t) q(t)$, and the ``circle minus" subtraction is $\left(p \ominus q\right)(t) = \frac{p(t)-q(t)}{1+\mu(t) q(t)}$.

The $\Delta$-derivative on  time scales for $t\in \T^\kappa$, where $\T^\kappa=\T$  and $\T^\kappa=\T\backslash\{T_M\}$, if $\T$ has a left-scattered maximum, is defined as follows.  
\begin{definition}[See \protect{\cite[Definition 1.10]{Bohner1}}]
    Assume $f\colon \T\to \R$ and let $t\in \mathbb{T}^\kappa$. Then, $f^\Delta(t)$ is the number (provided it exists) such that for any $\epsilon>0$, there exists $\delta>0$ such that 
    $$|f(\sigma(t))-f(s)-f^\Delta (t) [\sigma(t)-s]|\leq \epsilon |\sigma(t)-s|,$$
    for all $s\in U=(t-\delta, t+\delta)\cap \T$. The value $f^\Delta(t)$ is called the delta derivative of $f$ at $t$.
\end{definition}

\begin{remark}\label{rem:fsigma}
    For a differentiable function $f$, $f^\sigma=f+\mu f^\Delta$ (see \cite{Bohner1}).
\end{remark}

\begin{theorem}[See \protect{\cite[Theorem~1.70]{Bohner1}}]\label{Thm:anti}
Let $f$ be regulated. Then there exists a function $F$ which is pre-differentiable with region of differentiation $D$ such that 
$F^\Delta(t)=f(t)$ for all $t\in D$. 
\end{theorem}

In this case, $F$ is called the pre-antiderivative which can be used to formulate a time scales analogue of the Fundamental Theorem of Calculus. 

\begin{definition}[See \protect{\cite[Definition~1.71]{Bohner1}}]\label{def:Cauchy}
    Assume $f\colon \T\to \R$ is a regulated function. Any function $F$ as in Theorem~\ref{Thm:anti} is called a pre-antiderivatve of $f$. We define the indefinite integral of a regulated function $f$ by 
    $$\int f(t)\Delta t = F(t)+C,$$
    where $C$ is an arbitrary constant and $F$ s a pre-antiderivative of $f$. The Cauchy integral is then defined by 
    $$\int_r^s f(\tau)\, \Delta \tau = F(s)-F(r)$$
    for all $r,s\in \T$.
\end{definition}

If $p \in \mathcal{R}$ and $t_0 \in \mathbb{T}$, then the initial value problem
\begin{equation*}
y^{\Delta}=p(t)y, \quad y(t_0)=1,
\end{equation*}
 possesses a unique solution, called the dynamic exponential function, denoted by $e_p(t,t_0)$.

The exponential function can be calculated, for $p\in \mathcal{R}$, from the fact that  
 $$e_p(t,s)=\begin{cases}
\exp\left(\int_s^t \,\frac{ {\rm Log}(1+\mu(\tau) p(\tau)) }{\mu(\tau)} \, \Delta \tau \right)& \mbox{if} \quad \mu(\tau)\neq 0,\\
\exp\left(\int_s^t \, p(\tau) \, \Delta \tau \right) & \mbox{if} \quad \mu(\tau)=0.
\end{cases}$$

Some useful properties of the exponential function are given below, see \cite[Theorem~2.36 \& 2.44]{Bohner1}. 
\begin{remark}\label{rem:exp_Properties} 
If $p, q\in \mathcal{R}$, then 
\begin{itemize}
\item[a)] $e_0(t,s)=1$ and $e_p(t,t)=1$
\item[b)] $e_p(\sigma(t),s)=(1+\mu(t)p(t))e_p(t,s)$
\item[c)] $e_{\ominus p}(t,t_0)= \frac{1}{e_p(t,t_0)}=e_p(t_0,t).$
\item[d)] $e_p(t,s)e_p(s,r)=e_p(t,r)$
\item[e)] $e_p(t,s)e_q(t,s)=e_{p\oplus q}(t,s)$
\item[f)] If $p\in \mathcal{R}^+$, then $e_p(t,t_0)>0$ for all $t\in \T$.
\end{itemize}
\end{remark}

\begin{theorem}[See \protect{\cite[Theorem~2.39]{Bohner1}}]\label{thm:int_exp}
    If $p\in \mathcal{R}$ and $a, b, c, \in \T$, then
    \begin{equation}\label{eq:integrale}
\int_a^b \, p(t) e_p(c,\sigma(t)) \Delta t= e_p(c,a)-e_p(c,b).
\end{equation}
\end{theorem}

The exponential function can also be used to solve a linear, nonhomogeneous, dynamic equation. 

\begin{theorem}[Variation of Constants, see \protect{\cite[Theorem 2.1]{Bohner1}}]\label{thm:VoC}
Assume that  $p \in \mathcal{R}$ and $f \in  C_{\rm rd}.$ Let $t_0 \in \mathbb{T}$ and $y_0 \in \mathbb{R}$. The unique solution of the IVP
\begin{equation*}
y^{\Delta} = p(t)y+f(t), \quad y(t_0)=y_0,
\end{equation*}
is given by 
\begin{equation}\label{VoC}
y(t)=e_p(t,t_0)y_0 + \int_{t_0}^{t} e_p(t,\sigma(s))f(s)\, \Delta s.
\end{equation}\label{variationofconstants}
\end{theorem}

\begin{theorem}[Gronwall's Inequality, see 
\protect{\cite[Theorem 6.1]{Bohner1}}]\label{Thm:Gronwall} 
Let $y, f \in C_{{\rm rd}}$ and $p\in \mathcal{R}^+$.  Then 
$$ y^{\Delta}(t)\leq p(t) y(t) + f(t) \quad \mbox{for all} \quad t \in \T$$
implies that if $t_0\in \T$
$$ y(t)\leq y(t_0)e_p(t,t_0) + \int_{t_0}^t \, e_p(t,\sigma(\tau))f(\tau) \Delta \tau \quad \mbox{for all} \quad t \in \T.$$
\end{theorem}

\begin{lemma}[See 
\protect{\cite[Lemma 2]{MR2145447}}]\label{Lem2in1} 

For nonnegative $p\colon \T\to \R$ with $-p\in \mathcal{R}^+$, we have the following inequalities:
\begin{equation*}
    1-\int_{t_0}^t p(s)\, \Delta s \, \leq \, e_{-p}(t,t_0)\, \leq \, \exp\left\{\int_{t_0}^t p(s)\, \Delta s\right\}, \qquad \qquad \forall\, t\geq t_0. 
\end{equation*}

\end{lemma}

The following theorem follows directly from the proof of \protect{\cite[Theorem 1.76]{Bohner1}}.

\begin{theorem}\label{thm:monotonicity} 
   Assume that  $g\in  C_{{\rm rd}}$ and $g:\T \to \mathbb{R}$.  If there exist $T_1< T_2 \in  \T$ such that $g^{\Delta} > 0 (<0)$  for all $t\in [T_1,T_2]\cap \T$, then 
    $g$ is strictly increasing (decreasing) for all $t \in [T_1,T_2]\cap \T$.  
\begin{proof} This follows using the proof  of Theorem 1.76 in \cite{Bohner1}.
\end{proof}
    
    % \in \T$ with  $t_0\leq T_1<T_2$ such that $g^{\Delta} > 0 (<0)$ for all $t\geq \T$, then $g$ is strictly increasing (decreasing) for all $t \in(T_1,T_2)$, t \in \T$.    
\end{theorem}

\begin{remark}  We refer to  $g$ as component-wise monotonic for $t\in[T_1,T_2] \cap \T$ if  $g$ satisfies Theorem~\ref{thm:monotonicity} for $t \in [T_1,T_2] \cap \T$ with the obvious implications for the direction field of a two dimensional  phase plane.  
\end{remark}

Using the Gronwall inequality and properties of the exponential function, we can easily show the following result. 

\begin{lemma}\label{lem:convergence} 
Assume $\T$ is unbounded above and let $\epsilon >0$. If there exists $T \in \T$ such that 
    $z^\Delta(t) \geq \epsilon$ for all   $t>T$, then $\lim_{t \to \infty, t\in \T} z(t) = \infty$ and 
 if there exists $T\in\T$ such that  $w^\Delta(t) \leq -\epsilon$ for all  $t>T$, then 
 $\lim_{t \to \infty, t\in \T} w(t) = -\infty$. 
 \end{lemma}

\begin{proof}  
Assume  $z^\Delta (t)\geq \epsilon>0$ for  all   $t>T$.  Since $z$ is the pre-antiderivative of $z^\Delta$,  by  Definition~\ref{def:Cauchy},  for $t_0, t\in \T$ and $T<t_0<t$,  $$z(t)=z(t_0)+\int_{t_0}^t z^\Delta(\tau) \, \Delta \tau\geq z(t_0)+\epsilon (t-t_0).$$
Thus, for $\epsilon>0$, $\lim_{t\to \infty, t\in \T} z(t)=\infty$.

Similarly, if $w^\Delta \leq -\epsilon$ for  $t >T$,   after using the identity in  \cite[Theorem~1.76]{Bohner1}, for $T<t_0<t$ and $t_0,t\in T$, 
$$w(t)=w(t_0)+\int_{t_0}^t w^\Delta(\tau) \, \Delta \tau \leq w(t_0)-\epsilon (t-t_0).$$
Again, since $\epsilon>0$, $\lim_{t\to \infty, t \in \T} w(t)=-\infty$. 
\end{proof}

\subsection{Solution of a single species logistic model \eqref{eq:logsingle}}\label{SolSingleT}

\begin{proposition}\label{prop:1D_model}
The single species model:
\begin{equation}\label{eq:single_species} x^\Delta = -(\ominus r) x^\sigma \left(1-\frac{x}{K} \right), \qquad t\in \mathbb{T}
\end{equation}
with parameters $r>0$ and  $K>0$  has unique solution
$$x(t)=  \frac{e_r(t,t_0)K x(t_0)}{K+x(t_0)(e_r(t,t_0)-1)}\geq 0,$$
with equality if and only if $x(t_0)=0$.
\end{proposition}
\begin{proof}
Since,  $\overline{x}=0$ is an equilibrium,  if $x(t_0)=0$, then $x(t)=0$ for all $t\geq t_0, \, t\in \T$. 
Let $u=\frac{1}{x}$ and assume that $x(t_0)>0$.
Then,
\begin{equation}\label{eq:1D_unique}
u^{\Delta} = (\ominus r) u - (\ominus r)\frac{1}{K}. 
\end{equation}
Using the variation of constants formula given in Theorem~\eqref{thm:VoC} with $p(s)=(\ominus r)(s)$, Theorem~\ref{thm:int_exp},  and Remark~\ref{rem:exp_Properties}, the unique solution of \eqref{eq:1D_unique} is given by:
\begin{align*} u(t)&=e_{\ominus r}(t,t_0)u(t_0) - 
\frac{1}{K}\int_{t_0}^{t} e_{\ominus r}(t,\sigma(s))((\ominus r)(s)) \, \Delta s \\
 & = e_{\ominus r}(t,t_0)u(t_0) -  \frac{1}{K}(e_{\ominus r}(t,t_0)- e_{\ominus r}(t,t))\\
 &= \frac{1}{e_r(t,t_0)}u(t_0)+\frac{1}{K}\left(1-\frac{1}{e_r(t,t_0)}\right).
\end{align*}
Replacing $u(t)$ by $\frac{1}{x(t)}$ and simplifying, the result follows noting that, $e_r(t,t_0)>1$, because $r>0$. 
\end{proof}

\subsection{Relation of logistic expressions}\label{Ap:Marcia}

In \cite{AkinBohner}, see also \cite[p.~18]{Bohner2}, the authors introduce two different logistic models, 
\begin{align}
y^\Delta = \left( \ominus (p+fy)\right)y \label{eq:LogB1}\\ 
x^\Delta = \left( p\ominus (fx)\right)x, \label{eq:LogB2}
\end{align}
where both can be solved explicitly as described in \cite{AkinBohner}.

First, notice that \eqref{eq:logsingle}, using that $z^\sigma = z+\mu z^\Delta$ for delta-differentiable functions $z$, 
$$z^\Delta = -(\ominus r)z^\sigma \left(1-\frac{z}{K}\right)=-(\ominus r)(z+\mu z^\Delta)\left(1-\frac{z}{K}\right)$$
so that, after rearranging, \eqref{eq:logsingle} is equivalent to 
$$z^\Delta \left(1+\mu (\ominus r)\left(1-\frac{z}{K}\right) \right)= -(\ominus r)z \left(1-\frac{z}{K}\right),$$
i.e., 
\begin{align*}
z^\Delta &= \frac{-(\ominus r)}{(1+\mu (\ominus r)\left(1-\frac{z}{K}\right)}z \left(1-\frac{z}{K}\right)= \frac{-(\ominus r)\left(1-\frac{z}{K}\right)}{(1+\mu (\ominus r)\left(1-\frac{z}{K}\right)} z
=\left\{\ominus \left\{ (\ominus r)\left(1-\frac{z}{K}\right)\right\}\right\}z,
\end{align*}
where we used that  $\ominus q=\frac{-q}{1+\mu q}$ for $q=(\ominus r)\left(1-\frac{z}{K}\right)$.
Since 
$$  (\ominus r)\left(1-\frac{z}{K}\right)= (\ominus r)-(\ominus r)\frac{z}{K}=p+fz,$$
for $p=\ominus r$ and $f=\frac{-(\ominus r)}{K}$. 

Note that \eqref{eq:logsingle} is also equivalent to the proposed logistic model in \cite{Marcia} since 
\begin{align*}
z^\Delta \stackrel{\eqref{eq:logsingle}}{=} -(\ominus r) z^\sigma \left(1-\frac{z}{K}\right)=-(\ominus r)z^\sigma + (\ominus r)z^\sigma \frac{z}{K}
= -(\ominus r)(z+\mu z^\Delta) + (\ominus r)z^\sigma \frac{z}{K},   
\end{align*}
so that after rearranging terms, we get that \eqref{eq:logsingle} is equivalent to 
\begin{equation*}
z^\Delta (1+\mu (\ominus r))= -(\ominus r)z+ (\ominus r)z^\sigma \frac{z}{K},  
\end{equation*}
i.e., 
\begin{align*}
    z^\Delta = \frac{-(\ominus r)}{(1+\mu (\ominus r))}z+ \frac{(\ominus r)}{(1+\mu (\ominus r))}z^\sigma \frac{z}{K}=rz-rz^\sigma \frac{z}{K}=rz\left(1-\frac{z^\sigma}{K}\right),
\end{align*}
where we used that 
$$ \frac{(\ominus r)}{(1+\mu (\ominus r))} = \ominus (-(\ominus r))=-r.$$
Thus, \eqref{eq:logsingle} and the logistic model proposed in \cite{Marcia} are identical and both are special cases of \eqref{eq:LogB1}. 

Finally, note that the Beverton--Holt model introduced in \cite{BohnerWarth} of the form 
\begin{equation}\label{BHWarth}
z^\Delta = \alpha z^\sigma \left(1-\frac{z}{K}\right)
\end{equation}
is identical to \eqref{eq:logsingle} for $\alpha=-(\ominus r)$.

\section{Proof of Theorem~\ref{thm:bound}}\label{A:Proofbnd}

\begin{proof}
To show that solutions with nonnegative initial conditions are bounded,  note that since  $- (\ominus r)=\frac{r}{1+\mu r}>0$, \,  $x^{\sigma}\geq 0$ by \eqref{eq:xy_sigma}, and $y(t)\geq 0$ by Proposition~\ref{prop:nonneg}. Then 
$$x^\Delta = -(\ominus r)x^\sigma \left(1-\frac{x}{K}-\alpha y\right)\leq -(\ominus r)x^\sigma \left(1-\frac{x}{K}\right).$$

Let $u=\frac{1}{x}$. Then, $u^{\Delta}= -\frac{x^{\Delta}}{x x^{\sigma}}$ and since $-(\ominus r)>0$, 
the above inequality implies that
$$u^\Delta \geq  (\ominus r)u -(\ominus r)\frac{1}{K}.$$
Then, for $v=-u$, we have  
$$v^\Delta \leq  (\ominus r)v +(\ominus r)\frac{1}{K}.$$
Since   $1+ \mu  (\ominus r)=1+\mu \frac{-r}{1+\mu r}=\frac{1}{1+\mu r} >0$, $(\ominus r) \in \mathcal{R}^+$ and $(\ominus r)\frac{1}{K}\in \mathcal{C}_{rd}$,   by  Theorem~\ref{Thm:Gronwall}, 
\begin{align*}
v(t)&\leq e_{(\ominus r)}(t,t_0)v(t_0)+\int_{t_0}^t e_{(\ominus r)}(t,\sigma(s)) (\ominus r)\frac{1}{K}\Delta s\\
&\stackrel{\eqref{eq:integrale}}{=}e_{(\ominus r)}(t,t_0)v(t_0)+\frac{1}{K}\left(e_{(\ominus r)}(t,t_0)-1\right).
\end{align*}
Thus, 
\begin{equation}\label{usol}
u(t)\geq \frac{1}{K}(1-e_{(\ominus r)}(t,t_0))+e_{(\ominus r)}(t,t_0)u(t_0).
\end{equation}
Since $(\ominus r)\in \mathcal{R}^+$, $e_{\ominus r}(t,t_0)>0$, (see Remark~\ref{rem:exp_Properties} f)). Also, $e_{(\ominus r)}(t,t_0)=e_{-p}(t,t_0)$ for $p:=-(\ominus r)=\frac{r}{1+\mu(t)r}>0$. Since $p$ is nonnegative and $-p\in \mathcal{R}^+$, we apply Lemma~\ref{Lem2in1} to conclude that for $t\in (t_0,\infty)_\T$,  
$$e_{\ominus r}(t,t_0)=e_{-p}(t,t_0)\leq e^{-\int_{t_0}^t p(s)\Delta s
}=e^{\int_{t_0}^t (\ominus r)\Delta s
}=e^{-\int_{t_0}^t \frac{r}{1+\mu(s)r}\Delta s
}<1.$$

Thus, $1-e_{(\ominus r)}(t,t_0)>0$ and we obtain,  after substituting back for $x(t)$, again  using   $u(t)=\frac{1}{x(t)}$ in \eqref{usol}, 
$$x(t)\leq \frac{Kx(t_0)}{x(t_0)(1-e_{(\ominus r)}(t,t_0))+Ke_{(\ominus r)}(t,t_0)}.$$
Since $e_{(\ominus r)}(t,t_0)\in [0,1]$, the solution is bounded.

The argument to show that $y(t)$ is bounded is similar. 
\end{proof}


\begin{thebibliography}{99}

\bibitem{Agarwal_2002}
R. Agarwal and  M. Bohner and D. O'Reagan and  A. Peterson. {\em Dynamic Equations on time scales: a survey}. J. Comp. Appl. Math., vol.~141, pp.~1--26, 2002.

\bibitem{Edelstein1988}
Edelstein-Keshet, L.. {\em Mathematical Models in Biology}. Society for Industrial and Applied Mathematics (SIAM, 3600 Market Street, Floor 6, Philadelphia, PA 19104), 1988.

\bibitem{Hutchinson1978}
Hutchinson, G. E.. {\em An Introduction to Population Ecology}. Yale University Press, 1978.

\bibitem{mathematica}
Wolfram Research Inc. {\em Mathematica 12.2}, 2020.

\bibitem{brauer2011}
Brauer, F. and Castillo-Chavez, C.. {\em Mathematical Models in Population Biology and Epidemiology}. Springer New York, 2011.

\bibitem{Galor2007}
Galor, O.. {\em Discrete Dynamical Systems}.. Springer-Verlag Berlin Heidelberg, 2007.

\bibitem{Braun1979}
Braun, M. {\em Differential Equations and Their Applications}. Springer-Verlag New York, 1979.

\bibitem{Pielou1969}
Pielou, E. C. {\em An Introduction to Mathematical Ecology}. Wiley-Interscience, New York, 1969.

\bibitem{GailMary2011}
Ballyk, M. M. and Wolkowicz, G. S. K. {\em Classical and resource-based competition: a unifying graphical approach}. J. Math. Biol., vol.~62, pp.~81--109, 2011.

\bibitem{Gause1934}
Gause, G. F.. {\em The Struggle for Existence}. Hafner Publishing, New York, 1934.

\bibitem{StWoBo1}
Streipert, S. H. and Wolkowicz, G. S. K. and Bohner, M. {\em An alternative discrete predator-prey model}. Bulletin of Mathematical Biology, vol.~84, no.~67, pp.~1--34, 2022.

\bibitem{StWo2}
Streipert, S. H. and Wolkowicz, G. S. K. {\em A method to derive discrete population models}. Springer Proceedings in Mathematics \& Statistics, vol.~Advances in Discrete Dynamical Systems, Difference Equations, and Applications, 2022.

\bibitem{MATLAB:2020b}
MATLAB. {\em version R2020b}. The MathWorks Inc., 2020.

\bibitem{Roeger2013}
Roeger, Lih-Ing Wu and Lahodny, Jr., Glenn. {\em Dynamically consistent discrete Lotka--Volterra Competition Systems}. J. Differ. Equ. Appl., vol.~19, no.~2, pp.~191--200, 2013.

\bibitem{Roeger2009}
Roeger, Lih-Ing W. and Gelca, Razvan. {\em Dynamically consistent discrete-time Lotka--Volterra competition models}. Discrete Contin. Dyn. Syst., no.~Dynamical systems, differential equations and applications.               7th AIMS Conference, suppl., pp.~650--658, 2009.

\bibitem{Smith1998}
Hal L. Smith. {\em Planar competitive and cooperative difference equations}. Journal of Difference Equations and Applications, vol.~3, no.~5--6, pp.~335--357, 1998.

\bibitem{Liu2001}
Liu, P. and Elaydi, S. N.. {\em Discrete competitive and cooperative models of Lotka–-Volterra type}. Journal of Computational Analysis and Applications, vol.~3, no.~1, pp.~53--73, 2001.

\bibitem{Cushing2004}
Cushing, J. M. and Levarge, S. and Chitnis, N. and Henson, S. M.. {\em Some discrete competitive models and the competitive exclusion principle}. Journal of Difference Equations and Applications, vol.~10, no.~13--15, pp.~1139--1151, 2004.

\bibitem{FRANKE1991111}
John E. Franke and Abdul-Aziz Yakubu. {\em Global attractors in competitive systems}. Nonlinear Analysis: Theory, Methods \& Applications, vol.~16, no.~2, pp.~111-129, 1991.

\bibitem{Baigent2016}
Baigent, Stephen. {\em Convexity of the carrying simplex for discrete-time planar competitive Kolmogorov systems}. Journal of Difference Equations and Applications, vol.~22, no.~5, pp.~609-622. Taylor \& Francis, 2016.

\bibitem{Mickens1}
Mickens, Ronald E.. {\em Exact solutions to a finite-difference model of a nonlinear reaction-advection equation: Implications for numerical analysis}. Numerical Methods for Partial Differential Equations, vol.~5, no.~4, pp.~313-325, 1989.

\bibitem{Mickens2}
Mickens, Ronald E.. {\em Genesis of elementary numerical instabilities in               finite-difference models of ordinary differential equations}. Proceedings of Dynamic Systems and Applications, no.~1308304, pp.~251--257. Dynamic, Atlanta, GA, 1994.

\bibitem{Mickens3}
Mickens, Ronald E. and Smith, Arthur. {\em Finite-difference models of ordinary differential equations:               influence of denominator functions}. J. Franklin Inst., vol.~327, no.~1, pp.~143--149, 1990.

\bibitem{Mickens4}
Mickens, Ronald E.. {\em Nonstandard finite difference models of differential               equations}., no.~1275372, pp.~xii+249. World Scientific Publishing Co., Inc., River Edge, NJ, 1994.

\bibitem{Hadeler1994}
Gouz\'{e} J.-L. and Hadeler, K. P.. {\em Monotone flows and order intervals}. Nonlinear World, vol.~1, no.~1, pp.~23--34, 1994.

\bibitem{Leslie1958}
Leslie, P. H.. {\em A stochastic model for studying the properties of certain               biological systems by numerical methods}. Biometrika, vol.~45, no.~93436, pp.~16--31, 1958.

\bibitem{Allen2007}
Allen, L. J. S.. {\em An Introduction to Mathematical Biology}. Pearson/Prentice Hall, 2007.

\bibitem{CabralBalreira2014}
Cabral Balreira and Elaydi, Saber and Lu\'{i}s, Rafael. {\em Local stability implies global stability for the planar Ricker competition model}. Discrete and Continuous Dynamical Systems - Series B, vol.~19, no.~2, pp.~323--351, 2014.

\bibitem{Luis2011}
Lu\'{i}s, Rafael and Elaydi, Saber. {\em Open problems in some competition models}. Journal of Difference Equations and Applications, vol.~17, pp.~1873-1877, 2011.

\bibitem{Ladas2001}
Kulenovic, M. R. S. and Ladas, G.. {\em Dynamics of Second Order Rational Difference Equations: With Open Problems and Conjectures}. CRC Press, 2001.

\bibitem{Lotka1910}
Lotka, A.J. {\em Contribution to the Theory of Periodic Reaction}. J. Phys. Chem., vol.~14, no.~3, pp.~271–-274, 1910.

\bibitem{Lotka}
Lotka, Alfred J. {\em Analytical Note on Certain Rhythmic Relations in Organic Systems}. Proceedings of the National Academy of Sciences, vol.~6, no.~7, pp.~410--415. National Academy of Sciences, 1920.

\bibitem{kocic1993}
Kocic, V. L. and Ladas, G.. {\em Global Behavior of Nonlinear Difference Equations of Higher Order with Applications}. Springer Netherlands, 1993.

\bibitem{doi:10.1080/17513758.2011.581764}
Rafael   Lu\'{i}s and  Saber   Elaydi  and  Henrique   Oliveira. {\em Stability of a Ricker-type competition model and the competitive exclusion principle}. Journal of Biological Dynamics, vol.~5, no.~6, pp.~636-660. Taylor \& Francis, 2011.

\bibitem{Gumus2022}
Ozlem Ak G\"{u}m\"{u}s  and  Qianqian Cui and George Maria Selvam and Abraham Vianny  {\em Global stability and bifurcation analysis of a discrete time SIR epidemic model}. Miskolc. Math., vol.~23, no.~1, pp.~193--210, 2022.

\bibitem{Din2016}
Qamar Din. {\em Qualitative behavior of a discrete sir epidemic model}. Int. J. Biomath., vol.~6, no.~9, pp.~1--15, 2016.

\bibitem{Brauer2010}
Fred Brauer and Zhilan Feng and Carlos Castillo-Chávez. {\em Discrete epidemic models}. Mathematical Biosciences and Engineering, vol.~7, no.~1, pp.~1--15, 2010.

\bibitem{Kulenovic2021}
Kulenovi\'{c}, M. R. S. and Nurkanovi\'{c}, M. and Yakubu, Abdul-Aziz. {\em Asymptotic behavior of a discrete-time density-dependent SI epidemic model with constant recruitment}. J. Appl. Math. Comput., vol.~67, no.~1-2, pp.~733--753, 2021.

\bibitem{Voit2021}
Eberhard O. Voit and Daniel V. Oliven\c{c}a. {\em Discrete Biochemical Systems Theory}. Front. Mol. Biosci., vol.~9, no.~874669, 2022.

\bibitem{Thierry2010}
Huillet, Thierry E. {\em On discrete-time multiallelic evolutionary dynamics driven by               selection}. J. Probab. Stat., no.~2671781, pp.~Art. ID 580762, 27, 2010.

\bibitem{MR1910672}
Yuan, Zhaohui and Huang, Lihong and Chen, Yuming. {\em Convergence and periodicity of solutions for a discrete-time               network model of two neurons}. Math. Comput. Modelling, vol.~35, no.~9-10, pp.~941--950, 2002.

\bibitem{StaceySmith}
Smith?, Robert. {\em Modelling disease ecology with mathematics}., vol.~5, no.~3720721, pp.~295. American Institute of Mathematical Sciences (AIMS),               Springfield, MO, 2017.

\bibitem{Gopalsamy}
K. Gopalsamy, I. Leung. {\em Delay induced periodicity in a neural netlet of excitation and inhibition}. Physica D: Nonlinear Phenomena, vol.~89, no.~3-4, pp.~395--426, 1996.

\bibitem{FARIA2000129}
Teresa Faria. {\em On a Planar System Modelling a Neuron Network with Memory}. Journal of Differential Equations, vol.~168, no.~1, pp.~129--149, 2000.

\bibitem{smith_JDEA_1998}
H. L. Smith. {\em Planar competitive and cooperative difference equations}. Journal of Difference Equations and Applications, vol.~3, pp.~335--357, 1998.

\bibitem{StWo}
Streipert, S. and Wolkowicz, G. S. K.. {\em An augmented phase plane approach for discrete planar maps: Introducing next-iterate  operators}. Mathematical Biosciences, vol.~355, pp.~108924, 2023.

\bibitem{Hilger1988}
Hilger, Stefan. {\em Ein Maßkettenkalk\"{u}l }., 1988.

\bibitem{BohnerWarth}
Bohner, Martin and Warth, Howard. {\em The Beverton--Holt dynamic equation}. Applicable Analysis, vol.~86, no.~8, pp.~1007--1015, 2007.

\bibitem{Bohner1}
Bohner, M. and Peterson, A. {\em Dynamic Equations on Time Scales: An Introduction with Applications}., no.~1843232 (2002c:34002), pp.~x+358. Birkh\, 2001.

\bibitem{Bohner2}
Bohner, M. and Peterson, A.C. {\em Advances in Dynamic Equations on Time Scales}. Birkh\"{a}user, 2011.

\bibitem{Olszewski}
S. Olszewski. {\em Time Scale and its Application in Perturbation Theory}. Zeitschrift für Naturforschung A, vol.~46, no.~4, pp.~313--320, 1991.

\bibitem{KloedenZmorzynska2006}
Peter E. Kloeden and Alexandra Zmorzynska. {\em Lyapunov functions for linear nonautonomous dynamical equations on time scales}. Advances in difference equations, vol.~2006, no.~Article ID 69106, pp.~1 -- 10, 2006.

\bibitem{SIEGMUND2002255}
Stefan Siegmund. {\em A spectral notion for dynamic equations on time scales}. Journal of Computational and Applied Mathematics, vol.~141, no.~1, pp.~255--265, 2002.

\bibitem{martynyuk2016}
Martynyuk, A.A.. {\em Stability Theory for Dynamic Equations on Time Scales}. Springer International Publishing, 2016.

\bibitem{DAVIS20071291}
John M. Davis and Ian A. Gravagne and Billy J. Jackson and Robert J. Marks and Alice A. Ramos. {\em The Laplace transform on time scales revisited}. Journal of Mathematical Analysis and Applications, vol.~332, no.~2, pp.~1291--1307, 2007.

\bibitem{MR2145447}
Bohner, Martin. {\em Some oscillation criteria for first order delay dynamic               equations}. Far East J. Appl. Math., vol.~18, no.~3, pp.~289--304, 2005.

\bibitem{StTS}
Streipert, S.. {\em Introduction to Time Scales}. Nonlinear Systems - Recent Developments and Advances



\bibitem{BookTSSt}
Tucker, Ledyard R. and others. {\em Nonlinear Systems - Recent Developments and Advances}. Contributions to Mathematical Psychology. IntechOpen, 2023.

\bibitem{Lotka1925}
Lotka, A.J.. {\em Elements of Physical Biology}.. Williams \& Wilkins, 1925.

\bibitem{Volterra1926a}
Volterra, V.. {\em Fluctuations in the Abundance of a Species considered Mathematically}. Nature, vol.~118, pp.~558--560, 1926.

\bibitem{Volterra1926b}
Volterra, Vito. {\em Variazioni e fluttuazioni del numero d'individui in specie animali	conviventi}. Mem. Acad. Lincei Roma, vol.~2, pp.~31--113, 1926.

\bibitem{Kingsland2015}
Sharon Kingsland. {\em Alfred J. Lotka and the origins of theoretical population ecology}. Proceedings of the National Academy of Sciences, vol.~112, no.~31, pp.~9493--9495, 2015.

\bibitem{HernandezGarcia2009}
Hern\'{a}ndez-Garc\'{i}a, Emilio and L\'{p}pez, Cristobal and  Pigolotti, Simone  and Andersen, Ken H. {\em Species Competition: Coexistence, Exclusion and Clustering}. Philosophical Transactions: Mathematical, Physical and Engineering Sciences, vol.~367, no.~1901, pp.~3183--3195, 2009.

\bibitem{Gavina2018}
Gavina, M. K. A. and  Tahara, Takeru and  Tainaka, Kei-ichi and  Ito, Hiromu and  Morita, Satoru and  Ichinose, Genki and Okabe, Takuya  and  Togashi, Tatsuya and   Nagatani, Takashi and  Yoshimura, Jin. {\em Multi-species coexistence in Lotka--Volterra competitive systems with crowding effects}. Sci. Rep., vol.~8, pp.~1198, 2018.

\bibitem{Mooij2024}
Mooij, M. N. and Baudena, M. and von der Heydt, A. S. and Kryven, I.. {\em Stable coexistence in indefinitely large systems of competing species}. Proceedings of the Royal Society A: Mathematical, Physical and Engineering Sciences, vol.~480, no.~2299, pp.~20240290, 2024.

\bibitem{Marcia}
M\'{a}rcia Lemos-Silva and Delfim F.M. Torres. {\em Logistic equation on time scales}. Examples and Counterexamples, vol.~8, pp.~100197, 2025.

\bibitem{AkinBohner}
Elvan Akin-Bohner and Martin Bohner. {\em Miscellaneous Dynamic Equations}. Methods Appl. Anal., vol.~10, no.~1, pp.~11--30, 2003.

\end{thebibliography}
\end{document}